\documentclass[12pt,reqno,oneside,a4paper]{amsart}
\usepackage[british]{babel}

\usepackage{graphicx}
\usepackage{amscd}
\usepackage{amsmath}
\usepackage{amsfonts}
\usepackage{amssymb}
\usepackage{comment}
\usepackage{color}
\usepackage[usenames,dvipsnames,table,xcdraw]{xcolor}
\usepackage{pdfsync}
\usepackage{bbm}
\usepackage{dsfont}
\usepackage[colorlinks=true]{hyperref}
\usepackage{enumerate}
\usepackage{booktabs}
\usepackage{float}
\usepackage{wrapfig}
\usepackage{caption}
\usepackage{subcaption}
\usepackage{longtable}
\usepackage{listings}
\usepackage[T1]{fontenc}
\usepackage[scaled]{beramono}
\usepackage{tikz}
\usepackage{pgffor}
\usepackage[all]{xy}

\definecolor{trp}{rgb}{1,1,1}

\definecolor{red}{rgb}{1,0,.2}
\newcommand*{\clrred}[1]{{\color{red} #1}}
\newcommand{\cred}[1]{\clrred{ #1}}

\newtheorem{theorem}{Theorem}[section]
\theoremstyle{plain}

\newtheorem{claim}[theorem]{Claim}

\newtheorem{cor}[theorem]{Corollary}
\newtheorem{corollary}[theorem]{Corollary}

\newtheorem{definition}[theorem]{Definition}
\newtheorem{example}{Example}

\newtheorem{lemma}[theorem]{Lemma}

\newtheorem{prop}[theorem]{Proposition}
\newtheorem{remark}[theorem]{Remark}

\numberwithin{equation}{section}

\newcommand{\R}{\mathbb{R}}
\newcommand{\N}{\mathbb{N}}

\newcommand{\proj}{\mathrm{proj}}

\newcommand{\ii}{\mathbf{i}}
\newcommand{\jj}{\mathbf{j}}
\newcommand{\kk}{\mathbf{k}}
\newcommand{\iih}{\boldsymbol{\hat\imath}}
\newcommand{\jjh}{\boldsymbol{\hat\jmath}}
\newcommand{\kkh}{\mathbf{\hat k}}
\newcommand{\ih}{\hat\imath}
\newcommand{\jh}{\hat\jmath}
\newcommand{\iiv}{\overline{\imath}}
\newcommand{\jjv}{\overline{\jmath}}

\newcommand*{\e}[1]{\text{e}^{#1}}

\newcommand*{\arabicdec}[1]{\the\numexpr\value{#1}-1\relax}

\usepackage{anysize}

\usepackage{caption}

\definecolor{blue}{rgb}{0,0,1}

\definecolor{red}{rgb}{1,0,.7}
\newcommand{\rk}[1]{\cred{ #1}\marginpar{\cred{$\spadesuit$}}}

\begin{document}
\title[Triangular Gatzouras--Lalley-type carpets]{Triangular Gatzouras--Lalley-type planar carpets with overlaps}

\author{Istv\'an Kolossv\'ary}
\address{Istv\'an Kolossv\'ary, Budapest University of Technology and Economics, MTA-BME Stochastics Research Group, P.O. Box 91, 1521 Budapest, Hungary;  \newline MTA Alfr\'ed R\'enyi Institute of Mathematics} \email{istvanko@math.bme.hu}

\author{K\'aroly Simon}
\address{K\'aroly Simon, Budapest University of Technology and Economics, MTA-BME Stochastics Research Group, P.O. Box 91, 1521 Budapest, Hungary} \email{simonk@math.bme.hu}


\begin{abstract}
We construct a family of planar self-affine carpets with overlaps using lower triangular matrices in a way that generalizes the original Gatzouras--Lalley carpets \cite{GatzourasLalley92} defined by diagonal matrices. Assuming the rectangular open set condition, Bara\'nski proved for this construction in \cite{BaranskiTriag_2008} that for typical parameters, which can be explicitly checked, the inequalities between the Hausdorff, box and affinity dimension of the attractor are strict. We generalize this result to overlapping constructions, where we allow complete columns to be shifted along the horizontal axis (as in \cite{fraser_shmerkin_2016} and \cite{pardo-simon}) or allow parallelograms to overlap within a column in a transversal way. Our main result is to show sufficient conditions under which these overlaps do not cause the drop of the dimension of the attractor. Several examples are provided to illustrate the results, including a self-affine smiley, a family of self-affine continuous curves, examples with overlaps and an application of our results to some three-dimensional systems.
\end{abstract}
\date{\today}

\maketitle

\thispagestyle{empty}
\section{Informal Introduction}

Gatzouras--Lalley carpets \cite{GatzourasLalley92} are the attractors of self-affine Iterated Function Systems (IFS) on the plane whose first level cylinders are aligned into columns using orientation preserving maps with linear parts given by diagonal matrices, see left hand side of Figure~\ref{fig:GLandTGL}. In this paper we consider a natural generalization of such carpets by replacing the diagonal matrices with lower triangular ones so that the column structure is preserved, see right hand side of Figure~\ref{fig:GLandTGL} and Definition~\ref{def:triagLalleyG}. 
\begin{figure}[h]
		\centering 	\includegraphics[width=.6\linewidth]{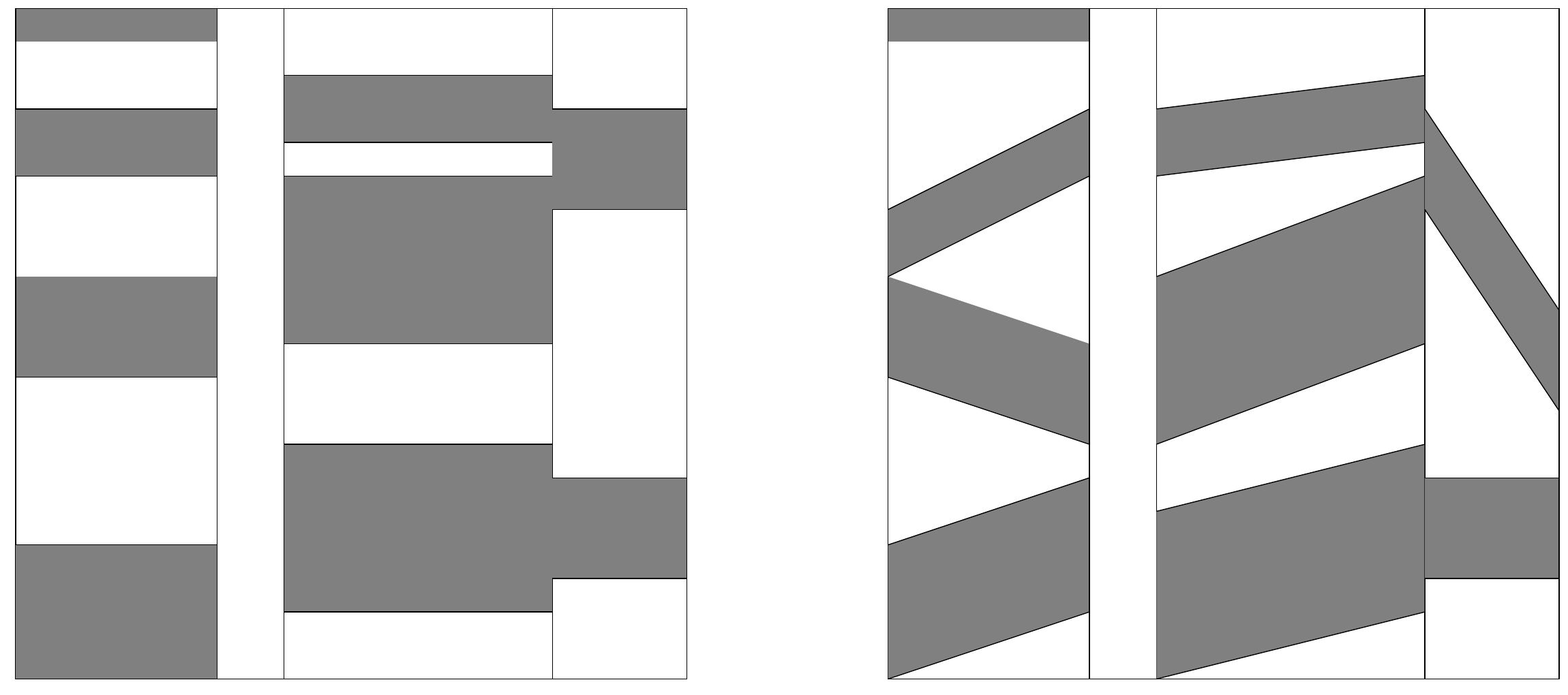}
	\caption{The IFS defining a Gatzouras--Lalley carpet on the left and a triangular Gatzouras--Lalley-type carpet on the right.}
	\label{fig:GLandTGL}
\end{figure}

We call them Triangular Gatzouras--Lalley-type (TGL) planar carpets, indicating that the linear part of the maps defining the IFS are triangular matrices and it is a natural generalization of the Gatzouras--Lalley construction. Such a particular TGL carpet (see the left hand side of Figure~\ref{fig:carpetex2}) appeared in the paper of Falconer and Miao~\cite{FalconerMiao07} and later in the paper of B\'ar\'any~\cite{barany15_LYformula}. For this particular example the box and the Hausdorff dimension of the attractor are the same. However, Bara\'nski~\cite{BaranskiTriag_2008} showed that this not true in general for TGL carpets under the assumption that the interior of the first level cylinders are disjoint (like the image in the right hand side of Figure~\ref{fig:GLandTGL}). In this paper we further generalize this result by allowing different types of overlaps between the cylinders. 

\begin{figure}[H]
	\centering
	\includegraphics[width=.85\linewidth]{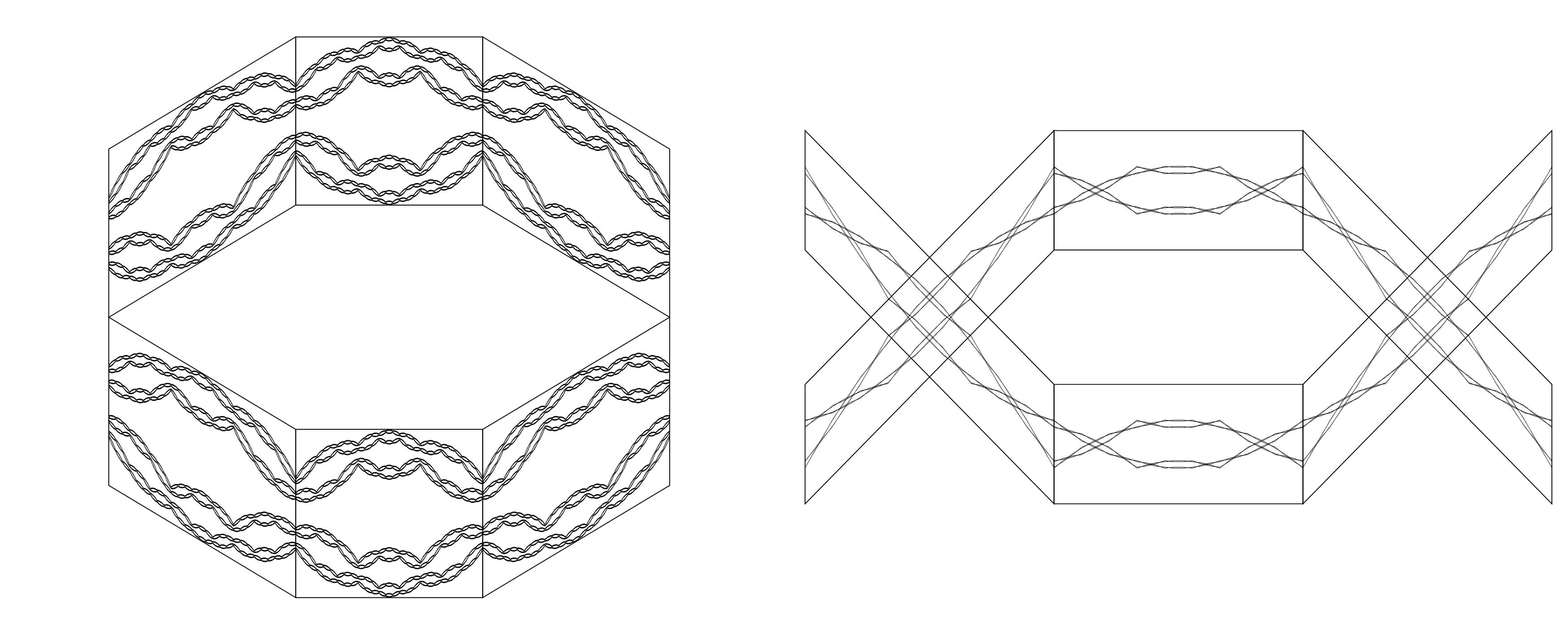}
	\caption{The attractor $\Lambda_a$ from Subsection \ref{subsec:Ex2} with parameter $a=3/10$ (left) and Subsection \ref{subsec:overlappingex} with parameter $a=3/20$ (right), shown together with the outlines of the images of $f_i([0,1]^2)$.}\label{fig:carpetex2}
\end{figure}

We distinguish three different kind of TGL carpets with overlapping cylinders, see Figure~\ref{fig:TGLOverlaps}. Under some conditions we compute the (typically different values of the) box- and Hausdorff dimension for carpets like the first two on Figure~\ref{fig:TGLOverlaps}. If the overlaps are as sophisticated as on the right hand side of Figure~\ref{fig:TGLOverlaps} then we can compute only the Hausdorff dimension. The Hausdorff dimension is not equal to the box dimension in the examples of Figure~\ref{fig:TGLOverlaps}, but they are equal for the example on the right hand side of Figure~\ref{fig:carpetex2}. 

\begin{figure}[H]
	\centering
	\includegraphics[width=.9\linewidth]{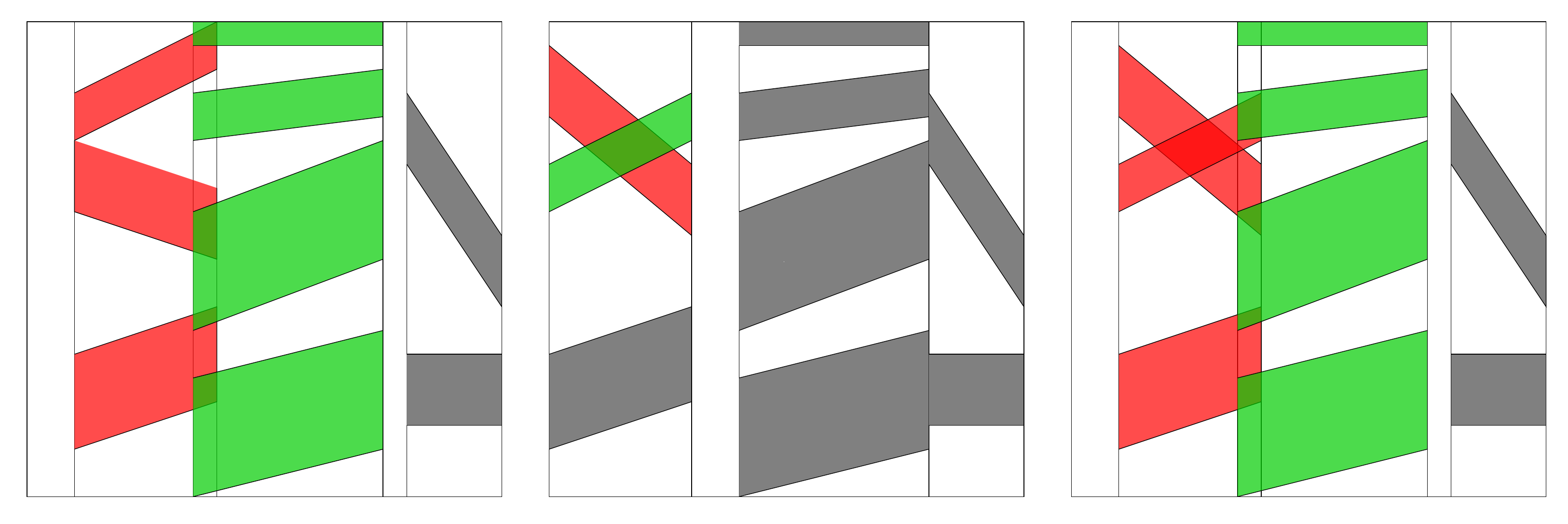}
	\caption{Triangular Gatzouras--Lalley-type carpets with different overlaps. Left: shifted columns satisfying Hochman's exponential separation condition. Center: non-overlapping columns, transversality condition. Right: mixture of both.}
	\label{fig:TGLOverlaps}
\end{figure}

\subsection{Methods used to handle overlaps}

For each type of overlap and dimension we used different methods:
\begin{itemize}
\item The upper bounds on the Hausdorff and box dimensions (after some simple observations)
 follow from proper adaptations of the results of Gatzouras--Lalley~\cite{GatzourasLalley92} and Fraser~\cite{Fraser12Boxlike}, respectively. 
\item To estimate the Hausdorff dimension from below we use the Ledrappier-Young formula of B\'ar\'any and K\"aenm\"aki \cite{BARANYAnti201788} (cited in Theorem \ref{thm:BaranyKaenmaki}) for self-affine measures. We show that this lower bound equals the upper bound
\begin{itemize}
	\item in case of overlapping like on the second figure of Figure \ref{fig:TGLOverlaps}
	by an argument inspired by the transversality method introduced in \cite{barany_rams_simon_triang_2017};
	\item in case of overlapping like on the third figure of Figure \ref{fig:TGLOverlaps},
	we introduce a new separation condition for the self-similar IFS obtained as the projection of the TGL carpet under consideration to the horizontal line. This separation condition is a	non-trivial consequence of Hochman's Exponential Separation Condition \cite{Hochman_Annals14}. We prove this in  Appendix \ref{app:NDD_WAUC} since it could be of separate interest.
\end{itemize}
  \item To estimate the box dimension from below we could not simply use the Hausdorff dimension of the attractor, because, in our case it is (typically) strictly smaller than the box dimension. Therefore,
\begin{itemize}
	\item in case of overlapping like on the first figure of Figure \ref{fig:TGLOverlaps} we used the method of Fraser and Shmerkin \cite{fraser_shmerkin_2016}: the main idea is to pass to a special subsystem of a higher iterate of the IFS which has non-overlapping columns;
	\item in case of overlapping like on the second figure of Figure \ref{fig:TGLOverlaps} we introduced a new argument to count overlapping boxes. It uses transversality and a result of Lalley~\cite{Lalley88} based on renewal theory, which gives the precise asymptotics of the number of boxes needed to cover the projection of the attractor to the horizontal line.
\end{itemize}  
\end{itemize}

\section{Formal Introduction}\label{sec:intro}

A self-affine Iterated Function System (IFS) is a finite list of contracting affine mappings on $\R^d$ of the form $ \mathcal F = \{ f_i(\underline{x}):=   A_i\underline{x} + t_i\}_{i=1}^N$, where the $A_i$ are non-singular $d\times d$ matrices and $t_i\in\R^d$ are translation vectors. It is well-known that there exists a unique non-empty compact subset $\Lambda_{\mathcal{F}}=\Lambda$ of $\R^d$, called the self-affine set or the attractor associated to $\mathcal{F}$, such that
\begin{equation*}
\Lambda = \bigcup_{i=1}^N f_i ( \Lambda).
\end{equation*}
For basic dimension theoretic definitions such as the Hausdorff, packing and (lower and upper) box dimension of a set and the Hausdorff and local dimension of measures we refer to Falconer \cite{FalconerBook}. Throughout, the Hausdorff, packing, lower and upper box dimension will be denoted by $\dim_{\rm H}, \dim_{\rm P}, \underline{\dim}_{\rm B}, \overline{\dim}_{\rm B}$ and $\dim_{\rm B}$, respectively.

A general upper bound for all aforementioned dimensions is given by the affinity dimension $\dim_{\rm{Aff}}$, introduced by Falconer~\cite{falconer_1988}, which comes from the "most natural" cover of the set. All self-affine sets satisfy
\begin{equation*}
\dim_{\rm H} \Lambda \leq \dim_{\rm P} \Lambda \leq \overline{\dim}_{\rm B}\Lambda \leq \min\{\dim_{\rm{Aff}}\Lambda,d\}.
\end{equation*}
In a generic sense, equality of dimensions is typical. Falconer proved in his seminal paper \cite{falconer_1988} that for fixed linear parts $\{A_1,\ldots,A_N\}$ if $\|A_i\|< 1/3$ and the translations are chosen randomly according to $N\times d$ dimensional Lebesgue measure then all the aforementioned dimensions of the self-affine set are equal. The $1/3$ bound was later relaxed by Solomyak~\cite{solomyak_1998} to $1/2$, which is sharp due to an example of Przytycki and Urba\'nski~\cite{PrzytyckiUrbanski1989}. Very recently B\'ar\'any, Hochman and Rapaport~\cite{BaranyHochmanRapaport} greatly improved these results in two dimensions by giving specific, but mild conditions on $\{A_1,\ldots,A_N\}$ under which the dimensions are equal. However, in specific cases, which do not fall under these conditions, strict inequality is possible. Planar carpets form a large class of examples in $\R^2$ for which this exceptional behavior is typical. The highly regular column and/or row structure causes the drop of the Hausdorff dimension. 
We continue with the formal definition of TGL carpets and then present some pictures to informally explain our contribution.

\subsection{Triangular Gatzouras--Lalley-type carpets}\label{subsec:setup} Denote the closed unit square by $R=[0,1]\times[0,1]$. Let $\mathcal A=\{A_1,\ldots,A_N\}$ be a family of $2\times 2$ invertible, strictly contractive, real-valued lower triangular matrices. The corresponding self-affine IFS is the collection of affine maps
\begin{equation}\label{def:IFS_F}
\mathcal F = \{f_i(\underline{x}):=  A_i\underline{x} + t_i\}_{i=1}^N, \text{ where } A_i=\begin{pmatrix}
b_i & 0 \\ d_i & a_i
\end{pmatrix} \text{ and } t_i=\begin{pmatrix} t_{i,1} \\ t_{i,2}
\end{pmatrix},
\end{equation}
for translation vectors $t_i$, with $t_{i,1},t_{i,2}\geq 0$. We assume that $a_i,b_i\in(0,1)$.

Orthogonal projection of $\mathcal{F}$ to the horizontal $x$-axis, denoted $\mathrm{proj}_x$, generates an important self-similar IFS on the line
\begin{equation}\label{def:verticalIFS}
\widetilde{\mathcal{H}} = \{\widetilde{h}_i(x):=  b_i x + t_{i,1}\}_{i=1}^N.
\end{equation}
We denote the attractor of $\mathcal F$ and $\widetilde{\mathcal{H}}$ by $\Lambda=\Lambda_{\mathcal{F}}$ and $\Lambda_{\widetilde{\mathcal{H}}}$ respectively.

\begin{definition}\label{def:triagLalleyG}
	We say that an IFS of the form \eqref{def:IFS_F} is \texttt{triangular} \texttt{Gatzouras}--\texttt{Lalley}-\texttt{type} \texttt{(TGL)}  and we call its attractor $\Lambda$ a triangular Gatzouras--Lalley-type planar carpet (TGL carpet for short)
	if the following conditions hold:
	\begin{enumerate}[(a)]
		\item direction-$x$ dominates, i.e.
		\begin{equation}\label{ass:dirx_dominates}
		0<a_i<b_i<1\; \text{ for all } i\in[N]:=\{1,2,\ldots,N\},
		\end{equation}
		\item column structure: there exists a partition of $[N]$ into $M>1$ sets $\mathcal{I}_1, \dots ,\mathcal{I}_M$ with cardinality $|\mathcal{I}_{\ih}|=N_{\ih}>0$ so that
		\begin{equation}
		\mathcal{I}_1=\{1,\ldots,N_1\} \text{ and }\mathcal{I}_{\ih} = \{N_1+\ldots+N_{\ih-1}+1,\ldots,N_1+\ldots+N_{\ih}\} \label{def:Partition}
		\end{equation}
		for $\ih=2,\ldots,M$. Assume that for two distinct indices $k$ and $\ell \in\{1,\ldots,N\}$
		\begin{equation}\label{ass:column_structure}
		\text{if there exists } \ih\in\!\{1,\ldots,M\} \text{ such that } k,\ell\!\in\!\mathcal I_{\ih}, \text{ then }
		\begin{cases}
		b_k=b_{\ell}=:r_{\ih}, \\ t_{k,1}=t_{\ell,1}=:u_{\ih}.
		\end{cases}
		\end{equation}
		We also introduce
		\begin{equation}\label{a72}
		\mathcal H = \{ h_{\ih}(x):=  r_{\ih} x + u_{\ih}\}_{\ih=1}^M,
		\end{equation}
		and we observe that the attractor $\Lambda_{\mathcal{H}}$ of $\mathcal{H}$
		is identical with $\Lambda_{\widetilde{\mathcal{H}}}$.
		\item we assume that $\sum_{j \in\mathcal{I}_{\ih} }a_j \leq 1$ holds for every $\ih \in\{1,\ldots,M\}$ and the non-overlapping column structure
		\begin{equation}\label{ass:NonOverlapCols}
		u_{\ih}+r_{\ih}\leq u_{\ih+1} \text{ for } \ih=1,\ldots,M-1 \;\text{ and }\; u_{M}+r_{M}\leq 1.
		\end{equation}
		\item Without loss of generality we always assume in this paper that
		\begin{description}
			\item[(A1)] $f_i(R)\subset R$ for all $i\in [N]$ and
			\item[(A2)] The smallest and the largest fixed points of the functions of $\mathcal{H}$ are
			$0$  and $1$ respectively.
		\end{description}
		
	\end{enumerate}
	Observe that the definition allows overlaps within columns (like the second figure in Figure~\ref{fig:TGLOverlaps}), but columns do not overlap.
	
	We say that $\Lambda$ is a \texttt{shifted TGL carpet} if we drop the assumption \eqref{ass:NonOverlapCols}, that is  non-overlapping column structure is NOT assumed, we require only that $\sum_{\ih=1}^M r_{\ih}\leq 1$ (like the first figure in Figure~\ref{fig:TGLOverlaps}).
\end{definition}

We often consider the following special cases:

\begin{definition}\label{def:UniformFibreDiagHomo}
We say that a shifted TGL carpet $\Lambda$ has \texttt{uniform vertical fibres} if
\begin{equation}\label{eq:UniformFibre}
\sum\limits_{j\in \mathcal{I}_{\ih} }a_j^{s-s_{\mathcal{H}}}=1 \text{ for every } \ih \in[M],
\end{equation}
where $s=\dim_{\rm B}\Lambda$ and $s_{\mathcal{H}}=\dim_{\rm B}\Lambda_{\mathcal{H}}$.

Furthermore, we call $\Lambda$ a \texttt{diagonally homogeneous shifted TGL carpet} if 
\begin{equation*}
b_i\equiv b\text{ and } a_i\equiv a \text{ for every } i\in[N].
\end{equation*}
In particular, a diagonally homogeneous carpet has uniform vertical fibres if $N/M\in\N$ and $N_{\ih}=N/M$ for every $\ih\in\{1,\ldots,M\}$.
\end{definition}

The special case when $N_{\ih}=1$ for all $\ih=1,\ldots,M$ is treated in the paper of B\'ar\'any, Rams and Simon~\cite[Lemma~3.1]{barany_rams_simon_triang_2017}.

\subsubsection*{Some notation}
The map $f_i$ is indexed by $i\in[N]$. To indicate which column $i$ belongs to in the partition \eqref{def:Partition} we use the function
\begin{equation}\label{def:phiFunc}
\phi: \{1,2,\ldots,N\}\to\{1,2,\ldots,M\},\;\; \phi(i):=\ih\; \text{ if } i\in\mathcal{I}_{\ih}.
\end{equation}
With this notation we can formulate the column structure \eqref{ass:column_structure} as
\begin{equation}\label{a73}
\text{if } \phi(k)=\phi(\ell)=\ih, \text{ then } b_k=b_{\ell}=:r_{\ih} \text{ and }  t_{k,1}=t_{\ell,1}=:u_{\ih}.
\end{equation}
Throughout, $i$ is an index from $[N]$, while $\ih$ with the hat is an index corresponding to a column from $\{1,\ldots,M\}$. We use analogous notation for infinite sequences $\ii=i_1i_2\ldots$ and $\iih=\ih_1\ih_2\ldots$, see Subsection~\ref{subsec:SymbNotation} for details.

For compositions of maps we use the standard notation $f_{i_1\ldots i_n}:=f_{i_1}\circ f_{i_2}\circ\ldots \circ f_{i_n}$, where $i_j\in~\{1,\ldots,N\}$. Similarly, for products of matrices we write
\begin{equation*}\label{a94}
A_{i_1 \dots i_n}:=A_{i_1}\cdot \ldots \cdot A_{i_n}:= \left(
\begin{array}{cc}
b_{i_1 \ldots i_n} & 0 \\
d_{i_1 \ldots i_n} & a_{i_1 \ldots i_n} \\
\end{array}
\right).
\end{equation*}
Immediate calculations give $b_{i_1 \ldots i_n}=b_{i_1}\cdot \ldots\cdot b_{i_n},\, a_{i_1 \ldots i_n}=a_{i_1}\cdot \ldots\cdot a_{i_n}$ and
\begin{equation}\label{a93}
d_{i_1 \dots i_n}=
\sum\limits_{\ell =1}^{n}
d_{i_{\ell }} \cdot
\prod_{k<\ell }a_{i_k} \cdot \prod_{r=\ell +1}^{n}b_{i_r},
\end{equation}
where by definition  $\prod\limits_{k<1}a_{i_k}:=1$  and $\prod\limits_{r=n+1}^{n}b_{i_r}:=1$. The image $R_{i_1\ldots i_n}:=f_{i_1\ldots i_n}(R)$ is a parallelogram with two vertical sides, see Figure \ref{a89}. We refer to $b_{i_1 \ldots i_n}$ as the \textit{width}, $a_{i_1 \ldots i_n}$ as the \textit{height} and $\gamma_{i_1 \dots i_n}$ as the angle of the longer side of the parallelogram $R_{i_1 \dots  i_n}$, in other words
\begin{equation}\label{a95}
\tan \gamma_{i_1 \dots i_n}:=\frac{d_{i_1 \dots i_n}}{b_{i_1 \dots i_n}}.
\end{equation}
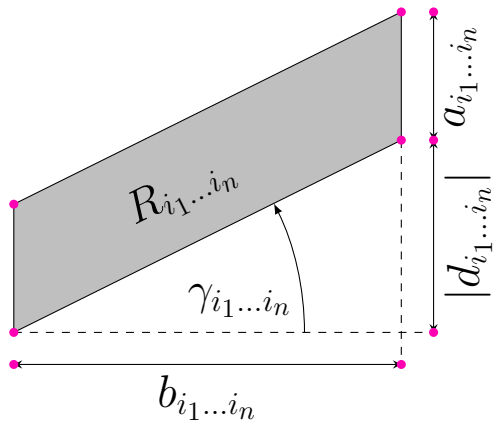
\begin{figure}[H]
	\centering
	\begin{tikzpicture}[scale=8.5, >=stealth]
	\tikzstyle help lines=[color=gray,very thin]
	\tikzstyle en grid=[style=help lines]

	\coordinate (O) at (0,0);
	\coordinate (B) at (0.6,0.3);
	\coordinate (B1) at (0.6,-0.05);
	\coordinate (C) at (0.6,0.5);
	\coordinate (O1) at (0,-0.05);
	\coordinate (D) at (0,0.2);
	\coordinate  (E) at (0.9,-0.3);
	\coordinate  (F) at (0.65,0);
	\coordinate  (G) at (0.65,0.3);
	\coordinate  (H) at (0.65,0.5);
	
	\draw[dashed] (O)--(F);

	
	\draw[dashed] (B)--(B1);
	
	\draw[<->] (O1)--(B1) node[pos=0.5,sloped,below] {\Large{$b_{i_1 \dots i_n}$}};
	\draw[<->] (G)--(H) node[pos=0.5,sloped,below] {\Large{$a_{i_1 \dots i_n}$}};
	\draw[<->] (F)--(G) node[pos=0.5,sloped,below] {\Large{$|d_{i_1 \dots i_n}|$}};

	\draw[fill=lightgray] (O)--(B)--(C)--(D)--(O);
	
	\node[rotate=26.56] at (0.26,0.22) {\Large{$R_{i_1 \dots i_n}$}};
	
	\draw[-latex] (0:0.45cm) arc (0:26.565:0.45cm);	
	
	\node at (0.35,0.05) {\Large{$\gamma_{i_1 \dots i_n}$}};
	
	\foreach \point in {O,B,C,D,O1,B1,F,G,H}
	\fill [red] (\point) circle (0.2pt);
	\end{tikzpicture}
	\caption{The skewness of $R_{i_1\ldots i_n}:=f_{i_1\ldots i_n}([0,1]^2)$}
	\label{a89}
\end{figure}
Since direction-$x$ dominates, $R_{i_1\ldots i_n}$ is extremely long and thin for large $n$. A simple argument gives that $|\tan\gamma_{i_1 \dots i_n}|$ remains uniformly bounded away from $+\infty$.
\begin{lemma}\label{lemma:d/b_bounded}
	There exists a non-negative constant $K_0<\infty$ such that for every $n$ and every finite length word $i_1\ldots i_n$
	\begin{equation*}
	\left|\frac{d_{i_1\ldots i_n}}{b_{i_1\ldots i_n}}\right|\leq K_0.
	\end{equation*}
\end{lemma}
\begin{proof}
	Since direction-$x$ dominates, $\max_i\{a_i/b_i\}<1$, hence using \eqref{a93}
	\begin{equation*}
	\left|\frac{d_{i_1\ldots i_n}}{b_{i_1\ldots i_n}}\right| \leq \frac{|d_{i_1}|}{b_{i_1}}+ \sum_{k=2}^{n} \frac{|d_{i_k}|}{b_{i_k}} \prod_{j=1}^{k-1}\frac{a_{i_j}}{b_{i_j}} \leq  \frac{\max_i\{|d_i|/ b_i\}}{1-\max_i\{a_i/b_i\}}<\infty.
	\end{equation*}
\end{proof}

\subsection{Our contribution explained with pictures}\label{subsec:OurContrib}
A natural way to depict an IFS $\mathcal{F}=\{f_i\}$ is to provide the images $f_i(R)$, where $R$ is
the smallest rectangle which contains $\Lambda$. Without loss of generality, we always assume in this paper that $R=[0,1]^2$. The correspondence between the shifted TGL IFS and a figure showing the collection of images of $R$ is unique.

The shaded rectangles and parallelograms in Figure~\ref{fig:GLandTGL} show the images of $R$ under the orientation preserving affine maps defining a Gatzouras--Lalley (GL) carpet~\cite{GatzourasLalley92} on the left, see Definition~\ref{def:GLCarpets}, and a triangular Gatzouras--Lalley-type (TGL) carpet on the right. These are typical examples which satisfy the Rectangular Open Set Condition (ROSC), see Definition \ref{def:separations}. Furthermore, there is a correspondence between the rectangles and parallelograms so that the height and width of corresponding ones coincide. We call the Gatzouras--Lalley carpet the \emph{GL-brother} of the TGL carpet, see Definition~\ref{def:GLBrother}. Even though the ROSC holds, it is not immediate that the dimension of the two attractors should be the same. The parallelograms can be placed in a way that there is no bi-Lipschitz map between the two attractors. Nevertheless, Bara\'nski essentially shows in \cite{BaranskiTriag_2008} that assuming the ROSC the Hausdorff and box dimension of a TGL carpet is equal to the respective dimension of its GL brother.

The IFS in Figure~\ref{fig:GLandTGL} is an example for which $\dim_{\rm H}\Lambda<\dim_{\rm B}\Lambda<\dim_{\rm{Aff}}\Lambda$. If the orthogonal projection of $\Lambda$ to the $x$-axis is the whole $[0,1]$ interval, then the box- and affinity dimensions are equal, see Corollary~\ref{cor:dimB=dimA}. Figure~\ref{fig:smiley} shows such an example, where the outlines of $f_i(R)$ are shown together with the attractor, which we call the "self-affine smiley".

\begin{figure}[]
	\centering
	\includegraphics[width=.41\linewidth]{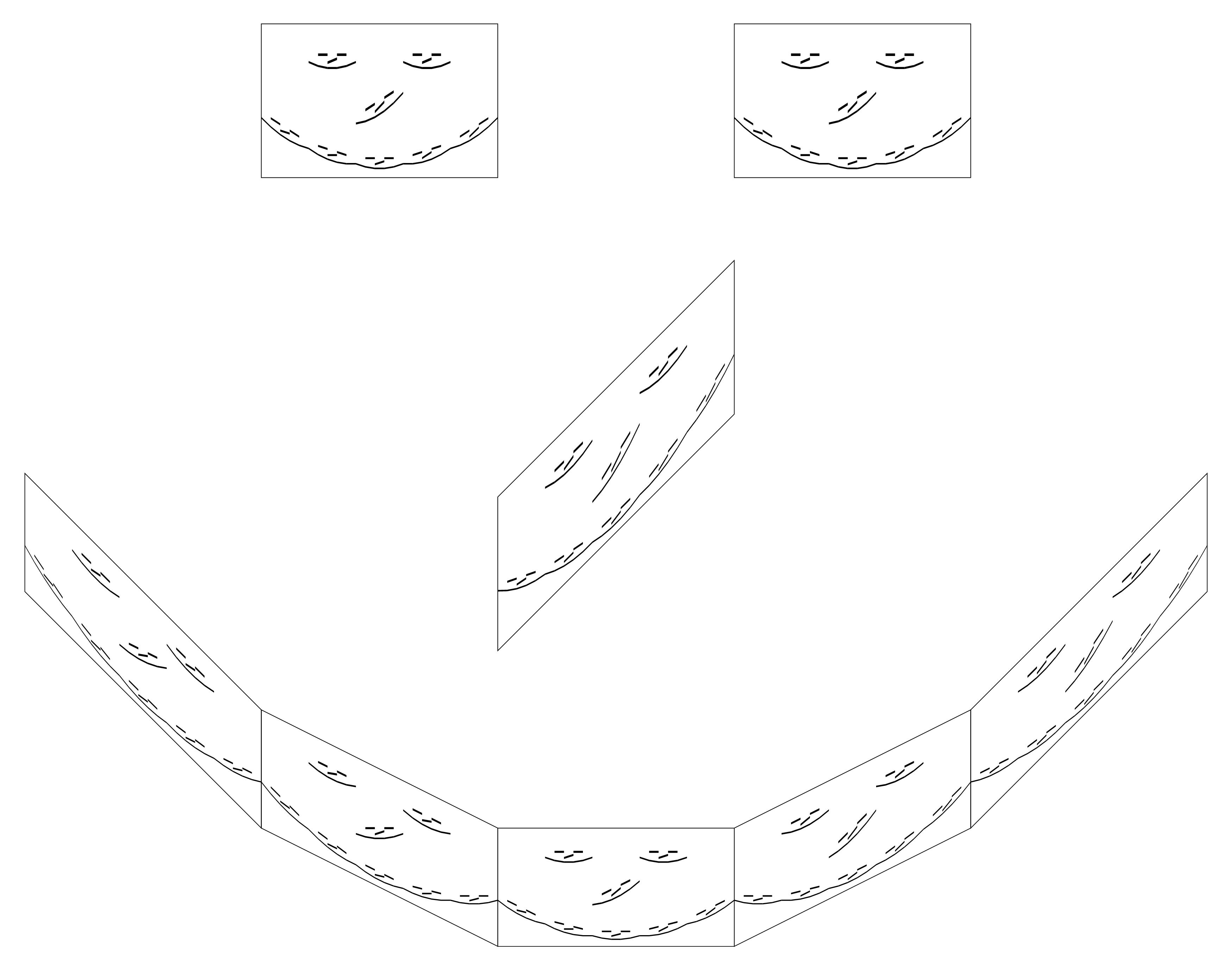}
	\caption{The "self-affine smiley", whose $\dim_{\rm H}=1.20665\ldots<1.21340\ldots=\dim_{\rm B}=\dim_{\rm{Aff}}$, see Subsection~\ref{subsec:smiley}.}
	\label{fig:smiley}
\end{figure}

A special class of examples consists of the diagonally homogeneous carpets, recall Definition~\ref{def:UniformFibreDiagHomo}. The well-known Bedford--McMullen carpets \cite{Bedford84_phd,mcmullen84} form a proper subclass of these TGL carpets. The attractor on the left-hand  side of Figure \ref{fig:carpetex2} first appeared in \cite[Figure 1 (a)]{FalconerMiao07} and then again in \cite[Subsection 4.3]{barany15_LYformula}. It is exceptional in the class of TGL carpets, since $\dim_{\rm H}\Lambda=\dim_{\rm B}\Lambda=\dim_{\rm{Aff}}\Lambda$. This is because it has uniform vertical fibres. In all these examples only the boundary of the cylinder sets $f_i(R)$ could intersect.

The main contribution of the present paper is that different types of overlaps are allowed in our construction, recall Figure~\ref{fig:TGLOverlaps}. On the left, the columns are shifted in a way that the IFS $\mathcal{H}$ on the $x$-axis generated by the columns satisfies Hochman's Exponential Separation Condition, see Definition~\ref{def:34}. This type of shifted columns was considered by Fraser and Shmerkin~\cite{fraser_shmerkin_2016} and Pardo-Sim\'on~\cite{pardo-simon} on different carpets. In the center, columns do not overlap, however, parallelograms within a column may do so if a certain transversality like condition holds, see Definition~\ref{def:separations}. The one on the right on Figure~\ref{fig:TGLOverlaps} has both types of overlaps.

By modifying the translation vectors in the example on the left-hand  side of Figure~\ref{fig:carpetex2}, we get a brother with overlaps seen on the right-hand side, for which we show in Subsection \ref{subsec:overlappingex} that transversality holds. Another concrete overlapping example satisfying transversality is "$X\equiv X$" in Figure \ref{fig:XequalX}, for which there is strict inequality between the Hausdorff, box and affinity dimensions. If instead, the construction would be "$X=X$", then the Hausdorff and box dimensions would be equal. Moreover, if there were no empty columns in this example, then the box and affinity dimensions would coincide.

\begin{figure}[]
	\centering
	\includegraphics[width=.6\linewidth]{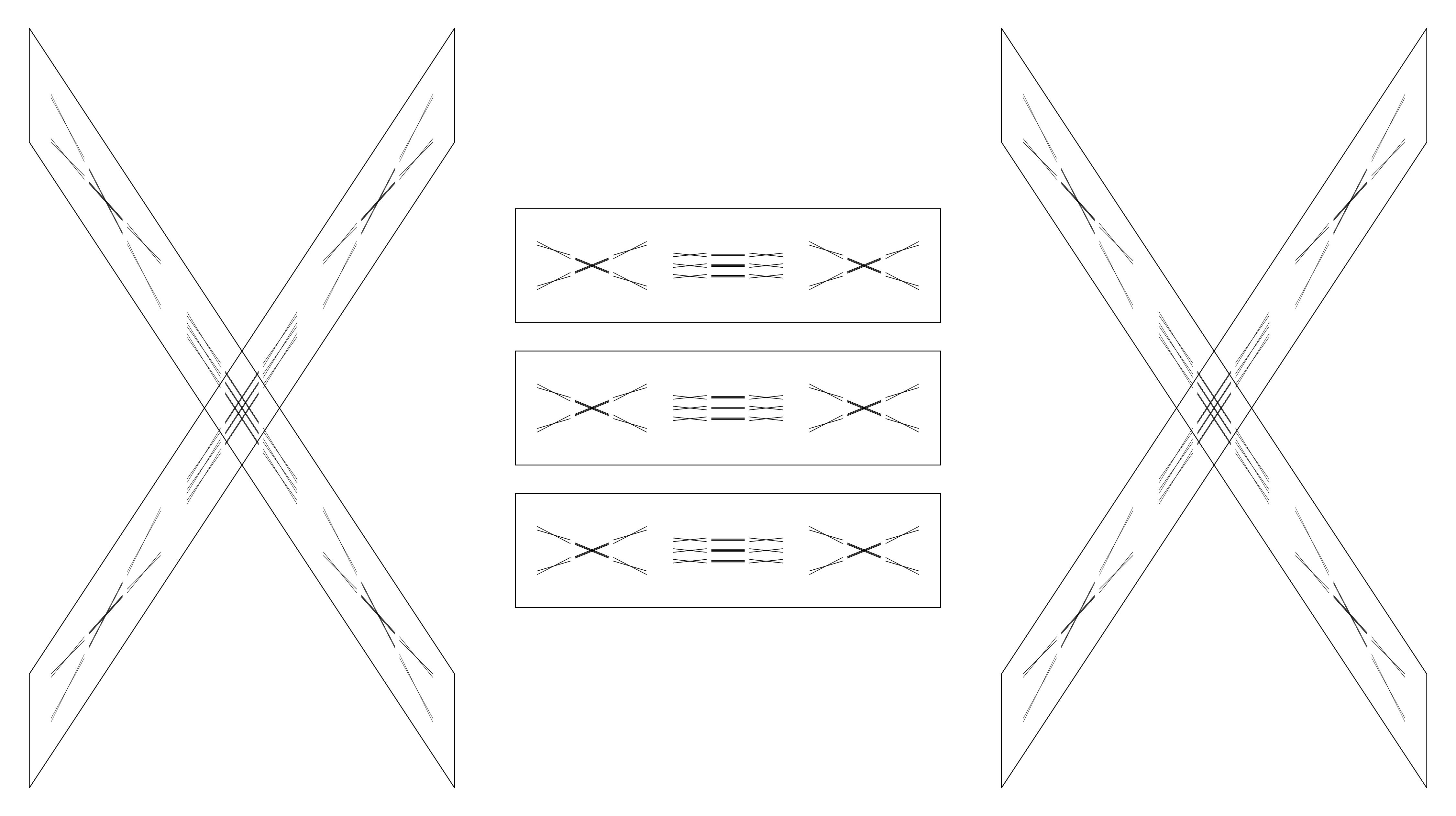}
	\caption{Example "$X\equiv X$" from Subsection~\ref{subsec:ExXequivX} for which $\dim_{\rm H}\Lambda=1.13259\ldots<\dim_{\rm B}\Lambda=1.13626\ldots<\dim_{\rm{Aff}}\Lambda=1.2170\ldots\,$.}
	\label{fig:XequalX}
\end{figure}

Section~\ref{sec:MainRes} contains the formal statements of all our main results. Roughly speaking, we show that for any TGL carpet $\Lambda$
\begin{equation*}
\dim \Lambda \leq \dim \widetilde{\Lambda},
\end{equation*}
where $\widetilde{\Lambda}$ is the GL brother of $\Lambda$, see Definition~\ref{def:GLBrother}, and $\dim$ means either box or Hausdorff dimension, see Theorems~\ref{thm:mainresUpperbound} and \ref{thm:mainBox}. When ROSC holds and $\mathcal{H}$ satisfies Hochman's condition, then equality can be deduced from recent works \cite{BARANYAnti201788,Fraser12Boxlike}. Our main contribution is that in the presence of overlaps described above, we give sufficient conditions under which $\dim \Lambda$ does not drop below $\dim\widetilde{\Lambda}$, see Theorems~\ref{thm:maindimresult} and \ref{thm:BoxwOverlaps}. In particular, for the Hausdorff dimension we allow both types of overlaps simultaneously, however for the box dimension we can prove our results only if
at most one of the two types of overlaps occurs. 

For a discussion on generalizing towards orientation reversing maps, see Subsection~\ref{subsec:NegEntries}. In particular, we calculate the dimension of a family of self-affine continuous curves $\Lambda_a$, which is generated by an IFS $\mathcal{F}_a$ containing a map that reflects on the $y$-axis, see Figure~\ref{fig:zipperex}. The formal treatment of all these examples is done in Section \ref{sec:ex}.

\begin{figure}[h]
	\centering
	\includegraphics[width=.75\linewidth]{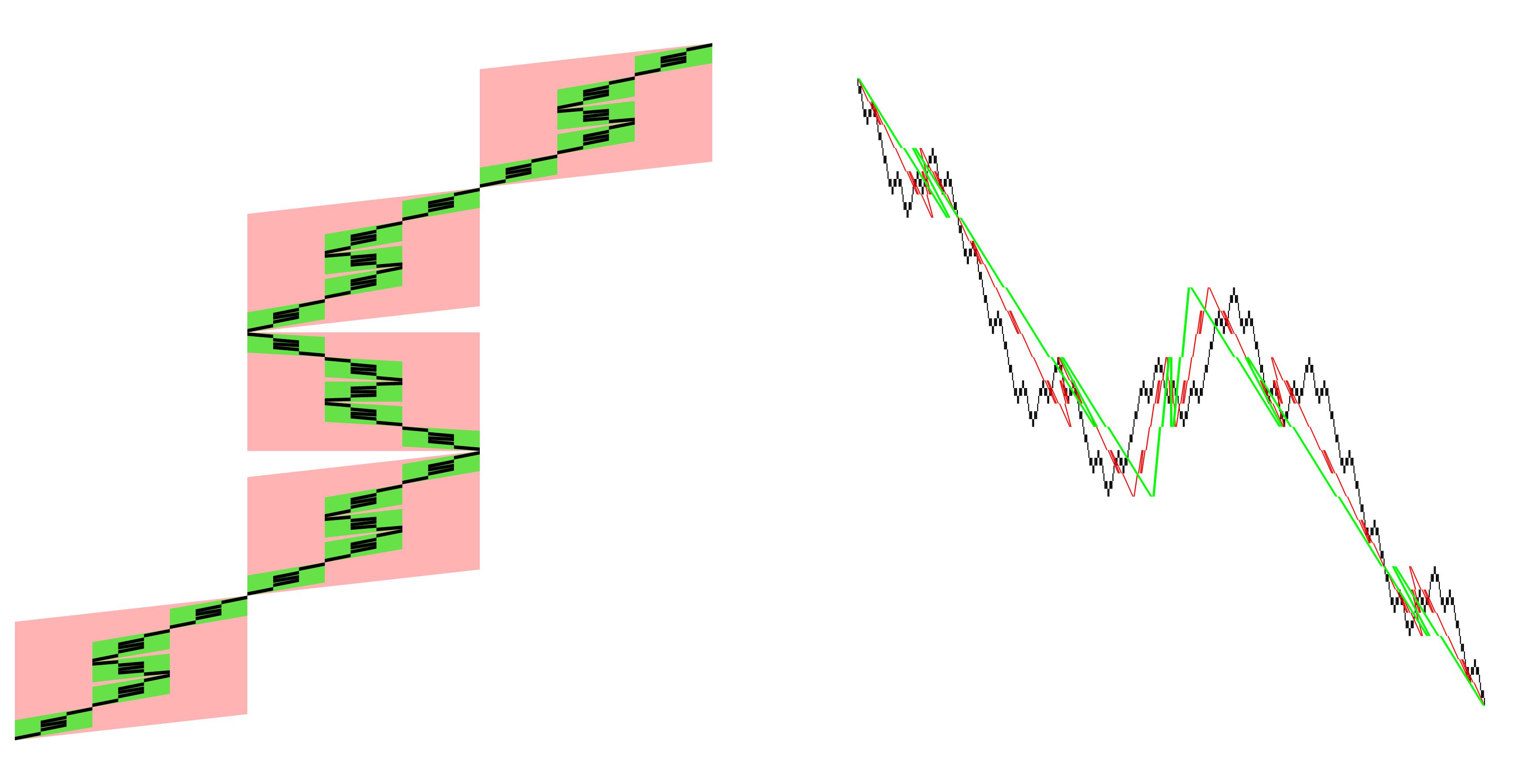}
	\caption{Left: first ({\color{red}red}), second ({\color{green}green}) and third (black) level cylinders of $\mathcal{F}_a$. Right: $\Lambda_a$ rotated 90 degrees from Subsection \ref{subsec:exZipper} with parameter $a=0.2$ (black), $0.12$ ({\color{red}red}) and $0.08$ ({\color{green}green}).}\label{fig:zipperex}
\end{figure}

One motivation to study self-affine fractals of overlapping construction is that sometimes the dimension of a higher dimensional fractal of non-overlapping construction coincides with its lower dimensional orthogonal projection which can be a self-affine fractal of overlapping construction. We obtain such a set in 3D by starting from a TGL carpet with overlaps on the $xy$-plane and then "lift" it to 3D so that the interiors of the first level cylinders are disjoint. Figure~\ref{fig:3Dcarpet} shows such an example with the first level cylinders (left), the attractor (center) and the projection of the cylinders and attractor to the $xy$-plane (right). Section~\ref{sec:ThreeD} contains the formal treatment of this type of construction.
\begin{figure}[h]
	\centering
	\includegraphics[width=.97\linewidth]{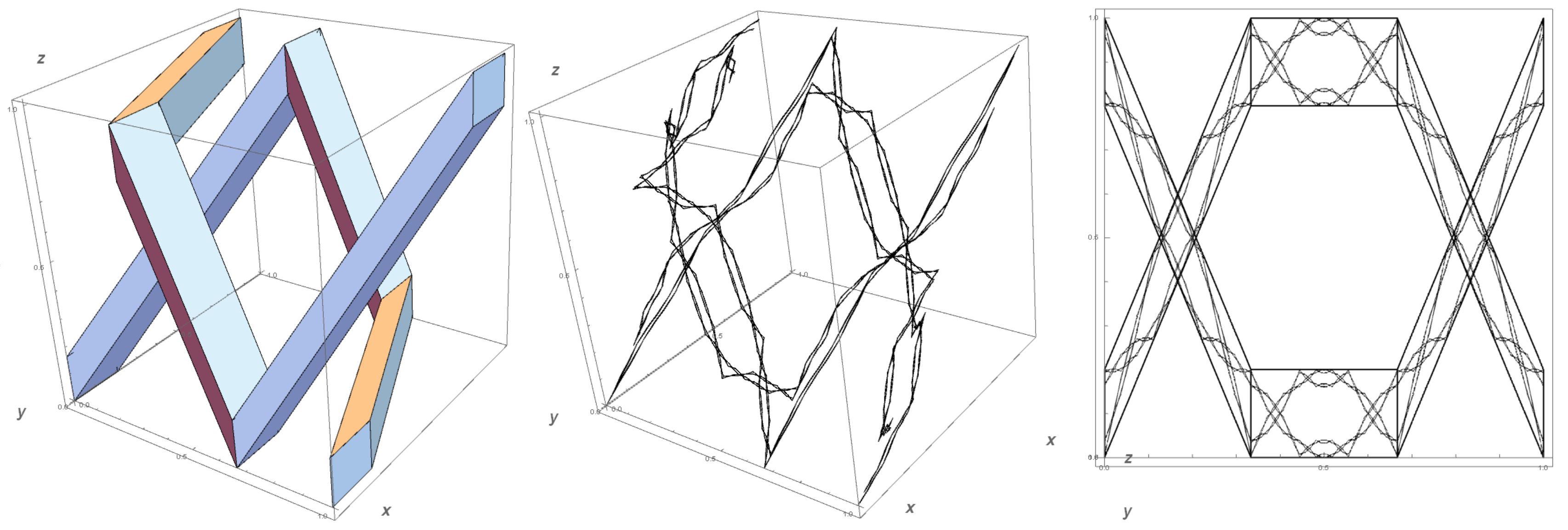}
	\caption{A three-dimensional fractal whose dimension is equal to the dimension of its orthogonal projection to the $xy$-plane.}\label{fig:3Dcarpet}
\end{figure}

\subsection{Brief overview of planar carpets}\label{subsec:carpets}
Independently of each other, Bedford~\cite{Bedford84_phd} and McMullen~\cite{ mcmullen84} were the first to study planar carpets. They split the unit square $R$ into $m$ columns of equal width and $n$ rows of equal height for some integers $n>m\geq 2$ and considered iterated function systems (IFS) of the form
\begin{equation*}
f_{(i,j)}(\underline{x}):= \begin{pmatrix} 1/m & 0 \\ 0 & 1/n \end{pmatrix} \begin{pmatrix} x \\ y \end{pmatrix} + \begin{pmatrix} i/m \\ j/n
\end{pmatrix}
\end{equation*}
for $(i,j)\in A\subseteq \{0,\ldots,m-1\}\times\{0,\ldots,n-1\}$. They gave explicit formulas for the Hausdorff and box-counting dimension of the corresponding attractor $\Lambda$.
It turns out that $\dim_{\rm H}\Lambda=\dim_{\rm B}\Lambda$ is atypical, namely, equality holds if and only if $\Lambda$ has uniform vertical fibres. 

Later Gatzouras and Lalley~\cite{GatzourasLalley92} generalized the results to the following class of IFSs.
\begin{definition}\label{def:GLCarpets}
	A self-affine IFS $\widetilde{ \mathcal F}$ is a Gatzouras--Lalley (GL) IFS and its attractor $\widetilde{\Lambda}$ is a \texttt{GL carpet} if $\widetilde{ \mathcal F}$ is a TGL IFS as in Definition~\ref{def:triagLalleyG} with the additional assumptions that all off-diagonal elements $d_i=0$ and the rectangular open set condition (ROSC) holds, i.e.
	\begin{equation}\label{ass:ROSC}
	\tilde f_i((0,1)^2)\cap \tilde f_j((0,1)^2) = \emptyset\; \text{ for all } i\neq j.
	\end{equation}
\end{definition}

\begin{definition}\label{def:GLBrother}
	Let $\Lambda$ be a shifted TGL carpet generated from the IFS $\mathcal{F}$ of the form \eqref{def:IFS_F}. We say that the Gatzouras--Lalley IFS
	\begin{equation*}
	\widetilde{ \mathcal F} = \{\tilde f_i(\underline{x}):=  \tilde A_i\underline{x} +\tilde{t}_i\}_{i=1}^N, \text{ where } \tilde A_i=\begin{pmatrix}
	\tilde b_i & 0 \\ 0 & \tilde a_i
	\end{pmatrix} \text{ and } \tilde{t}_i=\begin{pmatrix} \tilde{t}_{i,1} \\ \tilde{t}_{i,2}
	\end{pmatrix},
	\end{equation*}
	and its attractor $\widetilde{\Lambda}$ is the \texttt{GL brother} of $\mathcal{F}$ and $\Lambda$, respectively, if $\tilde a_i=a_i$ and $\tilde b_i=b_i$ for every $i\in[N]$, furthermore, $\widetilde{ \mathcal F}$ has the same column structure \eqref{a73} as $\mathcal{F}$. If the shifted TGL carpet $\Lambda$ is actually a TGL carpet (that is $\Lambda$ has non-overlapping column structure) then we also require that $\widetilde{t}_{i,1}=t_{i,1}$ holds for all $i\in[N]$.
	
	There always exists such a brother since we assume Definition \ref{def:triagLalleyG} (c) and $\sum_{\ih=1}^M r_{\ih}\leq 1$. Throughout, the GL brother of $\Lambda$ will always be denoted with the extra tilde $\widetilde{\Lambda}$.
\end{definition}

A standard technique to give a lower bound for the Hausdorff dimension of the attractor $\widetilde{ \Lambda} = \bigcup_{i\in [N]} \tilde f_{i} ( \widetilde\Lambda)$ is to study self-affine measures $\nu_{\mathbf p}$, i.e. compactly supported measures with support $\widetilde{ \Lambda}$ satisfying
\begin{equation*}
\nu_{\mathbf p}= \sum_{i=1}^{N} p_i \nu_{\mathbf p}\circ \tilde f_i^{-1},
\end{equation*}
for some probability vector $\mathbf{p}=(p_1,\ldots,p_N)$. Let $\mathcal{P}$ be the set of all probability distributions on the set $[N]$ and $\mathcal{P}_0$ be the subset when all $p_i>0$. By definition
\begin{equation*}
\sup_{\mathbf{p}\in\mathcal{P}} \dim_{\rm H} \nu_{\mathbf p} \leq \dim_{\rm H} \widetilde{ \Lambda}.
\end{equation*}
Gatzouras and Lalley proved that there always exists a $\mathbf{p}^\ast$ for which the supremum is attained, furthermore $\mathbf{p}^\ast\in\mathcal{P}_0$. Let
\begin{equation*}
\alpha^\ast:= \dim_{\rm H} \nu_{\mathbf p^\ast}=\sup_{\mathbf{p}\in\mathcal{P}} \dim_{\rm H} \nu_{\mathbf p}.
\end{equation*}
They explicitly calculated
\begin{equation}\label{eq:dimHBernoulli}
\dim_{\rm H} \nu_{\mathbf p} = \frac{\sum_{i=1}^N p_i \log p_i}{\sum_{i=1}^N p_i \log a_i} + \left( 1- \frac{\sum_{i=1}^N p_i \log b_i}{\sum_{i=1}^N p_i \log a_i}\right) \frac{\sum_{\ih=1}^M q_{\ih} \log q_{\ih}}{\sum_{i=1}^N p_i \log b_i},
\end{equation}
where $q_{\ih}=\sum_{j\in\mathcal{I}_{\ih}} p_j$. This formula is a special case of the Ledrappier--Young formula, see Subsection~\ref{subsec:LYformula} for details and references. For Bedford--McMullen carpets the optimal $\mathbf p^\ast$ can be given by routine use of the Lagrange multipliers method. The main result of \cite{GatzourasLalley92} is that for a GL carpet the $\alpha^\ast$ bound is sharp, i.e.
\begin{equation*}
\alpha^\ast = \dim_{\rm H} \widetilde{ \Lambda}.
\end{equation*}
In \cite{GatzourasLalley92} Gatzouras and Lalley also gave an implicit formula to calculate the box dimension of their carpet. Let $s_x$ be the unique real such that $\sum_{\ih=1}^{M}r_{\ih}^{s_x}=1$ ($r_{\ih}$ was defined in \eqref{ass:column_structure}). Then $\dim_{\rm B}\widetilde{ \Lambda}=s$ is the unique real such that
\begin{equation*}
\sum\limits_{i=1}^{N} b_{i}^{s_x}a_{i}^{s-s_x}=1.
\end{equation*}
Again, equality of $\dim_{\rm H} \widetilde{ \Lambda}$ and $\dim_{\rm B} \widetilde{ \Lambda}$ is highly atypical. It holds if and only if the $\alpha^\ast$-dimensional Hausdorff measure of $\widetilde{ \Lambda}$, denoted $\mathcal{H}^{\alpha^\ast}(\widetilde\Lambda)$, is positive and finite, which is equivalent to $\widetilde\Lambda$ having uniform vertical fibres, recall \eqref{eq:UniformFibre}.
For Bedford--McMullen carpets Peres showed in \cite{peres_94infinitemeasure} that $\mathcal{H}^{\alpha^\ast}(\widetilde\Lambda)=\infty$ when $\dim_{\rm H} \widetilde{ \Lambda}< \dim_{\rm B} \widetilde{ \Lambda}$.

More recently, Bara\'nski~\cite{BARANSKIcarpet_2007} kept the row and column structure but relaxed \eqref{ass:dirx_dominates} by allowing an arbitrary subdivision of the horizontal and vertical axis. After appropriately choosing which direction is "dominant", the results resemble that of \cite{GatzourasLalley92}. Continuing this work, Bara\'nski showed in \cite{BaranskiTriag_2008} how the result in \cite{BARANSKIcarpet_2007} can be adapted to obtain the dimension result for TGL carpets assuming ROSC~\eqref{ass:ROSC}. Diagonal systems assuming only ROSC and no further restrictions on the translations were studied by Feng--Wang~\cite{FengWang2005} and Fraser~\cite{Fraser12Boxlike}. Former determined the $L^q$ spectrum of self-affine measures $\nu_{\mathbf p}$ and in particular the box dimension of the attractor. In \cite{Fraser12Boxlike} linear isometries which map $[-1,1]^2$ to itself are allowed and the box dimension is determined. Fraser called these box-like sets. Observe that in all the mentioned papers the ROSC was assumed.

Carpets with overlaps were not studied until the last few years. Fraser and Shmerkin~\cite{fraser_shmerkin_2016} shift the columns of  Bedford--McMullen carpets to get overlaps, while Pardo-Sim\'on~\cite{pardo-simon} allows shifts in both directions of Bara\'nski carpets. Relying on a recent breakthrough by Hochman~\cite{Hochman_Annals14} on the dimension of self-similar measures on the line, both papers show that apart from a small exceptional set of parameters the results in \cite{Bedford84_phd,mcmullen84} and \cite{BARANSKIcarpet_2007} remain valid in the overlapping case. This is the type of shifted columns that can be seen in Figure~\ref{fig:TGLOverlaps}.

We finish the section by formalizing the separation conditions between cylinders sets.

\subsection{Separation conditions}\label{subsec:separationconds}
In our main results, we assume different extents of separation for the parallelograms $f_i(R)$, recall Figures~\ref{fig:GLandTGL} and \ref{fig:TGLOverlaps}. This will be considered in Subsection~\ref{a21}. In Subsection~\ref{a20} we consider separation conditions for $\mathcal{H}$ which are actually conditions about the extent of separation of the column structure.

\subsubsection{Separation of the cylinder parallelograms}\label{a21}
\begin{definition}[Separation conditions for a shifted TGL IFS $\mathcal{F}$]\label{def:separations}
	We say that
	\begin{itemize}
		\item $\mathcal{F}$ satisfies the \texttt{rectangular open set condition (ROSC)} if the strong open set condition (SOSC) holds for $\mathcal{F}$ with $U=(0,1)^2$, i.e. $\Lambda\cap U\neq\emptyset$ with
		\begin{equation*}
		\bigcup_{i=1}^N f_i(U)\subseteq U \,\text{ and }\, f_i(U)\cap f_j(U)=\emptyset \text{ for every } i\neq j.
		\end{equation*}
		\item \texttt{each column independently satisfies the ROSC} if for every $\ih\in[M]$ and $k,\ell\in\mathcal{I}_{\ih}$ we have $f_k(U)\cap f_{\ell}(U)=\emptyset$. In other words, if the interior of two first level cylinders intersects, then they are from different columns.
		\item
		$\mathcal{F}$ satisfies the\texttt{ transversality condition}
		if there exists a $K_1>0$ such that for every $n$ and words $(i_1\ldots i_n), (j_1\ldots j_n)\in\left\{1, \dots ,N\right\}^n$ with  $\phi(i_k)=\phi(j_k)$ for $k=1,\ldots, n$ and $i_1\neq j_1$  ($\phi$ was defined in \eqref{def:phiFunc}), we have
		\begin{equation}\label{a33}
		|\mathrm{proj}_x(
		\mathrm{int} (R_{i_1\ldots i_n} )\cap \mathrm{int} (R_{j_1\ldots j_n} ))| < K_1\cdot \max\{ a_{i_1}\cdot\ldots\cdot a_{i_n},\, a_{j_1}\cdot\ldots\cdot a_{j_n}\}.
		\end{equation}
	\end{itemize}
\end{definition}


Given two finite words $i_1\ldots i_n$ and $j_1\ldots j_n,\, i_1\neq j_1$, the angle of the two corresponding parallelograms $R_{i_1\ldots i_n}$ and $R_{j_1\ldots j_n}$ can be defined as the angle between their non-vertical sides. The transversality condition ensures that any such pair of parallelograms in the \textit{same column} have either disjoint interior or have an angle uniformly separated from zero.

Observe that this definition of transversality coincides in the diagonally homogeneous case with the one in \cite{barany_rams_simon_triang_2017}. In
\cite[Section 1.5]{barany_rams_simon_triang_2017} a sufficient condition for the transversality condition was given. Namely, the authors introduced a self-affine IFS $\widehat{\mathcal{S}}$ in $\mathbb{R}^3$ which is (in our setup)
\begin{equation*}\label{a98}
\widehat{\mathcal{S}}:=\left\{
\widehat{S}_i(\underline{x},z):=\left(f_i(\underline{x}),T_i(z)\right)
\right\}_{i=1}^N, \quad (\underline{x},z)\in [0,1]^2\times \mathbb{R},
\end{equation*}
where $\left\{f_i\right\}_{i=1}^{N}$ was defined in \eqref{def:IFS_F} and $T_i:\mathbb{R}\to\mathbb{R}$
is given by
\begin{equation*}
\mathcal{T}:=\left\{T_i(z):=\frac{a_i}{b_i} \cdot z+\frac{b_i}{d_i}\right\}_{i=1}^{N}.
\end{equation*}
\begin{figure}[h]
	\centering
	\includegraphics[width=.27\linewidth]{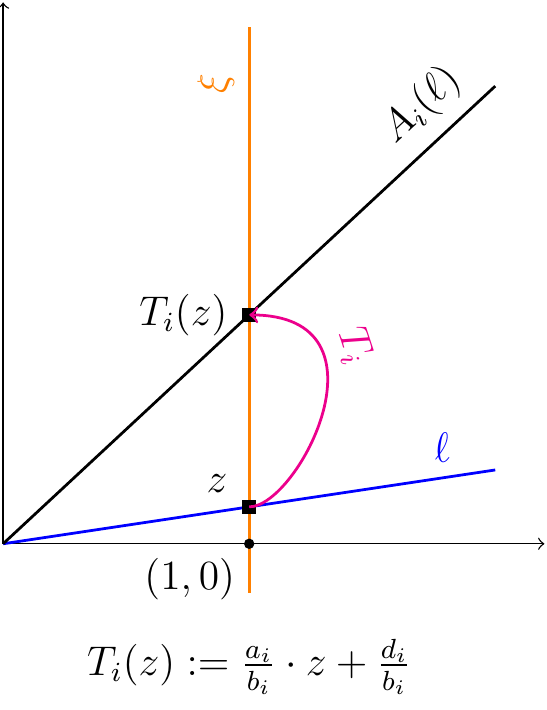}
	\caption{The IFS $\mathcal{T}$, where $z$ and $(1,z)$ are identified }\label{a88}
\end{figure}
The relevance of the IFS $\mathcal{T}$ is that
\begin{equation}\label{a91}
\tan \gamma_{i_1 \dots i_n}=T_{i_1 \dots i_n}(0).
\end{equation}
Indeed, from the definition \eqref{a95} of $\tan \gamma_{i_1 \dots i_n}$ and formula \eqref{a93} it immediately follows that
\begin{equation*}
\tan \gamma_{i_1 \dots i_n} =
\frac{d_{i_1}}{b_{i_1}}+\sum\limits_{\ell =2}^{n}\frac{d_{i_{\ell }}}{b_{i_{\ell }}} \cdot \prod\limits_{k=1 }^{\ell -1}\frac{a_{i_k}}{b_{i_k}}
= T_{i_1 \dots i_n}(0).
\end{equation*}
Using the same argument as in the proof of \cite[Lemma 1.2]{barany_rams_simon_triang_2017}
we obtain that
\begin{lemma}\label{a97}
	If $\widehat{\mathcal{S}}$ satisfies the Strong Separation Property (that is $\widehat{S}_i(\widehat{\Lambda})\cap\widehat{S}_j(\widehat{\Lambda})=\emptyset $ if $i\ne j$ and $\widehat{\Lambda}$ is the attractor of the IFS $\widehat{\mathcal{S}}$) then the transversality condition holds.
\end{lemma}

The next lemma gives a different, easy-to-check sufficient condition for transversality.

\begin{lemma}\label{a87}
	Let $\mathcal{P}_{\jh}:=\left\{(k,\ell ): k,\ell \in \mathcal{I}_{\jh},\quad k\ne\ell ,\quad R_k^{\circ}\cap R_{\ell }^{\circ}\ne \emptyset \right\}$, where $A^{\circ}$ denotes the interior of a set $A$. Moreover, we introduce
	\begin{equation*}
	s_k:=\frac{d_k}{b_k},\;\;  r_k:=\frac{a_k}{b_k}, \;\;
	{r}^*:=\max\limits_{1 \leq k \leq N}r_k, \;\;  b_{\min}:=\min\limits_{1 \leq k \leq N}b_k \,\text{ and }\, s_*:=\min\limits_{1 \leq \jh \leq M \atop \mathcal{P}_{\jh}\ne \emptyset }
	\min\limits_{(k,\ell) \in\mathcal{P}_{\jh}}
	|s_k-s_\ell |.
	\end{equation*}
	Assume that
	\begin{equation*}\label{a85}
	s_* >
	2\frac{1}{b_{\min}} \cdot \frac{r^*}{1-r^*}
	\mbox{ or equivalently }
	\frac{s_*b_{\min}}{2+s_*b_{\min}}>r^*.
	\end{equation*}
	Then the transversality condition holds.
	
	In particular, in the diagonally homogeneous case transversality holds if
	\begin{equation}\label{a79}
	\frac{a}{b}<\frac{d_*}{2+d_*},
	\end{equation}
	where $d_*:=\min\limits_{1 \leq \jh \leq M \atop \mathcal{P}_{\jh}\ne \emptyset }
	\min\limits_{(k,\ell) \in\mathcal{P}_{\jh} }
	|d_k-d_\ell |$.
\end{lemma}

\begin{proof}
	
	Using that $R_k \subset [0,1]^2$ we obtain that $|d_k| < 1$. Hence
	\begin{equation}\label{a82}
	|s_k| \leq \frac{1}{b_{\min}}.
	\end{equation}
	For an $m\in \left\{1, \dots ,M\right\}$
	let $\Sigma_m:=\left\{\mathbf{j}\in\Sigma, j_1\in\mathcal{I}_m\right\}$.  The transversality condition holds if there exists $c>0$ such that for every $n$, for all $m\leq M$ with $\mathcal{P}_m\ne \emptyset$ and 
	\begin{equation}\label{a84}
	\mbox{for all } \mathbf{i},\mathbf{j}\in\Sigma \mbox{ with } (i_1,j_1)\in\mathcal{P}_m,
	\mbox{ we have: }
	\left|  \gamma_{\mathbf{i}|n}-\gamma_{\mathbf{j}|n}\right|>c.
	\end{equation}
	It follows from  \eqref{a91} that \eqref{a84} holds whenever for all such pair of $\mathbf{i}, \mathbf{j}$ and for all $n$
	\begin{equation*}\label{a83}
	|s_{i_1}-s_{j_1}|-\sum\limits_{\ell =2}^{n}
	\left(
	s_{i_\ell } \cdot \prod\limits_{k=1}^{\ell -1}r_{i_k}
	-
	s_{j_\ell } \cdot \prod\limits_{k=1}^{\ell -1}r_{j_k}
	\right)
	\end{equation*}
	is greater than the same positive constant uniformly. However by \eqref{a82} this holds if
	\begin{equation*}\label{a81}
	s_* >
	2\frac{1}{b_{\min}} \cdot \frac{r^*}{1-r^*}.
	\end{equation*}
\end{proof}




\subsubsection{Separation of the columns}\label{a20}

We will also need some separation conditions for the column structure which are represented by separation properties of $\mathcal{H}$, recall \eqref{a72}.

The symbolic spaces for $\mathcal{F}$ and $\mathcal{H}$ are
\begin{equation*}
\Sigma:=\left\{1, \dots ,N\right\}^{\mathbb{N}} \;\text{ and }\; \Sigma_{\mathcal{H}}:=\left\{1, \dots ,M\right\}^{\mathbb{N}}.
\end{equation*}
The natural projections form $\Sigma\to \Lambda$ and
$\Sigma_{\mathcal{H}}\to \Lambda_{\mathcal{H}}$  are $\Pi$ and $\Pi_{\mathcal{H}}$ respectively, see Subsection~\ref{subsec:SymbNotation} for details. Whenever we are given a probability vector  $\mathbf{p}$ on $\left\{1, \dots ,N\right\},$  we always associate to it  another probability vector $\mathbf{q}$ on $\left\{1, \dots ,M\right\}$  such that
\begin{equation}\label{def:q_l}
q_{\ih} :=\sum\limits_{j\in \mathcal{I}_{\ih} } p_j.
\end{equation}
Slightly abusing the notation we write $\mathcal{P}_0$ for both the set of the probability vectors of positive components on $\left\{1, \dots ,N\right\}$ and $\left\{1, \dots ,M\right\}$. The Bernoulli measure $\mathbf{p}^{\N}$ on $\Sigma$ is denoted $\mu_{\mathbf p}$ and its push forward is $\nu_{\mathbf p}=\Pi_\ast\mu_{\mathbf p}=\mu_{\mathbf p}\circ\Pi^{-1}$. Analogously for $\mu_{\mathbf q}$ and $\nu_{\mathbf q}$.

\begin{definition}[Separation conditions for $\mathcal{H}$]\label{def:34}

We say that $\mathcal{H}$ satisfies
\begin{itemize}
\item \texttt{Hochman's Exponential Separation Condition (HESC)}
(see \cite[p. 775]{Hochman_Annals14})
if there exist an $\varepsilon>0$ and  $n_k\uparrow\infty $ such that
for
\begin{equation*}\label{a37}
\Delta_n:=\min\limits_{\iiv,\jjv\in\left\{1 \dots M\right\}^n \atop \iiv\ne\jjv}
\left\{
\begin{array}{ll}
|h_{\iiv}(0)-h_{\jjv}(0)|   , & \hbox{if $h'_{\iiv}(0)=h'_{\jjv}(0)$;} \\
\infty , & \hbox{otherwise.}
\end{array}
\right.
\end{equation*}
we have $\Delta_{n_k}>\e{-\varepsilon \cdot {n_k}}$. Here $h'$ denotes the derivative of the function $h$.
\item \texttt{Weak Almost Unique Coding (WAUC)} if for all Bernoulli measures $\mu_{\mathbf q}$ there exists $\mathcal{B}_\mathcal{H}\subset\Sigma_{\mathcal{H}}$ (may depend on $\mathbf{q}$) for which
\begin{equation*}\label{b98}
\mu_{\mathbf q}(\mathcal{B}_\mathcal{H})=0 \text{ and for every } \iih\in \Sigma_{\mathcal{H}}\setminus\mathcal{B}_\mathcal{H}:\
\#(\Pi_{\mathcal{H}}^{-1}\Pi_{\mathcal{H}}(\iih)\setminus\mathcal{B}_\mathcal{H})=1.
\end{equation*}
\texttt{Almost Unique Coding (AUC)} holds if for every Bernoulli measure $\mu_{\mathbf q}$ and for $\mu_{\mathbf q}$-a.e. $\iih\in \Sigma_{\mathcal{H}}:\, \#\Pi_{\mathcal{H}}^{-1}\Pi_{\mathcal{H}}(\iih)=1$.
\item \texttt{No Dimension Drop (NDD)} if for all push forward measures $\nu_{\mathbf q}=(\Pi_{\mathcal{H}})_\ast\mu_{\mathbf q}$
\begin{equation*}\label{a31}
\dim_{\rm H} \nu_{\mathbf q}=
\frac{-\sum_{\ih=1}^M q_{\ih} \log q_{\ih}}{-\sum_{\ih=1}^M q_{\ih} \log r_{\ih}}.
\end{equation*}
\end{itemize}

\end{definition}

The following implications hold between these conditions
\begin{equation}\label{eq:CondForH}
\mathrm{HESC} \Longrightarrow \mathrm{NDD} \Longleftrightarrow WAUC.
\end{equation}
$\mathrm{HESC} \Longrightarrow \mathrm{NDD}$ follows from Hochman's work \cite[Theorem~1.1]{Hochman_Annals14}. AUC implies NDD from Feng--Hu \cite[Theorem 2.8 and Corollary 4.16]{FengHu09}, but we do not know if the reverse direction $\mathrm{NDD} \Longrightarrow \mathrm{AUC}$ holds or not. Feng informed us \cite{FengOral} that he can prove the equivalence $\mathrm{NDD} \Longleftrightarrow \mathrm{WAUC}$ for ergodic measures. Unfortunately, this result has not yet been written down at the time of preparation of this paper. However, we use it only for Bernoulli measures. For completeness, we give our own complete proof of $\mathrm{NDD} \Longleftrightarrow \mathrm{WAUC}$ for Bernoulli measures in Appendix~\ref{app:NDD_WAUC}.

The set $\mathcal{U}$ of translations $(u_{1},\ldots,u_M)$ defining $\mathcal{H}$ for which HESC does not hold is small. It is stated in \cite[Proposition~2.7]{pardo-simon} that it essentially follows from \cite[Theorem~1.10]{HochmanIndDim} that the Hausdorff and packing dimension of $\mathcal{U}$ is $M-1$, in particular $\mathcal{U}$ has 0 $M$-dimensional Lebesgue measure. Moreover, \cite[Theorem~1.5]{Hochman_Annals14} states that if the parameters $(r_1,\ldots,r_M,u_1,\ldots,u_M)$ defining $\mathcal{H}$ are all algebraic, then HESC does not hold if and only if there is an exact overlap, i.e. $\Delta_n=0$ for some $n$.

\section{Main results}\label{sec:MainRes}

We now state our main results for the Hausdorff dimension of shifted TGL carpets in Subsection~\ref{subsec:ResHausdorff}, the box dimension in Subsection~\ref{subsec:ResBox} and discuss diagonally homogeneous carpets in Subsection~\ref{subsec:ResDiagHomo}. For a discussion on generalizing towards negative entries in the main diagonal, see Subsection~\ref{subsec:NegEntries}.

\subsection{Hausdorff dimension}\label{subsec:ResHausdorff}
For any vector $\mathbf{c}=(c_1,\ldots,c_K)$ with strictly positive entries and a probability vector $(p_1,\ldots,p_K)$ we write
\begin{equation*}
\langle\mathbf{c}\rangle_{\mathbf{p}}:=  \prod_{i=1}^K c_i^{p_i}.
\end{equation*}
When no confusion is made, we suppress $\mathbf{p}$ and write $\langle\mathbf{c}\rangle=\langle\mathbf{c}\rangle_{\mathbf{p}}$. Throughout, we use this notation for the vectors $\mathbf{a}=(a_1,\ldots,a_N),\, \mathbf{b}=(b_1,\ldots,b_N),\, \mathbf{p}=(p_1,\ldots,p_N)$, $\mathbf{N}=(N_1,\ldots,N_M)$ and $\mathbf{q}=(q_1,\ldots,q_M)$, where $\mathbf{q}$ is derived from $\mathbf{p}$ via \eqref{def:q_l}. Using this notation let us denote the function on the right-hand side of \eqref{eq:dimHBernoulli} by
\begin{equation}\label{def:D(p)}
D(\mathbf{p}):= \frac{\log \langle\mathbf{p}\rangle_{\mathbf{p}}}{\log \langle\mathbf{a}\rangle_{\mathbf{p}}} + \left( 1- \frac{ \log \langle\mathbf{b}\rangle_{\mathbf{p}}}{ \log \langle\mathbf{a}\rangle_{\mathbf{p}}}\right) \frac{\log \langle\mathbf{q}\rangle_{\mathbf{q}}}{ \log \langle\mathbf{b}\rangle_{\mathbf{p}}} = \frac{\log \langle\mathbf{q}\rangle}{ \log \langle\mathbf{b}\rangle} + \frac{\log \langle\mathbf{p}\rangle - \log \langle\mathbf{q}\rangle}{\log \langle\mathbf{a}\rangle}.
\end{equation}
\begin{theorem}[Upper bound]\label{thm:mainresUpperbound}
	Regardless of overlaps, for any shifted triangular Gatzouras--Lalley-type planar carpet $\Lambda$
	\begin{equation*}
	\dim_{\rm H} \Lambda \leq \sup_{\mathbf{p}\in\mathcal{P}} D(\mathbf{p}) =:\alpha^\ast.
	\end{equation*}
	Furthermore, there always exists a $\mathbf{p}^\ast\in\mathcal{P}_0$ for which $ D(\mathbf{p}^\ast) = \alpha^\ast$.
\end{theorem}
The proof is given in Section \ref{sec:upperbound}. Throughout, let $\mathbf{q}^\ast$ denote the vector $q_{\ih}^\ast=\sum_{j\in\mathcal{I}_{\ih}} p^\ast_j$. The next theorem states sufficient conditions under which the Hausdorff dimension of a self-affine measure $\nu_{\mathbf p}$ on $\Lambda$ is equal to $D(\mathbf{p})$.


\begin{theorem}\label{thm:maindimresult}\
	Let $\mathbf{p}\in\mathcal{P}_0$, $\mu_{\mathbf{p}}:=\mathbf{p}^{\mathbb{N}}$ and
	$\nu_{\mathbf{p}}:=\Pi_*\mu_{\mathbf{p}}$. For a shifted triangular Gatzouras--Lalley-type planar carpet $\Lambda$ we have
	\begin{equation*}\label{a57}
	\dim_{\rm H} \nu_{\mathbf{p}}=D(\mathbf{p})
	\end{equation*}
	if the horizontal IFS $\mathcal{H}$ satisfies Hochman's Exponential Separation Condition (in particular, always holds for non-overlapping columns) and
	\begin{enumerate}[(i)]
		\item either each column independently satisfies the ROSC or
		\item $\Lambda$ satisfies transversality (see Definition \ref{def:separations}) and the following inequality holds:
		\begin{equation}\label{cond:main}
		\dfrac{\log \langle\mathbf{a}\rangle_{\mathbf{p}} }{\log \langle\mathbf{b}\rangle_{\mathbf{p}} }>1+\dfrac{\log \langle\mathbf{N}\rangle_{\mathbf{q}} }{-\log \langle\mathbf{q}\rangle_{\mathbf{q}}}.
		\end{equation}
	\end{enumerate}
\end{theorem}
We remark that Proposition \ref{prop:diaghomoCondition} provides
a simple way to check condition \eqref{cond:main} in the diagonally homogeneous case.
Section \ref{sec:dimH_lowerbound} is devoted to the proof of this theorem. As an immediate corollary, we get
\begin{corollary}[Sufficient conditions]\label{cor:dimHSufCond}
	Whenever a shifted TGL carpet $\Lambda$ satisfies the conditions of Theorem~\ref{thm:maindimresult} with replacing $\mathbf{p}$ and $\mathbf{q}$ in \eqref{cond:main} by $\mathbf{p}^\ast$ and $\mathbf{q}^\ast$, then
	\begin{equation*}\label{a54}
	\dim_{\rm H} \Lambda=\alpha^\ast.
	\end{equation*}
\end{corollary}				

\subsection{Box dimension}\label{subsec:ResBox}
Recall the IFSs $\mathcal{\widetilde H}$ \eqref{def:verticalIFS} and $\mathcal{H}$ \eqref{a72} obtained by projecting $\mathcal{F}$ to the $x$-axis. Recall $s_x$ was defined so that $\sum_{\ih=1}^{M}r_{\ih}^{s_x}=1$ and let $\tilde{s}_x$ be the unique real such that $\sum_{i=1}^{N}b_i^{\tilde{s}_x}=1$. Furthermore, introduce
\begin{equation*}
s_{\mathcal{H}}:= \dim_{\rm B} \Lambda_{\widetilde{\mathcal{H}}} = \dim_{\rm B} \Lambda_{\mathcal{H}}.
\end{equation*}
Since $\Lambda_{\mathcal{H}}$ is a self-similar set, $s_{\mathcal{H}}$ is well defined. If $\Lambda$ is a TGL carpet then $s_{\mathcal{H}}=s_x$, otherwise $s_{\mathcal{H}}\leq s_x$. The affinity dimension $\dim_{\mathrm{Aff}}$ of $\Lambda$ can be deduced from the result of Falconer--Miao~\cite[Corollary 2.6]{FalconerMiao07} together with the description in \cite[Subsection~1.3]{barany_rams_simon_triang_2017} and the fact that direction-$x$ dominates: $\dim_{\rm{Aff}}\Lambda=s_A$ is the unique real such that
\begin{equation}\label{def:dimA}
\sum\limits_{i=1}^{N} b_{i}^{\min\{\tilde{s}_x,1\}}a_{i}^{s_A-\min\{\tilde{s}_x,1\}}=1.
\end{equation}
In particular, if $\tilde{s}_x<1$ then $s_A=\tilde{s}_x$, otherwise $s_A$ solves $\sum_{i=1}^{N} b_{i}a_{i}^{s_A-1}=1$. So $s_A$ only depends on the main diagonals $(b_i,a_i)$, but not on the off-diagonal elements $d_i$. So,
the affinity dimension of a shifted TGL carpet $\Lambda$ is
and its GL brother coincide.

The following theorem gives an upper bound for $\dim_{\rm B}\Lambda$, which can be strictly smaller than $s_A$. It was proved for diagonal iterated function systems by Feng--Wang in \cite[Corollary~1]{FengWang2005} and also follows from Fraser's work \cite[Theorem~2.4, Corollary~2.7]{Fraser12Boxlike}. Here we extend its scope to triangular IFSs. In a different context, Hu~\cite{Hu98BoxDim} studied a related problem, where a version of Bowen's formula determines the box dimension.

\begin{theorem}[Upper bound]\label{thm:mainBox}
Regardless of overlaps, for any shifted triangular Gatzouras--Lalley-type planar carpet $\Lambda$
\begin{equation*}
\dim_{\rm P} \Lambda = \overline{\dim}_{\rm B} \Lambda \leq s\leq s_A,
\end{equation*}
where $s$ is the unique solution of the equation
\begin{equation}\label{a74}
\sum_{i=1}^{N} b_{i}^{s_{\mathcal{H}}}a_{i}^{s-s_{\mathcal{H}}}=1.
\end{equation}
In particular, if $\Lambda$ satisfies the ROSC, then $\dim_{\rm P}\Lambda=\dim_{\rm B} \Lambda = s$.
\end{theorem}

\begin{corollary}[Equality of box- and affinity dimension]\label{cor:dimB=dimA}
For any shifted TGL carpet $\Lambda$
\begin{equation*}
s=s_A \quad\Longleftrightarrow\quad s_{\mathcal{H}}=\min\{\tilde{s}_x,1\}.
\end{equation*}
\end{corollary}
\begin{proof}
Follows immediately from comparing equations \eqref{def:dimA} and \eqref{a74} defining $s_A$ and $s$, respectively, together with the fact that $a_i<1$ and $b_i/a_i>1$ for every $i=1.\ldots,N$.
\end{proof}

\begin{remark} \
\begin{enumerate}[a)]
\item The proof of Fraser \cite{Fraser12Boxlike} does not make use of any column structure \eqref{a73}. Hence, Theorem \ref{thm:mainBox} immediately extends to an IFS $\mathcal{F}$ of the form \eqref{def:IFS_F} as long as direction-$x$ dominates ($0<a_i<b_i<1$) and the ROSC holds.
\item Since $\Lambda$ is compact and every open set intersecting $\Lambda$ contains a bi-Lipschitz image of $\Lambda$, we get that $\dim_{\rm P} \Lambda = \overline{\dim}_{\rm B} \Lambda$, see \cite[Corollary 3.9]{FalconerBook}.
\end{enumerate}
\end{remark}

Handling overlaps to calculate the box dimension is a greater challenge, since typically $\dim_{\rm H}\Lambda<\dim_{\rm B}\Lambda$ and thus the usual technique of giving a lower bound by bounding the Hausdorff dimension from below does not suffice. Hence, a new counting argument was necessary.

\begin{theorem}[Box dimension with overlaps]\label{thm:BoxwOverlaps}
For a shifted TGL carpet $\Lambda$ we have $\underline{\dim}_{\rm B} \Lambda\geq s$, hence $\dim_{\rm B}\Lambda=s$, if either of the following hold:
\begin{enumerate}[(i)]
\item $\mathcal{H}$ satisfies HESC and each column independently satisfies the ROSC or
\item $\Lambda$ is a TGL carpet, satisfies transversality and the following inequality:
\begin{equation}\label{cond:BoxDimCond}
-\log\langle \widetilde{\mathbf{p}} \rangle_{\widetilde{\mathbf p}}+ \log\langle \widetilde{\mathbf{q}} \rangle_{\widetilde{\mathbf q}} < s_{\mathcal{H}}(\log\langle \mathbf{b} \rangle_{\widetilde{\mathbf p}}-\log\langle \mathbf{a} \rangle_{\widetilde{\mathbf p}}),
\end{equation}
where $\widetilde{\mathbf p}:=(\widetilde{p}_1,\ldots,\widetilde{p}_N)$ and $\widetilde{\mathbf q}:=(\widetilde{q}_1,\ldots,\widetilde{q}_M)$ are defined by equation \eqref{a74}:
\begin{equation}\label{def:pandqForBox}
\widetilde{p}_i=b_i^{s_{\mathcal{H}}}a_i^{s-s_{\mathcal{H}}}  \;\text{ and }\;
\widetilde{q}_{\ih}=\sum_{j\in\mathcal{I}_{\ih}} b_j^{s_{\mathcal{H}}}a_j^{s-s_{\mathcal{H}}}.
\end{equation}
\end{enumerate}

\end{theorem}

The analogue of the following sufficient and necessary condition for the equality of the box- and Hausdorff dimensions was proved in \cite[Theorem 4.6]{GatzourasLalley92}.
\begin{theorem}[Equality of box- and Hausdorff dimension]\label{b62}
	Assume the shifted TGL carpet $\Lambda$ satisfies ROSC and $\mathcal{H}$ satisfies No Dimension Drop. Then the following three conditions are equivalent,
	\begin{equation}\label{eq:dimB=Hiff}
	\dim_{\rm H}\Lambda=\dim_{\rm B}\Lambda \;\Longleftrightarrow\; s_{\mathcal{H}}=\dim_{\mathrm H}\nu_{\widetilde{\mathbf q}} \;\Longleftrightarrow\; \sum\limits_{j\in \mathcal{I}_{\ih} }a_j^{s-s_{\mathcal{H}}}=1 \text{ for every } \ih \in[M].
	\end{equation}
\end{theorem}
All results for box dimension are proved in Section~\ref{sec:boxdim}.

\subsection{Diagonally homogeneous carpets}\label{subsec:ResDiagHomo}

We show how the conditions and formulas of our main results simplify in the diagonally homogeneous case. Recall the easy-to-check sufficient condition \eqref{a79} for transversality in Lemma~\ref{a87}. Moreover, observe that the vector $\widetilde{\mathbf{p}}$ becomes the uniform vector $\widetilde{p}_i=1/N$ and thus $\widetilde{q}_{\ih}=N_{\ih}/N$. A routine use of the Lagrange multipliers method gives the optimal $\mathbf{p}^\ast$
\begin{equation}\label{def:optpDiagHomo}
p_k^\ast= N_{\ih}^{\frac{\log b}{\log a}-1} \cdot \Big(\sum_{\jh=1}^M N_{\jh}^{\frac{\log b}{\log a}}\Big)^{-1} \text{ if } k\in \mathcal{I}_{\ih}.
\end{equation}
Thus, conditions \eqref{cond:main} and \eqref{cond:BoxDimCond} become
\begin{equation}\label{cond:DiagHomoCase}
\frac{\log \langle\mathbf{p}^\ast\rangle_{\mathbf{p}^\ast}}{\log \langle\mathbf{q}^\ast\rangle_{\mathbf{q}^\ast}} < \frac{\log a}{\log b} \;\;\text{ and }\;\; \frac{\log N}{\log M} +1 + \frac{\log\langle\widetilde{\mathbf{q}}\rangle_{\widetilde{\mathbf{q}}}}{\log M}< \frac{\log a}{\log b},
\end{equation}
respectively. If in addition, the system has uniform vertical fibres, then $\widetilde{p}_i=p_i^\ast=1/N$ also $\widetilde{q}_{\ih}=q_{\ih}^\ast=1/M$. Hence, both conditions \eqref{cond:main} and \eqref{cond:BoxDimCond} become
\begin{equation}\label{a0095}
\frac{\log N}{\log M} < \frac{\log a}{\log b}.
\end{equation}


Next, we give an equivalent explicit formulation of condition \eqref{cond:main}. Let $\varphi(y):=y\log y$ and for $x\in(0,1)$ define
\begin{equation*}
R(x):= x+ \left(r(x) -1\right)^{-1}, \text{ where } r(x)= \dfrac{\varphi\big( \sum_{\ih=1}^M N_{\ih}^x\big)}{\sum_{\jh=1}^M \varphi(N_{\jh}^x)}\,.
\end{equation*}

\begin{lemma}\label{lemma:PropofR(x)}
	$R(x)$ is a continuous, strictly monotone increasing function.
\end{lemma}
\begin{proof}
	Continuity is obvious. It is enough to show that $r(x)$ is strictly monotone decreasing. 
	Let $r'$ denote the derivative. Then
	\begin{align*}
	\underbrace{x\cdot \Big( \sum_{\jh=1}^M \varphi(N_{\jh}^x) \Big)^2 }_{>0}\cdot r'(x) &= \overbrace{\Big( \sum_{\jh=1}^M \varphi(N_{\jh}^x) \Big)^2}^{=:A} + \overbrace{\log \Big( \sum_{\ih=1}^M N_{\ih}^x\Big)\cdot \Big( \sum_{\jh=1}^M \varphi(N_{\jh}^x) \Big)^2 }^{=:B}  \\
	&- \underbrace{ \varphi\Big( \sum_{\ih=1}^M N_{\ih}^x\Big) \cdot \sum_{\jh=1}^M \varphi(N_{\jh}^x) }_{=:C} - \underbrace{ \varphi\Big( \sum_{\ih=1}^M N_{\ih}^x\Big) \sum_{\jh=1}^M \varphi(N_{\jh}^x)\cdot \log N_{\jh}^x }_{=:D}.
	\end{align*}
	We claim that $C> A$ and $D\geq B$, which will conclude the proof of the lemma. For brevity, write $y_{\ih}:=N_{\ih}^x$. $y_{\ih}=1 \Leftrightarrow N_{\ih}=1$, otherwise $y_{\ih}>1$.
	
	To show that $C> A$, it is enough to prove that for $1\leq u\leq v$
	\begin{equation}\label{eq:varphiadditive}
	\varphi(u) + \varphi(v) < \varphi(u+v).
	\end{equation}
	Then a simple induction implies that $\sum \varphi(y_{\ih})< \varphi(\sum y_{\ih})$. Recall $\varphi(1)=0$. The mean value theorem implies that
	\begin{align*}
	\varphi(u+v)-\varphi(v) &= u\cdot \varphi'(\xi),  &\text{for some }& \xi\in(v,u+v)\\
	\varphi(u)-\varphi(1) &= (u-1)\cdot \varphi'(\zeta), &\text{for some }& \zeta\in(1,u).
	\end{align*}
	Since the derivative $\varphi'(y)=1+\log y$ is strictly increasing and $\zeta<\xi$, we have $\varphi'(\zeta)<\varphi'(\xi)$. This implies \eqref{eq:varphiadditive}.
	To prove the other inequality
	\begin{equation*}
	D-B= \Big(\sum_{\ih=1}^M y_{\ih}\Big)\cdot\log \Big(\sum_{\ih=1}^M y_{\ih}\Big)  \sum_{\jh=1}^M \varphi(y_{\jh})\cdot \log y_{\jh} - \log \Big( \sum_{\ih=1}^M y_{\ih}\Big)\cdot \Big( \sum_{\jh=1}^M \varphi(y_{\jh}) \Big)^2.
	\end{equation*}
	We can pull out $\log \big(\sum y_{\ih}\big)>0$ and divide by it. This gives
	\begin{align*}
	\frac{D-B}{\log \big(\sum y_{\ih}\big)} &= \Big(\sum_{\ih=1}^M y_{\ih}\Big) \sum_{\jh=1}^M \varphi(y_{\jh})\cdot \frac{y_{\jh}\log y_{\jh}}{y_{\jh}} - \sum_{\jh=1}^M \varphi^2(y_{\jh}) - 2\sum_{\ih < \jh} \varphi(y_{\ih})\varphi(y_{\jh}) \\
	&= \sum_{\jh=1}^M \frac{\sum_{\ih=1}^M y_{\ih}}{y_{\jh}}\cdot \varphi^2(y_{\jh}) - \sum_{\jh=1}^M \varphi^2(y_{\jh}) - 2\sum_{\ih < \jh} \varphi(y_{\ih})\varphi(y_{\jh}) \\
	&= \sum_{\jh=1}^M \frac{\sum_{i\neq j} y_{\ih}}{y_{\jh}}\cdot \varphi^2(y_{\jh}) - 2\sum_{\ih < \jh} \varphi(y_{\ih})\varphi(y_{\jh}) \\
	&= \sum_{\ih < \jh}\Big( \frac{y_{\ih}}{y_{\jh}}\cdot \varphi^2(y_{\jh})+ \frac{y_{\jh}}{y_{\ih}}\cdot \varphi^2(y_{\ih}) - 2\varphi(y_{\ih})\varphi(y_{\jh}) \Big)  \\
	&=  \sum_{\ih < \jh} \Big( \sqrt{\frac{y_{\ih}}{y_{\jh}}}\varphi(y_{\jh}) - \sqrt{\frac{y_{\jh}}{y_{\ih}}}\varphi(y_{\ih}) \Big)^2 \geq 0.
	\end{align*}
\end{proof}

\begin{prop}\label{prop:diaghomoCondition}
The solution of the equation $R(x)=1$ is unique. Let $x_0$ denote this solution. Then in the diagonally homogeneous case
\begin{equation*}
\text{\eqref{cond:main} holds } \;\Longleftrightarrow\; \dfrac{\log b}{\log a}< x_0.
\end{equation*}
\end{prop}
\begin{remark}
Observe that all the conditions for transversality, \eqref{cond:main}, \eqref{cond:BoxDimCond} are satisfied if the heights of the parallelograms $R_i$ are "small enough" compared to their width. See the examples with overlaps in Section~\ref{sec:ex} for some explicit calculations.
\end{remark}

\begin{proof}[Proof of Proposition \ref{prop:diaghomoCondition}]
	Let $x:=\log b/\log a<1$. In the diagonally homogeneous case \eqref{cond:main} simplifies to
	\begin{equation*}
	\frac{\log a}{\log b} = \frac{1}{x} > 1+ \frac{\sum_{\ih=1}^M q_{\ih}^\ast \log N_{\ih}}{-\sum_{\ih=1}^M q_{\ih}^\ast \log q_{\ih}^\ast},
	\end{equation*}
	where $q_{\ih}^\ast=N_{\ih}^x/\sum_{\jh} N_{\jh}^x$. Multiplying each side by $x$ we get
	\begin{equation}\label{eq:inproofofProp1.7}
	1> x+ \frac{\sum_{\ih=1}^M q_{\ih}^\ast \log N_{\ih}^x}{-\sum_{\ih=1}^M q_{\ih}^\ast \log q_{\ih}^\ast}.
	\end{equation}
	It is straightforward to check that for any $y_1,\ldots,y_M\in\R$ and $q_{\ih}:=e^{y_{\ih}}/\sum_{\jh} e^{y_{\jh}}$
	\begin{equation*}
	-\sum_{\ih=1}^M q_{\ih}\log q_{\ih} + \sum_{\ih=1}^M q_{\ih}\cdot y_{\ih} = \log \sum_{\ih=1}^M e^{y_{\ih}}.
	\end{equation*}
	Applying this with $y_{\ih}=\log N_{\ih}^x$ (then $q_{\ih}=q_{\ih}^\ast$) in the denominator of \eqref{eq:inproofofProp1.7} we get that \eqref{cond:main} is equivalent to
	\begin{equation*}
	1> x+ \frac{\sum_{\ih=1}^M q_{\ih}^\ast \log N_{\ih}^x}{\log \sum_{\ih=1}^M N_{\ih}^x - \sum_{\ih=1}^M q_{\ih}^\ast \log N_{\ih}^x} = R(x).
	\end{equation*}
	For $x$ small enough \eqref{cond:main} holds, since $1/x$ tends to infinity while the right hand side remains finite. On the other hand for $x=1$ it does not hold. Hence, $R(x)<1$ for small enough $x$, while $R(1)\geq 1$. Thus, Lemma \ref{lemma:PropofR(x)} implies that there exists a unique $x_0\in(0,1)$ such that $R(x_0)=1$. So any $x<x_0$ satisfies \eqref{cond:main}.
\end{proof}

Finally, in the diagonally homogeneous case, the dimension formulas agree with the ones for Bedford--McMullen carpets.
\begin{cor}\label{cor:dimH_BMcarpet}
If a diagonally homogeneous shifted TGL carpet $\Lambda$ satisfies the conditions of Theorems~\ref{thm:maindimresult} and \ref{thm:BoxwOverlaps}, then
\begin{equation*}
\dim_{\mathrm H}\Lambda= \frac{1}{-\log b} \log \sum_{\jh=1}^M N_{\jh}^{\frac{\log b}{\log a}} \;\text{ and }\; \dim_{\mathrm B}\Lambda = \frac{\log N}{-\log a} + \left(1-\frac{\log b}{\log a}\right)\frac{\log M}{-\log b}.
\end{equation*}
In particular, $\dim_{\mathrm H}\Lambda = \dim_{\mathrm B}\Lambda$ if and only if $\Lambda$ has uniform vertical fibres.
\end{cor}
\begin{proof}
For diagonally homogeneous shifted TGL carpets the expression \eqref{def:D(p)} for $D(\mathbf{p})$  simplifies to
\begin{equation*}
D(\mathbf{p}) =  \frac{\log \langle\mathbf{p}\rangle_{\mathbf{p}}}{\log a} + \left( 1- \frac{\log b}{\log a}\right) \frac{\log \langle\mathbf{q}\rangle_{\mathbf{q}}}{ \log b}.
\end{equation*}
Applying this for $\mathbf{p}^\ast$ from \eqref{def:optpDiagHomo} gives the result $\dim_{\rm H}\Lambda=D(\mathbf{p}^\ast)$.

The equation for the box dimension $s=\dim_{\mathrm B}\Lambda$, recall \eqref{a74}, simplifies to
\begin{equation}\label{eq:BoxDimDiagHomo}
N\cdot b^{s_\mathcal{H}}\cdot a^{s-s_\mathcal{H}}=1.
\end{equation}
Since $\mathcal{H}$ has No Dimension Drop (recall Definition~\ref{def:34}), we have $s_\mathcal{H}=\log M/(-\log b)$. Substituting this back into \eqref{eq:BoxDimDiagHomo} and expressing $s$ from the equation gives the desired formula for $\dim_{\rm B} \Lambda$.

Comparing the formula for $\dim_{\rm B} \Lambda$ with the one for $D(\mathbf{p})$, we immediately get that equality holds if and only if $N_{\ih}=N/M$ for every $\ih\in\{1,\ldots,M\}$.
\end{proof}

\section{Preliminaries}\label{sec:Prelim}

In this section, we collect important notation, definitions, preliminary lemmas and cite results used in the proofs of the subsequent sections.
\subsection{Symbolic notation}\label{subsec:SymbNotation}
Throughout, we work simultaneously with the IFSs $\mathcal{F}, \widetilde{\mathcal{H}}$ and $\mathcal{H}$, which are defined in \eqref{def:IFS_F}, \eqref{def:verticalIFS} and \eqref{a72} respectively. Their attractors are $\Lambda$, $\Lambda_{\mathcal{H}}=\Lambda_{\widetilde{\mathcal{H}}}$  respectively.
We define the symbolic spaces
\begin{equation*}\label{a66}
\Sigma=\{1,2,\ldots,N\}^\mathbb N \;\text{ and }\; \Sigma_{\mathcal{H}}=\{1,2,\ldots,M\}^\mathbb N
\end{equation*}
with elements $\ii=i_1i_2\ldots\in\Sigma$ and $\iih=\ih_1\ih_2\ldots\in\Sigma_{\mathcal{H}}$. The function $\phi: \{1,2,\ldots,N\}\to\{1,2,\ldots,M\}$, recall \eqref{def:phiFunc}, naturally defines the map $\Phi: \Sigma\to\Sigma_{\mathcal{H}}$
\begin{equation}\label{def:Phi}
\Phi(\ii):= \iih = \phi(i_1)\phi(i_2)\ldots.
\end{equation}
Finite words of length $n$ are either denoted with a 'bar' like $\iiv=i_1\ldots i_n\in\Sigma_n$ or as a truncation $\ii|n=i_1\ldots i_n$ of an infinite word $\ii$, the length is denoted $|\cdot|$. The set of all finite length words is denoted by $\Sigma^\ast=\bigcup_n\Sigma_n$ and analogously $\Sigma_{\mathcal{H}}^\ast$. The left shift operator on $\Sigma$ and $\Sigma_{\mathcal{H}}$ is $\sigma$, i.e. $\sigma(\ii)=i_2i_3\ldots$ and $\sigma(\iih)=\ih_2\ih_3\ldots$.

The longest common prefix of $\ii$ and $\jj$ is denoted $\ii\wedge\jj$, i.e. its length is $|\ii\wedge\jj|=\min \{k:\; i_k\neq j_k\}-1$. This is also valid if one of them has or both have finite length. The $n$th level cylinder set of $\ii\in\Sigma$ is $[\ii|n] := \{\jj\in\Sigma:\; |\ii\wedge\jj|\geq n\}$. Similarly for $\iiv\in\Sigma_n$ and $\iih\in\Sigma_{\mathcal{H}}$.
Recall that $R=[0,1]^2$.
We use the standard notation $A_{\ii|n}=A_{i_1}\cdot \ldots\cdot A_{i_n}$ and  $f_{\ii|n}=f_{i_1}\circ f_{i_2}\circ\ldots \circ f_{i_n}$ to write
\begin{equation*}
\Lambda_{\ii|n} := f_{\ii|n}(\Lambda) \;\text{ and }\; R_{\ii|n}:=f_{\ii|n}(R)
\end{equation*}
for the $n$th level cylinder corresponding to $\ii$. The sets $\{R_{\ii|n}\}_{n=1}^\infty$ form a nested sequence of compact sets with diameter tending to zero, hence their intersection is a unique point $x\in\Lambda$. This defines the natural projection $\Pi:\Sigma\to\Lambda$
\begin{equation}\label{def:NaturalProj}
\Pi(\ii):=\lim_{n\to\infty} \bigcap_{n=1}^\infty R_{\ii|n}=\lim_{n\to\infty} f_{\ii|n}(\underline 0) = t_{i_1} + \sum_{n=2}^{\infty}A_{\ii|n-1}\cdot t_{i_n}.
\end{equation}
The natural projections generated by $\widetilde{\mathcal{H}}$ and $\mathcal{H}$ are
\begin{equation*}\label{a65}
\widetilde{\Pi}_{\mathcal{H}}(\ii) := \lim_{n\to\infty} \widetilde h_{\ii|n}(0), \;\; \ii\in\Sigma; \;\text{ and }\;
\Pi_{\mathcal{H}}(\iih) := \lim_{n\to\infty} h_{\iih|n}(0), \;\; \iih\in\Sigma_{\mathcal{H}}.
\end{equation*}
The following commutative diagram summarizes these notations:
\begin{equation}\label{eq:diagram}
\xymatrix{ \Sigma \ar[r]^\Phi  \ar[d]_\Pi \ar[rd]^{\widetilde{\Pi}_\mathcal{H}}  & %
	\Sigma_\mathcal{H} \ar[d]^{\Pi_\mathcal{H} }\\ %
	\Lambda \ar[r]_{\text{proj}_x} & \Lambda_{\mathcal{H}}  }%
\end{equation}
We also introduce the measurable partitions $\alpha$ and $\beta$ of $\Sigma$ whose classes containing an $\ii\in\Sigma$ are defined
\begin{equation}\label{eq:PartitionsAlphaBeta}
\alpha(\ii):= \Pi^{-1}\Pi(\ii) \;\text{ and }\; \beta(\ii):= \Phi^{-1}\Phi(\ii).
\end{equation}
The fact that these partitions are measurable are immediate consequences of the definition of measurability of a partition. Alternatively, this also follows from \cite[Theorem 2.2]{simmons2012conditional}.
Thus, $\alpha(\ii)$ contains those $\jj\in\Sigma$ for which $\Pi(\ii)=\Pi(\jj)\in\Lambda$ and $\beta(\ii)$ corresponds to the 'symbolic column' of $\ii$, i.e. for $\jj\in\beta(\ii)$ we have $\widetilde{\Pi}_{\mathcal{H}}(\ii)=\widetilde{\Pi}_{\mathcal{H}}(\jj)$. These partitions play an important role when handling overlaps.

Bernoulli measures on $\Sigma$ are key in obtaining the lower bound for $\dim_{\mathrm H}\Lambda$. Recall the set
\begin{equation*}
\mathcal{P}:= \{\mathbf{p}=(p_1,\ldots,p_N):\; p_i\geq 0,\; \sum_{i=1}^{N}p_i=1\}
\end{equation*}
of all probability distributions on the set $\{1,2,\ldots,N\}$ and let $\mathcal{P}_0$ denote the subset when all $p_i>0$. The Bernoulli measure on $\Sigma$ corresponding to $\mathbf{p}\in\mathcal{P}$ is the product measure $\mu_{\mathbf p}=\mathbf{p}^\N$, i.e. the measure of a cylinder set is $\mu_{\mathbf p}([\ii|n]) = p_{i_1}\cdot\ldots\cdot p_{i_n}$. All Bernoulli measures can be uniquely disintegrated according to the family of conditional measures $\mu_{\mathbf{p},\alpha(\ii)}=\mu_{\alpha(\ii)}$ generated by the measurable partition $\alpha$. That is for all Borel sets $U\subset \Sigma$ 
\begin{equation}\label{r96}
\mu_{\mathbf{p}}(U)=\int \mu_{\alpha(\mathbf{i})}(U)d\mu_{\mathbf{p}}(\mathbf{i}).
\end{equation}

The entropy of a Bernoulli measure $\mu_{\mathbf p}$ is
\begin{equation}\label{def:entropy}
h_{\mu_{\mathbf p}} = -\sum_{i=1}^N p_i \log p_i = -\log \langle\mathbf{p}\rangle_\mathbf{p}.
\end{equation}
The push forward $\nu_{\mathbf p}:= \Pi_\ast\mu_{\mathbf{p}}$ is the self-affine measure on $\Lambda$ defined by $\nu_{\mathbf p}=\mu_{\mathbf p}\circ\Pi^{-1}$ or equivalently
\begin{equation*}
\nu_{\mathbf p}= \sum_{i=1}^{N} p_i \nu_{\mathbf p}\circ f_i^{-1}.
\end{equation*}
Recall that a $\mathbf{p}\in\mathcal{P}$ defines another distribution $\mathbf{q}=(q_1,\ldots,q_M)$ via \eqref{def:q_l}. Then $\mu_{\mathbf q}=\mathbf{q}^{\N}$ is a Bernoulli measure on $\Sigma_{\mathcal{H}}$. Moreover, the self-similar measure on $\Lambda_{\mathcal{H}}$ is $\nu_{\mathbf q} = (\Pi_{\mathcal{H}})_\ast\mu_{\mathbf p} = (\mathrm{proj}_x)_\ast \nu_{\mathbf p}$. Our convention is that $\mu$ always denotes a measure on (some) symbolic space, while $\nu$ is supported on (a part of) $R$.

\subsection{Atypical parallelograms}\label{subsec:atypical}

The exponential rate of growth of the size of $n$th level parallelograms, the number of parallelograms in a column and the column's measure can vary a lot for different $\ii\in\Sigma$. However, in measure-theoretic sense those $\ii$ which behave atypically form a small set. Define the function
\begin{equation*}
X:\, \Sigma\to \R^+,\; X(\ii):= c_{i_1},
\end{equation*}
where $\mathbf{c}=(c_1,\ldots,c_N)$ is an arbitrary vector with strictly positive elements. Let
\begin{equation*}
X_n(\ii):= \prod_{j=0}^{n-1 }X(\sigma^j\ii) = \prod_{j=1}^{n}c_{i_j}.
\end{equation*}
In particular, if $\mathbf{c}=\mathbf{a}:=(a_1,\ldots,a_N)$ or $\mathbf{b}:=(b_1,\ldots,b_N)$, then $X_n(\ii)$ is the height and width of the parallelogram $R_{\ii|n}$. If $\mathbf{c}=\mathbf{N}:=(N_{\phi(1)},\ldots,N_{\phi(N)})$ or $\mathbf{q}:=(q_{\phi(1)},\ldots,q_{\phi(N)})$, then $X_n(\ii)$ gives the number of parallelograms in and the measure of the column $\Phi(\ii)|n$.

Fix an arbitrary $\mathbf{p}\in\mathcal{P}$. Recall the notation $\langle\mathbf{c}\rangle_{\mathbf{p}}:= \prod_{j=1}^Nc_j^{p_j}$. When no confusion is made, we suppress $\mathbf{p}$ and write $\langle\mathbf{c}\rangle=\langle\mathbf{c}\rangle_{\mathbf{p}}$. In the rest of the subsection $\delta>0$ is fixed. Define
\begin{equation}\label{r74}
\mathrm{Bad}_{\delta,n}(\mathbf{c}):=
\begin{cases}
\{\ii\in\Sigma:\; X_n(\ii)<\langle\mathbf{c}\rangle^{(1-\delta)n} \,\text{ or }\, X_n(\ii)>\langle\mathbf{c}\rangle^{(1+\delta)n}\}, & \text{if } \langle\mathbf{c}\rangle>1, \\
\{\ii\in\Sigma:\; X_n(\ii)<\langle\mathbf{c}\rangle^{(1+\delta)n} \,\text{ or }\, X_n(\ii)>\langle\mathbf{c}\rangle^{(1-\delta)n}\}, & \text{if } \langle\mathbf{c}\rangle<1.
\end{cases}
\end{equation}
The definition can be extended to a positive real $t$, by setting $\mathrm{Bad}_{\delta,t}(\mathbf{c}):= \mathrm{Bad}_{\delta,\lfloor t\rfloor}(\mathbf{c})$. Let $\mu_{\mathbf p}$ be the Bernoulli measure on $\Sigma$ defined by $\mathbf{p}\in\mathcal{P}$.
\begin{lemma}\label{lemma:BadisSmall}
	IF $\langle\mathbf{c}\rangle_{\mathbf{p}}\ne 1$ then there exists a constant $C$ and an $r\in(0,1)$ such that
	\begin{equation*}
	\mu_{\mathbf p}( \mathrm{Bad}_{\delta,n}(\mathbf{c}) ) < C\cdot r^n \;\text{ for every } n\geq1.
	\end{equation*}
	Hence, the Borel-Cantelli lemma immediately implies that
	\begin{equation*}
	\mu_{\mathbf p}(\ii\in\mathrm{Bad}_{\delta,n}(\mathbf{c}) \text{ for infinitely many } n) =0.
	\end{equation*}
\end{lemma}
\begin{proof}
	Assume $\langle\mathbf{c}\rangle>1$. Let $S_n(X):= \frac{1}{n}\sum_{j=0}^{n-1} \log X(\sigma^j\ii)$. Then
	\begin{equation*}
	\mu_{\mathbf p}(X_n(\ii) < \langle\mathbf{c}\rangle^{(1-\delta)n}) = \mu_{\mathbf p}( S_n(X) < (1-\delta)\log \langle\mathbf{c}\rangle).
	\end{equation*}
	The $\{\log X(\sigma^j\ii)\}_j$ are independent and identically distributed with expectation
	\begin{equation*}
	\mathds{E}(\log X) = \sum_{j=1}^N p_j\log c_j = \log \langle\mathbf{c}\rangle.
	\end{equation*}
	Hence, Cram\'er's large deviation theorem \cite[Theorem 2.1.24.]{DemboZeitouniLDP} implies that $\mu_{\mathbf p}(X_n(\ii) < \langle\mathbf{c}\rangle^{(1-\delta)n})$ decays exponentially fast in $n$. The argument for $X_n(\ii)>\langle\mathbf{c}\rangle^{(1+\delta)n}$ is exactly the same. The proof is analogous when $\langle\mathbf{c}\rangle<1$.
\end{proof}


\subsection{Ledrappier--Young formula}\label{subsec:LYformula}

Let $0<\alpha_2(A)\leq\alpha_1(A)<1$ denote the two singular values of a $2\times 2$ non-singular matrix $A$. Namely, $\alpha_i(A)$ is the positive square root of the $i$th largest eigenvalue of $A^TA$, where $A^T$ is the transpose of $A$. The geometric interpretation of the singular values is that the linear map $\underline{x} \mapsto A\underline{x}$ maps the unit disk to an ellipse with principal semi-axes of length $\alpha_2(A)$ and $\alpha_1(A)$. The singular values can also be expressed with the matrix norm: $\alpha_1(A)=\|A\|$ and $\alpha_2(A)=\|A^{-1}\|^{-1}$. For a family of matrices $\mathcal{A}=\{A_1,\ldots,A_N\}$, the asymptotic exponential growth rate of the semi-axes of the ellipses determined by the maps $\underline{x} \mapsto A_{i_1\ldots i_n}\underline{x}$ is given by the Oseledets theorem.
\begin{theorem}[Oseledets \cite{Ose68}]
	Let $\mathcal{A}=\{A_1,\ldots,A_N\}$ be a set of non-singular $2\times 2$ matrices with
	$\|A_i\|<1$ for $i\in\{1,\ldots,N\}$. Then for any ergodic $\sigma$-invariant measure $\mu$ on $\Sigma$ there exist constants $0 <\chi_{\mu}^1\leq \chi_{\mu}^2$ such that
	for $\mu$-almost every $\ii$
	\begin{align*}
	\lim_{n\to\infty}\frac 1 n \log \alpha_1(A_{i_1\ldots i_n}) &= \lim_{n\to\infty}\frac 1 n \log \|A_{i_1\ldots i_n}\| = -\chi_{\mu}^1, \\
	\lim_{n\to\infty}\frac 1 n \log \alpha_2(A_{i_1\ldots i_n}) &= \lim_{n\to\infty}\frac 1 n \log \|(A_{i_1\ldots i_n})^{-1}\|^{-1} = -\chi_{\mu}^2.
	\end{align*}
	The numbers $\chi_{\mu}^1$ and $\chi_{\mu}^2$ are called the Lyapunov-exponents of $\nu=\Pi_\ast\mu$. If $\chi_{\mu}^1\neq\chi_{\mu}^2$ then we say that $\mu$ has simple Lyapunov spectrum.
\end{theorem}

It is an easy exercise to calculate the Lyapunov exponents of Bernoulli measures $\mu_{\mathbf p}$ for a family of lower triangular matrices for which direction-$x$ dominates. For greater generality see Falconer--Miao \cite{FalconerMiao07}.
\begin{lemma}\label{lemma:Lyap_exp}
	Fix any $\mathbf{p}\in\mathcal{P}$ and a family of lower triangular matrices $\mathcal{A}=\{A_1,\ldots,A_N\}$ for which direction-$x$ dominates. Then the Lyapunov spectrum of the Bernoulli measure $\mu_{\mathbf p}$ is simple and the exponents equal
	\begin{equation*}\label{b63}
	\chi_{\nu_{\mathbf p}}^1 = -\sum_{i=1}^N p_i \log b_i = -\log \langle \mathbf{b} \rangle_{\mathbf{p}} \;\text{ and }\; \chi_{\nu_{\mathbf p}}^2 = -\sum_{i=1}^N p_i \log a_i = -\log \langle \mathbf{a} \rangle_{\mathbf{p}}.
	\end{equation*}
\end{lemma}
\begin{proof}[Sketch of proof]
	Both the singular values or the norm of $A_{i_1\ldots i_n}$ can be calculated directly. Since direction-$x$ dominates, the off-diagonal element does not play a role. An application of Oseledets theorem and the strong law of large numbers concludes the proof.
\end{proof}

The Ledrappier--Young formula originates from the seminal work of Ledrappier and Young~\cite{LedrappierYoungI_1985, LedrappierYoungII_1985} on determining the Hausdorff dimension of invariant measures of diffeomorphisms on compact manifolds. Through a succession of papers by Przytycki--Urba\'nski~\cite{PrzytyckiUrbanski1989}, Feng--Hu~\cite{FengHu09}, B\'ar\'any~\cite{barany15_LYformula} and B\'ar\'any--K\"{a}enm\"{a}ki~\cite{BARANYAnti201788} the formula was proved for the Hausdorff dimension of wider and wider classes of self-affine measures. In fact, Feng \cite{FengOral} recently announced that the Hausdorff dimension of the push-forward of a shift-invariant, ergodic measure $\mu$ satisfies a Ledrappier--Young type formula in full generality for any self-affine IFS on $\R^d$ which is contracting on average with respect to $\mu$. Also observe that the formulas proved in the earlier works of \cite{BARANSKIcarpet_2007, BaranskiTriag_2008, Bedford84_phd, GatzourasLalley92, mcmullen84} are all special cases of the Ledrappier--Young formula. The main result of \cite[Theorem 2.4, Corollary 2.8]{BARANYAnti201788} can be stated in a simpler form in our context when direction-$x$ dominates.

\begin{theorem}[\cite{BARANYAnti201788}, direction-$x$ dominates]\label{thm:BaranyKaenmaki}
	Let $\mathcal{F}$ be a shifted TGL-type IFS of the form \eqref{def:IFS_F}. Furthermore, using the notation from Subsection \ref{subsec:SymbNotation}, let $\mu_{\mathbf p}$ be any Bernoulli measure on $\Sigma$, $\nu_{\mathbf p}=\Pi_\ast\mu_{\mathbf{p}}$ its push forward and $\nu_{\mathbf q}=(\mathrm{proj}_x)_\ast \nu_{\mathbf p}$. Then, regardless of overlaps, $\nu_{\mathbf p}$ is exact dimensional and satisfies the Ledrappier--Young formula
	\begin{equation}\label{eq:dimHmuBaranyKaenmaki}
	\dim_{\rm H} \nu_{\mathbf p} = \frac{h_{\mu_{\mathbf p}}-H}{\chi_{\nu_{\mathbf p}}^2} + \left(1-\frac{\chi_{\nu_{\mathbf p}}^1}{\chi_{\nu_{\mathbf p}}^2}\right)\dim_{\rm H} \nu_{\mathbf q},
	\end{equation}
	where $H=-\int \log \mu_{\mathbf p, \alpha(\ii)} ([i_1]) \mathrm{d}\mu_{\mathbf p}(\ii)$. Recall $\{\mu_{\mathbf p, \alpha(\ii)}\}$ is the family of conditional measures of $\mu_{\mathbf p}$ defined by the measurable partition $\alpha(\ii)=\Pi^{-1}(\Pi(\ii))$.
	
	Moreover, if the IFS satisfies the ROSC and $\mathbf{p}\in\mathcal{P}_0$, then $H=0$.
\end{theorem}

\section{Upper bound for \texorpdfstring{$\dim_{\rm H}\Lambda$}{dimHLambda}}\label{sec:upperbound}

Consider a shifted triangular Gatzouras--Lalley-type planar carpet $\Lambda$ without any separation condition. To prove Theorem~\ref{thm:mainresUpperbound} we essentially lift the original argument in \cite{GatzourasLalley92}, formulated on the attractor $\Lambda$, to the symbolic space $\Sigma$. This can be done because the method in \cite{GatzourasLalley92} is completely symbolic in nature. Therefore, we only give a short sketch. 

The first step is to define a proper metric on $\Sigma$, which captures the distance between points on the attractor. Observe that for two points $\ii,\jj\in\Sigma$ the distance $|\Pi(\ii)-\Pi(\jj)|$ (recall \eqref{def:NaturalProj}) can be small even if $|\ii\wedge\jj|$ is small. This occurs if $|\Phi(\ii)\wedge\Phi(\jj)|=|\iih\wedge\jjh|$ (recall \eqref{def:Phi}) is much larger than $|\ii\wedge\jj|$, i.e. the corresponding cylinders belong to the same column for a long time.

\begin{lemma}\label{lemma:MetricSpace}
$(\Sigma,d)$ is a metric space, where the distance between $\ii,\jj\in\Sigma$ is defined
\begin{equation*}
d(\ii,\jj):= \prod_{k=1}^{|\iih\wedge\jjh|}b_{i_k} + \prod_{k=1}^{|\ii\wedge\jj|}a_{i_k}.
\end{equation*}
\end{lemma}

\begin{proof}
The fact that $d$ is non-negative and symmetric is trivial. Need to check the triangle inequality, for all $\ii,\jj,\kk\in\Sigma:\; d(\ii,\jj)\leq d(\ii,\kk)+d(\kk,\jj)$. If
\begin{itemize}
	\item $|\ii\wedge\kk|\leq|\ii\wedge\jj|$ then $\prod_{\ell=1}^{|\ii\wedge\kk|}a_{i_\ell}\geq \prod_{\ell=1}^{|\ii\wedge\jj|}a_{i_\ell}$,
	\item $|\ii\wedge\kk|>|\ii\wedge\jj|$ then $|\kk\wedge\jj|=|\ii\wedge\jj|$, thus $\prod_{\ell=1}^{|\kk\wedge\jj|}a_{i_\ell} = \prod_{\ell=1}^{|\ii\wedge\jj|}a_{i_\ell}$.
\end{itemize}
Analogously, if
\begin{itemize}
	\item $|\iih\wedge\kkh|\leq|\iih\wedge\jjh|$ then $\prod_{\ell=1}^{|\iih\wedge\kkh|}b_{i_\ell}\geq \prod_{\ell=1}^{|\iih\wedge\jjh|}b_{i_\ell}$,
	\item $|\iih\wedge\kkh|>|\iih\wedge\jjh|$ then $|\kkh\wedge\jjh|=|\iih\wedge\jjh|$, thus $\prod_{\ell=1}^{|\kkh\wedge\jjh|}b_{i_\ell} = \prod_{\ell=1}^{|\iih\wedge\jjh|}b_{i_\ell}$.
\end{itemize}
The triangle inequality now follows.
\end{proof}

The next step is to prove that the natural projection with this metric is Lipschitz.
\begin{lemma}\label{lemma:dimSigmageqdimLambda}
For any shifted triangular Gatzouras--Lalley-type planar carpet
\begin{equation*}
\dim_{\rm H}\Lambda \leq \dim_{\rm H} (\Sigma,d).
\end{equation*}
\end{lemma}

\begin{proof}

It is enough to show that there exists $C>0$ such that $|\Pi(\ii)-\Pi(\jj)|\leq C\cdot d(\ii,\jj)$, i.e. $\Pi:\Sigma\to\Lambda$ is a Lipschitz-function, which can not increase the Hausdorff dimension. 

For $\ii,\jj\in\Sigma$ let $k:=|\ii\wedge\jj|,\; \ell:=|\iih\wedge\jjh|$ and
\begin{equation*}
\begin{pmatrix} x \\ y \end{pmatrix} := \Pi(\ii)-\Pi(\jj) = \sum_{n=1}^{\infty}\begin{pmatrix} a_{\ii|n-1} t_{i_n,1} - a_{\jj|n-1} t_{j_n,1} \\ d_{\ii|n-1}t_{i_n,1} - d_{\jj|n-1}t_{j_n,1} + b_{\ii|n-1}t_{i_n,2} - b_{\jj|n-1}t_{j_n,2} \end{pmatrix}.
\end{equation*}
The first $k$ terms coincide in both coordinates and $b_{i_n}=b_{j_n}$ for $n=1,\ldots,\ell$. Thus,
\begin{align*}
x^2 &= a_{\ii|k}^2 \cdot \Big(\sum_{n=k+1}^{\infty}(a_{\sigma^k\ii|n-k-1} t_{i_n+k,1} - a_{\sigma^k\jj|n-k-1} t_{j_n+k,1}) \Big)^2 \leq a_{\ii|k}^2 \,,  \\
y^2 &\leq 2\Big[ \Big(\sum_{n=1}^{\infty}d_{\ii|n-1}t_{i_n,1}\Big)^2 + \Big(\sum_{n=1}^{\infty} d_{\jj|n-1}t_{j_n,1}\Big)^2 + \Big(\sum_{n=1}^{\infty}(b_{\ii|n-1}t_{i_n,2} - b_{\jj|n-1}t_{j_n,2})\Big)^2 \Big].
\end{align*}
In the first two sums using Lemma \ref{lemma:d/b_bounded} we can bound $d_{\ii|n-1}\leq K_0\cdot b_{\ii|n-1}$ and $d_{\jj|n-1}\leq K_0\cdot b_{\jj|n-1}$. Now we can pull out $b_{\ii|\ell}$ from all three sums. The remaining sums are all uniformly bounded in $\ii,\jj$ by some constant $c$. This gives
\begin{equation*}
y^2\leq 2K_0^2\cdot c\cdot b_{\ii|\ell}^2, \;\text{ thus }
|\Pi(\ii)-\Pi(\jj)| \leq \sqrt{a_{\ii|k}^2+2K_0^2\cdot c\cdot b_{\ii|\ell}^2} \leq C\cdot d(\ii,\jj).
\end{equation*}
\end{proof}

It remains to show that the value $\alpha^\ast$ maximizing the expression for $D(\mathbf{p})$ in \eqref{def:D(p)} is an upper bound for the Hausdorff dimension of $(\Sigma,d)$.
\begin{prop}\label{prop:dimSigmaUpperBound}
For any choice of parameters defining a shifted triangular Gatzouras--Lalley-type triangular carpet
\begin{equation*}
\dim_{\rm H} (\Sigma,d) \leq \alpha^\ast.
\end{equation*}
\end{prop}
\begin{proof}[Proof of Theorem \ref{thm:mainresUpperbound}]
The upper bound is a corollary of Lemma \ref{lemma:dimSigmageqdimLambda} and Proposition \ref{prop:dimSigmaUpperBound}. The compactness of $\mathcal{P}$ and the continuity of $D(\mathbf{p})$ implies that $\sup_{\mathbf{p}} D(\mathbf{p})$ is attained for some $\mathbf{p}^\ast\in\mathcal{P}$. Moreover, it is easy to check that $\mathbf{p}^\ast\in\mathcal{P}_0$, see \cite[Proposition 3.4]{GatzourasLalley92}.
\end{proof}
Proposition \ref{prop:dimSigmaUpperBound} is essentially proved in \cite[Section 5]{GatzourasLalley92} formulated on the attractor $\Lambda$. For completeness we sketch the main steps adapted to $(\Sigma,d)$ and cite \cite{GatzourasLalley92} when necessary. Most of the notation we bring over from \cite{GatzourasLalley92}.


The balls in $(\Sigma,d)$ are exactly the "approximate squares" defined in \cite[eq. (1.2)]{GatzourasLalley92}
\begin{align}
B_k(\ii)&:= \{\jj\in\Sigma: |\ii\wedge\jj|\geq L_k(\ii) \text{ and } |\iih\wedge\jjh|\geq k\}, \text{ where} \label{def:ApproxSquare}\\
L_k(\ii)&:= \max\Big\{n\geq 0: \prod_{j=1}^k b_{i_j} \leq \prod_{j=1}^n a_{i_j}\Big\}. \nonumber
\end{align}
Note, $k> L_k(\ii)$ for every $\ii$ and $k$, since $a_i<b_i$ for every $i$. The $\jj\in B_k(\ii)$ for which $|\ii\wedge\jj|=L_k(\ii)$ and $|\iih\wedge\jjh|=k$ are the ones for which $d(\ii,\jj)$ is maximal. The definition of $L_k(\ii)$ implies that
\begin{equation}\label{eq:RatioProdaoverProdb}
1\leq \frac{\prod_{j=1}^{L_k(\ii)}a_{i_j} }{ \prod_{j=1}^{k}b_{i_j} } \leq \max_i a_i^{-1}  \;\;\text{ for every } k,\ii.
\end{equation}
Hence, $\mathrm{diam} B_k(\ii)\leq C\cdot \prod_{j=1}^{k}b_{i_j}$ for some $C$ independent of $\ii$.

The main ingredient is a form of the mass distribution principle adapted to $(\Sigma,d)$.
\begin{lemma}\label{lemma:MassDistrPrinciple}
	Let $\mu$ be a probability measure on $\Sigma$ and assume
	\begin{equation*}
	\liminf_{k\to\infty} \frac{\log \mu(B_k(\ii))}{\log \prod_{j=1}^k b_{i_j} }\leq \alpha\; \text{ for \textbf{every} } \ii\in\Sigma,
	\end{equation*}
	where $B_k(\ii)$ is the approximate square defined in \eqref{def:ApproxSquare}. Then
	\begin{equation*}
	\dim_{\rm H}(\Sigma,d)\leq \alpha.
	\end{equation*}
\end{lemma}
\begin{proof}
	The assumption states that for every $\varepsilon,\delta>0$ and $\ii\in\Sigma$ there exists a $k(\ii)$ such that
	\begin{equation*}
	\prod_{j=1}^{k(\ii)} b_{i_j}<\delta \;\text{ and }\; \Big(\prod_{j=1}^{k(\ii)} b_{i_j}\Big)^{\alpha+\varepsilon}\leq \mu(B_{k(\ii)}(\ii)).
	\end{equation*}
	The collection $\{B_{k(\ii)}(\ii)\}_{\ii\in\Sigma}$ is a $\delta$-cover of $\Sigma$, thus the Vitali- or $5r$-covering lemma \cite{FalconerBookI_1986} implies that there exists a (perhaps uncountable) sub-collection $J\subset\Sigma$ of \textit{disjoint} balls $B_{k(\ii)}(\ii)$ giving a $5\delta$-cover of $\Sigma$, i.e.
	\begin{equation*}
	\Sigma \subseteq \bigsqcup_{\ii\in J} 5 B_{k(\ii)}(\ii) \;\text{ and }\; B_{k(\ii)}(\ii) \cap B_{k(\jj)}(\jj)=\emptyset \,\text{ for every }\, \ii\neq\jj\in J.
	\end{equation*}
	Hence, we can bound the $\alpha+\varepsilon$-dimensional Hausdorff measure
	\begin{equation*}
	\mathcal{H}^{\alpha+\varepsilon}_{5\delta}(\Sigma) \leq (5c)^{\alpha+\varepsilon} \sum_{\ii\in J} \Big(\prod_{j=1}^{k(\ii)} b_{i_j}\Big)^{\alpha+\varepsilon} \leq (5c)^{\alpha+\varepsilon} \sum_{\ii\in J} \mu(B_{k(\ii)}(\ii)) \leq (5c)^{\alpha+\varepsilon}\cdot \mu(\Sigma)
	\end{equation*}
	independent of $\delta$ and therefore $\mathcal{H}^{\alpha+\varepsilon}(\Sigma)\leq (5c)^{\alpha+\varepsilon}<\infty$ for every $\varepsilon>0$. Thus, $\dim_{\rm H}(\Sigma,d)\leq \alpha$.
\end{proof}
The lemma implies that to prove Proposition \ref{prop:dimSigmaUpperBound} it is enough to find a measure $\mu$ satisfying the condition of the lemma with the value $\alpha^\ast$. This can be achieved using the family of Gatzouras--Lalley Bernoulli measures introduced in \cite[eq. (5.2)]{GatzourasLalley92}. Let $\vartheta\in\R, \lambda\in\R$ and $\rho\in(0,1)$. Define the probability vector $\mathbf{p}=(p_1,\ldots,p_N)$ by
\begin{equation}\label{def:GLBernoulliMeasure}
p_i=p_i(\vartheta,\lambda,\rho):= C(\vartheta,\lambda,\rho)a_i^\vartheta b_i^{\lambda-\vartheta}(\gamma_i(\vartheta))^{\rho-1},\; \text{ where } \gamma_i(\vartheta) = \sum_{j\in\mathcal{I}_{\phi(i)}} a_j^\vartheta
\end{equation}
and $C(\vartheta,\lambda,\rho)$ normalizes so that $\sum_ip_i=1$. In fact \cite[Lemma 5.1]{GatzourasLalley92} shows that there exists a real-valued continuous function $\vartheta(\rho),\, \rho\in(0,1)$, such that for every $\rho\in(0,1)$
\begin{equation*}
C(\vartheta(\rho),\alpha^\ast,\rho)=1.
\end{equation*}
From now we work with such $\mathbf{p}$.
\begin{lemma}\label{lemma:GLMeasure}
	The Bernoulli-measure $\mu:= \mathbf{p}^{\N}$ on $\Sigma$ satisfies the condition of Lemma \ref{lemma:MassDistrPrinciple} with the optimal value $\alpha^\ast$, i.e.
	\begin{equation*}
	\liminf_{k\to\infty} \frac{\mu(B_k(\ii))}{\log \prod_{j=1}^k b_{i_j}} \leq \alpha^\ast\; \text{ for \textbf{every} } \ii\in\Sigma.
	\end{equation*}
\end{lemma}
\begin{proof}[Sketch of proof]
	By definition of $B_k(\ii)$
	\begin{equation*}
	\mu(B_k(\ii)) = \prod_{j=1}^{L_k(\ii)}p_{i_j}\cdot \prod_{j=L_k(\ii)+1}^{k} q_{\phi(i_j)} = \Big(\prod_{j=1}^{k} b_{i_j}\Big)^{\alpha^\ast} \cdot \frac{\prod_{j=1}^{L_k(\ii)}a_{i_j}^\vartheta}{\prod_{j=1}^{k}b_{i_j}^\vartheta} \cdot \frac{\big(\prod_{j=1}^k \gamma_{i_j}(\vartheta)\big)^\rho}{\prod_{j=1}^{L_k(\ii)} \gamma_{i_j}(\vartheta)},
	\end{equation*}
	where $q_{\phi(i_j)} =\sum_{\ell\in\mathcal{I}_{\phi(i_j)}}p_{\ell}$. Taking logarithm and dividing by $\log \prod_{j=1}^k b_{i_j}$ gives
	\begin{equation*}
	\frac{\log \mu(B_k(\ii))}{\log \prod_{j=1}^k b_{i_j}} = \alpha^\ast + \frac{\vartheta \log(\prod_{j=1}^{L_k(\ii)}a_{i_j} / \prod_{j=1}^{k}b_{i_j}) }{\log \prod_{j=1}^k b_{i_j}} +  \frac{ \log \frac{\big(\prod_{j=1}^k \gamma_{i_j}(\vartheta)\big)^\rho}{\prod_{j=1}^{L_k(\ii)} \gamma_{i_j}(\vartheta)} }{\log \prod_{j=1}^k b_{i_j}}.
	\end{equation*}
	Due to \eqref{eq:RatioProdaoverProdb}, the second term tends to zero as $k\to\infty$. We can increase the third term by replacing the denominator with $k\cdot \log \min_i b_i$. Hence, it is enough to prove that there exists $\rho\in(0,1)$ such that for $\vartheta=\vartheta(\rho)$
	\begin{equation*}
	\limsup_{k\to\infty}\, \frac{\rho}{k} \sum_{j=1}^k\log \gamma_{i_j}(\vartheta) - \frac{1}{k} \sum_{j=1}^{L_k(\ii)}\log \gamma_{i_j}(\vartheta) \,\geq\, 0.
	\end{equation*}
	This is exactly the statement in \cite[eq. (5.10)]{GatzourasLalley92}. For details see \cite[pg. 565-566]{GatzourasLalley92}.
\end{proof}

\begin{proof}[Proof of Proposition \ref{prop:dimSigmaUpperBound}]
	The Proposition is a direct corollary of Lemmas \ref{lemma:MassDistrPrinciple} and \ref{lemma:GLMeasure}.
\end{proof}

\section{Proof of Theorem~\ref{thm:maindimresult}}\label{sec:dimH_lowerbound}

Our goal is to show that the Ledrappier--Young formula \eqref{eq:dimHmuBaranyKaenmaki} of \cite{BARANYAnti201788} for $\dim_{\mathrm H}\nu_{\mathbf p}$, cited in Theorem~\ref{thm:BaranyKaenmaki},  always equals the formula for $D(\mathbf{p})$ in \eqref{def:D(p)} under the conditions of Theorem~\ref{thm:maindimresult}. For the rest of this proof, we fix a $\mathbf{p}\in\mathcal{P}_0$ and assume $\mathcal{H}$ satisfies Hochman's Exponential Separation Condition and either each column independently satisfies the ROSC or $\Lambda$ satisfies transversality and~\eqref{cond:main}.

The entropy of the system is $h_{\mu_{\mathbf p}}=-\log \langle \mathbf{p} \rangle_{\mathbf{p}}$ (recall \eqref{def:entropy}), the Lyapunov-exponents from Lemma~\ref{lemma:Lyap_exp} are $\chi_{\nu_{\mathbf p}}^2=-\log \langle\mathbf{a}\rangle_{\mathbf{p}}$ and $\chi_{\nu_{\mathbf p}}^1=-\log \langle\mathbf{b}\rangle_{\mathbf{p}}\,$. Hochman's Exponential Separation Condition for $\mathcal{H}$ implies No Dimension Drop for $\nu_{\mathbf q}$, recall \eqref{eq:CondForH}, hence $\dim_{\rm H} \nu_{\mathbf q}=\log \langle\mathbf{q}\rangle_{\mathbf{q}} / \log \langle\mathbf{b}\rangle_{\mathbf{p}}\,$. As a result, to prove the theorem it is enough to show that the integral
\begin{equation*}
H=-\int \log \mu_{\alpha(\ii)} ([i_1]) \mathrm{d}\mu_{\mathbf p}(\ii) = 0,
\end{equation*}
where $\{\mu_{\alpha(\ii)}\}$ is the family of conditional measures of $\mu_{\mathbf p}$ defined by the measurable partition $\{\alpha(\ii)=\Pi^{-1}(\Pi(\ii))\}$, recall \eqref{r96}. Since $-\log \mu_{\alpha(\ii)} ([i_1])\geq 0$, we have that $H=0$ if and only if
\begin{equation}\label{r85}
\mu_{\alpha(\ii)} ([i_1])=1 \mbox{ for $\mu_{\mathbf p}$-a.a. }\ii.
\end{equation}
Thus, it suffices to show that $\mu_{\alpha(\ii)}$ is concentrated on $\ii$ for $\mu_{\mathbf p}$-typical $\ii$. Overlaps arising from the translations of columns or from intersections within a column can in theory cause problems. However, the next two results ensure that there is a full measure subset of $\Sigma$ for which $\mu_{\alpha(\ii)}$ is a point mass distribution.

Recall from \eqref{eq:PartitionsAlphaBeta} that $\beta(\ii)=\Phi^{-1}\Phi(\ii)$ is the 'symbolic column' if $\ii$. The first claim ensures that there is a full measure subset $\Sigma_1\subset \Sigma$ where the translations of the columns have no effect.
\begin{claim}\label{r78} Assume Weak Almost Unique Coding holds for $\Sigma_{\mathcal{H}}$, recall Definition~\ref{def:34}. Then there exists a full measure subset $\Sigma_1 \subset \Sigma$ such that
	for all
	$ \mathbf{i}\in\Sigma_1$  and for all
	$ (\jh_1, \dots ,\jh_n)\ne
	(\ih_1, \dots ,\ih_n)$
	\begin{equation}\label{r93}
	\mu_{\alpha(\mathbf{i})}([j_1, \dots ,j_n])=0,
	\end{equation}
	where $\phi(i_k)=\ih_k$ and $\phi(j_k)=\jh_k$ for $k=1,\ldots,n$, recall \eqref{def:phiFunc} for the definition of $\phi$.
	
	Consequently, for every $\ii\in\Sigma_1$ we have
	\begin{equation}\label{r77}
	\mu_{\alpha(\ii)}(\beta(\ii)^c)=0.
	\end{equation}
\end{claim}
The second claim defines the full-measure set $\Sigma_2\subset\Sigma$ where intersections within columns have no effect.
\begin{prop}\label{r69}
	Assume that the  conditions of Theorem~\ref{thm:maindimresult} hold. Then there exists a ${\Sigma}_2 \subset \Sigma$, with
	$\mu_{\mathbf{p}}({\Sigma}_2)=1$ such that for every $\ii\in{\Sigma}_2$ and $k\in \mathcal{I}_{\phi(i_1)}\setminus \left\{i_1\right\}$
	\begin{equation*}\label{r70}
	\mu_{\alpha(\mathbf{i})}
	\left(
	\beta(\ii)\cap\alpha(\ii)\cap[k]
	\right)=0.
	\end{equation*}
\end{prop}
Theorem~\ref{thm:maindimresult} is a corollary of these two results. Sometimes we use the following notation:
\begin{definition}\label{r64}
	Let $F \subset \Sigma$ be a subset of full measure. Then we define
	\begin{equation*}\label{r63}
	\widetilde{F}:=\left\{\ii\in F:
	\mu_{\alpha(\ii)}(\Sigma\setminus F)=0
	\right\}.
	\end{equation*}
	Since $\mu_{\mathbf{p}}(F)=1$, the disintegration formula \eqref{r96} implies that $\mu_{\mathbf{p}}(\widetilde{F})=1$.
\end{definition}

\subsection{The proof of Theorem~\ref{thm:maindimresult} assuming Claim \ref{r78} and Proposition \ref{r69}}

\begin{proof}[Proof of Theorem~\ref{thm:maindimresult} assuming  Claim \ref{r78} and Proposition \ref{r69}]
	As we established above in \eqref{r85} that to prove the theorem it is enough to check that
	\begin{equation}\label{r84}
	\mu_{\alpha(\mathbf{i})}
	\left( \alpha(\mathbf{i})\cap [i_1]^c\right)
	=0, \mbox{ for $\mu$-a.a. } \mathbf{i}.
	\end{equation}
	Clearly,
	\begin{equation*}\label{r88}
	\alpha(\mathbf{i})\cap [i_1]^c \subset
	\Bigg(\bigcup\limits_{k\not\in\mathcal{I}_{\phi(i_1)}}\!
	(\alpha(\mathbf{i})\cap[k])\Bigg)
	\;\cup\;
	\Bigg(\bigcup\limits_{k\in\mathcal{I}_{\phi(i_1)}\setminus\left\{i_1\right\}}\!\!\!
	(\alpha(\mathbf{i})\cap[k])\Bigg).
	\end{equation*}
	It follows from \eqref{r93} that for every $\ii\in\Sigma_1$
	\begin{equation}\label{r86}
	\mu_{\alpha(\mathbf{i})}
	\Bigg( \bigcup\limits_{k\not\in\mathcal{I}_{\phi(i_1)}}
	(\alpha(\mathbf{i})\cap[k])\Bigg)=0,
	\end{equation}
	where $\Sigma_1$ is defined in Claim \ref{r78}.
	Thus, to prove the theorem we only need to verify that
	\begin{equation}\label{r82}
	\mu_{\alpha(\mathbf{i})}
	\Bigg(
	\bigcup\limits_{k\in\mathcal{I}_{\phi(i_1)}\setminus\left\{i_1\right\}}
	(\alpha(\mathbf{i})\cap[k])
	\Bigg)=0 \mbox{ for $\mu$-a.a. } \mathbf{i}.
	\end{equation}
	We can write
	\begin{multline*}\label{r80}
	\bigcup\limits_{k\in\mathcal{I}_{\phi(i_1)}\setminus\left\{i_1\right\}}
	(\alpha(\mathbf{i})\cap[k])
	\subset
	\Sigma_1^c\cup
	\Sigma_2^c \\
	\cup
	\underbrace{\Bigg(\bigcup\limits_{k\in\mathcal{I}_{\phi(i_1)}\setminus\left\{i_1\right\}}
		(\alpha(\mathbf{i})\cap\beta(\ii)\cap[k])\Bigg)}_U
	\;\cup\;
	\underbrace{\Bigg(\bigcup\limits_{k\in\mathcal{I}_{\phi(i_1)}\setminus\left\{i_1\right\}}
		(\alpha(\mathbf{i})\cap\beta(\ii)^c\cap[k])\Bigg)}_V.
	\end{multline*}
	
	It follows from Proposition~\ref{r69} that
	$\mu_{\alpha(\ii)}(U)=0$ for all $\ii\in{\Sigma}_2$ and it follows from Claim~\ref{r78} that
	$ \mu_{\alpha(\mathbf{i})}(V)=0$ for all $ \mathbf{i}\in\Sigma_1$.
	So, for all $\ii\in\Sigma_1\cap{\Sigma}_2$ \eqref{r82} holds, which together with \eqref{r86} yields that \eqref{r84} holds. This completes the proof of Theorem~\ref{thm:maindimresult} assuming Claim~\ref{r78} and Proposition~\ref{r69}.
\end{proof}

\subsection{The proof of Claim  \ref{r78}}

\begin{proof}[Proof of Claim \ref{r78}]
	In the definition of Weak Almost Unique Coding, recall Definition~\ref{def:34}, there is a set $\mathcal{B}_{\mathcal{H}} \subset \Sigma_\mathcal{H}$ defined in such a way that for $\Sigma'_\mathcal{H}:=\Sigma_\mathcal{H}\setminus \mathcal{B}_\mathcal{H}$ we have $\mu_{\mathbf{q}}(\Sigma'_\mathcal{H})=1$ and
	\begin{equation*}\label{r99}
	\iih\in \Sigma'_\mathcal{H}
	\;\Longleftrightarrow\;
	\Sigma'_\mathcal{H}\cap \big( \Pi_{\mathcal{H}}^{-1}\Pi_{\mathcal{H}}(\iih) \big)=
	\left\{\iih\right\},
	\end{equation*}
	where $\Pi_\mathcal{H}$ is the natural projection from $\Sigma_\mathcal{H}$ to $\Lambda_{\mathcal{H}}$. Let
	$$
	\mathcal{B}:=\Phi^{-1}(\mathcal{B}_\mathcal{H}) \;\mbox{ and }\;
	\Sigma':=
	\Phi^{-1}\left(\Sigma'_\mathcal{H}\right).
	$$
	Since $\mu_{\mathbf{q}}(\Sigma'_{\mathcal{H}})=1$ we can define $\Sigma_1 :=\widetilde{\Sigma'}$ (recall the notation $\widetilde{}\ $ from Definition~\ref{r64}) so that $\mu_{\mathbf{p}}(\Sigma_1)=1$ and
	\begin{equation}\label{r94}
	\mu_{\alpha(\mathbf{i})}(\mathcal{B})=0 \quad\text{for all }
	\mathbf{i}\in\Sigma_1.
	\end{equation}
	Recall $\widetilde{\Pi}_\mathcal{H}$ is the natural projection from $\Sigma$ to $\Lambda_{\mathcal{H}}$. Observe that by definition
	\begin{equation}\label{r98}
	\mathbf{i}\in\Sigma'
	\;\Longrightarrow\;
	\Sigma'\cap \Big( \widetilde{\Pi}_\mathcal{H}^{-1}
	\widetilde{\Pi}_\mathcal{H}(\mathbf{i}) \Big)=\beta(\mathbf{i}).
	\end{equation}
	Since $\alpha(\ii)\subset \widetilde{\Pi}_\mathcal{H}^{-1}
	\widetilde{\Pi}_\mathcal{H}(\mathbf{i})$, we get from \eqref{r98} that
	$\mathbf{i}\in\Sigma'
	\Longrightarrow
	\Sigma'\cap\alpha(\mathbf{i}) \subset \beta(\mathbf{i})$. Equivalently,
	\begin{equation*}\label{r97}
	\mathbf{i}\in \Sigma'
	\Longrightarrow
	\alpha(\mathbf{i}) \subset \beta(\mathbf{i})\cup \mathcal{B}.
	\end{equation*}
	By definition
	\begin{equation*}\label{r87}
	[j_1, \dots ,j_n]\cap \beta(\mathbf{i})= \emptyset \;\mbox{ iff }\;
	(\jh_1, \dots ,\jh_n)\ne
	(\ih_1, \dots ,\ih_n).
	\end{equation*}
	That is for $\ii\in\Sigma_1$ whenever
	$(\jh_1, \dots ,\jh_n)\ne
	(\ih_1, \dots ,\ih_n)$
	then $ [j_1, \dots ,j_n]\cap \alpha(\mathbf{i}) \subset \mathcal{B}$. So,  \eqref{r94} implies that \eqref{r93} holds.
	
	To obtain \eqref{r77} from \eqref{r93}, we write $\beta(\ii)^c$ as a countable union
	$$
	\beta(\ii)^c=
	\bigcup\limits_{\ell =0}^\infty
	\left\{\jj\in\Sigma:|\iih\wedge\jjh|=\ell \right\}
	=
	\bigcup\limits_{\ell =0}^\infty
	\bigcup\limits_{ \jh_r=\ih_r, r \leq \ell  \atop
		j_{\ell +1}\ne i_{\ell +1}}
	\left[j_1, \dots j_{\ell +1}\right].
	$$
	By   \eqref{r93} the measure of each cylinder of the right hand side is
	$$
	\mu_{\alpha(\ii)}(\left[j_1, \dots j_{\ell +1}\right])=0 \,\mbox{ if }\,
	\left[\jh_1, \dots \jh_{\ell +1}\right]
	\ne\left[\ih_1, \dots ,\ih_{\ell +1}\right],\quad
	\ii\in\Sigma_1.
	$$
\end{proof}

\subsection{Proof of Proposition~\ref{r69}}

If the columns independently satisfy ROSC, then the proof of \cite[Corollary~2.8]{BARANYAnti201788} can be repeated in this setting, therefore we omit it. In the remainder we assume the shifted TGL carpet $\Lambda$ satisfies transversality and \eqref{cond:main}:
\begin{equation*}
\dfrac{\log \langle\mathbf{a}\rangle_{\mathbf{p}} }{\log \langle\mathbf{b}\rangle_{\mathbf{p}} }>1+\dfrac{\log \langle\mathbf{N}\rangle_{\mathbf{q}} }{-\log \langle\mathbf{q}\rangle_{\mathbf{q}}}.
\end{equation*}
Throughout this proof we fix $\delta>0$ small enough such that
\begin{equation}\label{a13}
1+\delta
+
\frac{(1+\delta)\log \langle\mathbf{N}\rangle_{\mathbf{q}}}
{\delta\log \langle\mathbf{b}\rangle_{\mathbf{p}}
	-
	\log\langle\mathbf{q}\rangle_{\mathbf{q}}
}
<
(1-\delta)
\frac{\log \langle\mathbf{a}\rangle_{\mathbf{p}}}
{\log\langle\mathbf{b}\rangle_{\mathbf{p}}}\,.
\end{equation}
This can be achieved since the expression is continuous in $\delta$ and we assume \eqref{cond:main}. The reason that we require this is that for such a  $\delta$ and
\begin{equation}\label{r44}
u:= (1-\delta)\frac{\log \langle\mathbf{a} \rangle_{\mathbf{p}}}{\log \langle\mathbf{b} \rangle_{\mathbf{p}}}-(1+\delta),
\end{equation}
the inequality in \eqref{a13} is equivalent to
\begin{align}\label{r59}
&\langle\mathbf{N}\rangle^{(1+\delta)}\cdot \langle\mathbf{q}\rangle^{u}  \cdot \langle\mathbf{b}\rangle^{-\delta u}<1.
\end{align}
At the very end of this proof we will need this.
The importance of the $u$ defined above comes from the fact that for an arbitrary $\ell $ and $k=u \cdot \ell$,

\begin{equation}\label{r54}
\langle\mathbf{b} \rangle^k_{\mathbf{p}}
=
\frac{\langle\mathbf{a} \rangle^{(1-\delta)\ell }_{\mathbf{p}}}
{\langle\mathbf{b} \rangle^{(1+\delta)\ell }_{\mathbf{p}}}.
\end{equation}
Recall $\alpha(\mathbf{i})=\Pi^{-1}\Pi(\mathbf{i}),\, \beta(\mathbf{i})=\Phi^{-1}\Phi(\mathbf{i})$, that $\widetilde{\Pi}_\mathcal{H}$ is the natural projection from $\Sigma$ to $\Lambda_{\mathcal{H}}$ and that in \eqref{r74} we define $\mathrm{Bad}_{\delta,n}(\mathbf{c})$ for a $\mathbf{c}=(c_1, \dots ,c_N)$ with $\langle c\rangle_{\mathbf{p}}\ne 1$. 

\subsubsection*{Further notation}

Recall that Hochman's Exponential Separation Condition implies that for the self-similar measure  $\nu_\mathbf{q}$ on $\Lambda_{\mathcal{H}}$ we have $\dim_{\rm H} \nu_{\mathbf{q}}=\log \langle\mathbf{q}\rangle_{\mathbf{q}} / \log \langle\mathbf{b}\rangle_{\mathbf{p}}$. Feng and Hu~\cite{FengHu09} proved that $\nu_\mathbf{q}$ is exact dimensional. That is for $K_1$ defined in \eqref{a33} and
\begin{multline}\label{r51}
S_{n_0}:=
\Big\{
\ii\in\Sigma: \forall n \geq n_0, \\
\nu_{\mathbf{q}}\left(
\left(\widetilde{\Pi}_{\mathcal{H}}(\ii)-3K_1
\langle\mathbf{b}\rangle^n,
\widetilde{\Pi}_{\mathcal{H}}(\ii)+3K_1\langle\mathbf{b}\rangle^n
\right)\right)\in\left(\langle\mathbf{q}\rangle^{n}
\langle\mathbf{b}\rangle^{\delta n},
\langle\mathbf{q}\rangle^{n}
\langle\mathbf{b}\rangle^{-\delta n}
\right)
\Big\}
\end{multline}
we have
\begin{equation}\label{r50}
\mu_{\mathbf{p}}\left(\bigcup\limits_{n_0=1}^{\infty }
S_{n_0}
\right)=
\lim\limits_{n_0\to\infty}
\mu_{\mathbf{p}}\left(S_{n_0}\right)=1.
\end{equation}

We define the set of symbols which are "good" from level $m$ on:
\begin{equation*}\label{r62}
\mathrm{Good}_m:
=\bigcap\limits_{n \geq m}\big(
\mathrm{Bad}_{\delta,n}(\mathbf{a})
\cup
\mathrm{Bad}_{\delta,n}(\mathbf{b})
\cup
\mathrm{Bad}_{\delta,n}(\mathbf{N})
\cup
\mathrm{Bad}_{\delta,n}(\mathbf{q})
\big)^c .
\end{equation*}
Note that it follows from  Lemma \ref{lemma:BadisSmall} that for
\begin{equation}\label{r65}
\mathrm{Good}:=
\bigcup\limits_{m=1}^{\infty }\mathrm{Good}_{m}, \,\text{ we have }\,   \mu_{\mathbf{p}}(\mathrm{Good})=1.
\end{equation}

To measure vertical distance and neighborhood on $\Lambda$ we define
\begin{equation*}\label{a39}
\mathrm{dist}_y((x_0,y_0),(x,y)) := \begin{cases}
|y-y_0|, &\text{if } x=x_0; \\ \infty. &\text{otherwise}, \end{cases}
\end{equation*}
For every $m \geq 1$
the function $L_m: \mathrm{Good}\to [0,1]$  is defined  as follows: if there exists no $\jj\in \mathrm{Good}_m$ with $j_1\ne i_1$ and $\Phi(\ii)=\Phi(\jj)$ then $L_m(\ii):=1$. Otherwise we define
\begin{equation*}
L_m(\ii):= \inf \{\mathrm{dist}_y(\Pi(\ii),\Pi(\jj)):\; \jj\in \mathrm{Good}_m\cap \beta(\ii)
 \text{ such that } j_1\neq i_1  \}.
\end{equation*}
Let
\begin{align*}\label{a50}
V_{\ell}^m &:= \{\ii\in\mathrm{Good}:\; L_m(\ii)<\langle\mathbf{a}\rangle^\ell\} \\
&=
\left\{\ii\in\mathrm{Good}:\;
\exists \jj\in \beta(\ii)\cap[i_1]^c \cap \mathrm{Good}_m,\,
\mathrm{dist}_y\left(\Pi(\ii),\Pi(\jj)\right)
<\langle\mathbf{a}\rangle^\ell
\right\}.
\end{align*}
Also define
\begin{equation}\label{r75}
  \mathcal{B}_2^m:=\!
  \{\ii\in\mathrm{Good}:\, L_m(\ii)=0\}
  =\left\{
  \ii\in\mathrm{Good}:\alpha(\ii)\cap\beta(\ii)\cap[i_1]^c
  \cap \mathrm{Good}_m
  \ne \emptyset
  \right\}.
\end{equation}
Clearly, $ \mathcal{B}_2^m \subset  \mathcal{B}_2^{m+1}$ since $\mathcal{B}_2^m=\cap_{\ell\geq m} V_{\ell}^m$. The key lemma states the following.

\begin{lemma}\label{r57}
For arbitrary $m \geq 1$ we have $\mu(\mathcal{B}_2^m)=0$.
\end{lemma}

\begin{proof}[Proof of Proposition \ref{r69} assuming Lemma \ref{r57}]
Let
$$
\mathcal{B}_2:=\bigcup\limits_{m=1}^{\infty }\mathcal{B}_2^m
=
\left\{
\ii\in\mathrm{Good}:\alpha(\ii)\cap\beta(\ii)\cap[i_1]^c
\cap \mathrm{Good}
\ne \emptyset 
\right\}.
$$
  By Lemma \ref{r57}, $\mu(\mathcal{B}_2)=0$. That is,
  if $\ii\in\Sigma_2:= \widetilde{\mathrm{Good}}\cap \mathcal{B}_2^c$
  then on the one hand $\mu_{\alpha(\ii)}(\mathrm{Good}^c)=0$,
  on the other hand
  $\alpha(\ii)\cap\beta(\ii)\cap[i_1]^c
  \cap \mathrm{Good}= \emptyset $. This implies that
  $\mu_{\alpha(\ii)}
  \left(
  \alpha(\ii)\cap\beta(\ii)\cap[i_1]^c
  \right)=0
  $, which completes the proof of Proposition~\ref{r69}.
\end{proof}

It remains to show Lemma~\ref{r57}. The
method of the proof was inspired by
\cite[Lemma~4.7]{barany_rams_simon_triang_2017}, however there are significant differences. On the one hand, in \cite{barany_rams_simon_triang_2017} the measure corresponding to $\nu_{\mathbf q}$ is absolutely continuous with $L^q$ density and in \cite{barany_rams_simon_triang_2017} the diagonal part of all the linear parts of all the mappings are identical. These differences required a much more subtle argument in this paper.

\begin{proof}[Proof of Lemma~\ref{r57}]
Recall that we fixed an $m$. Let $\ell  \geq m$. All sets and numbers from now on in this proof can be dependent of $m$ but $m$ is fixed so we omit it from notation.

We cover $V_{\ell}^m$ by the union of the $\Pi^{-1}$ pre-images of the parallelograms like the blue one ($R_{\iiv,\jjv}$) on the right hand side of Figure \ref{fig:DimHOverlap}. These are parallelograms  slightly bigger than the intersection of $R_{\iiv}$ and the $\langle\mathbf{a}\rangle^\ell$ neighborhood of $R_{\jjv}$ for $\iiv,\jjv\in\Sigma_\ell $ with $\Phi(\iiv)=\Phi(\jjv)$.

To control the size of $\ell$th level parallelograms and the number of parallelograms in any given $\ell$-th level column set
\begin{equation*}\label{r73}
\mathrm{Bad}_{\delta,\ell}^1 :=\mathrm{Bad}_{\delta,\ell}(\mathbf{a}) \,\cup\, \mathrm{Bad}_{\delta,\ell}(\mathbf{b}) \,\cup\, \mathrm{Bad}_{\delta,\ell}(\mathbf{N}) \;\text{ and }\;
\mathrm{Bad}_{\delta,\ell}^{1,\ast} := \{\ii|\ell:\; \ii\in\mathrm{Bad}_{\delta,\ell}^1\},
\end{equation*}
where $\mathrm{Bad}_{\delta,n}(\mathbf{c})$ was defined in \eqref{r74}.
Observe that $\mathrm{Bad}_{\delta,\ell}^1$ is the union of complete $\ell $-cylinders. That is
\begin{equation*}\label{r71}
  \mathrm{Bad}_{\delta,\ell}^{1} =
\bigcup\limits_{\pmb{\omega}\in\mathrm{Bad}_{\delta,\ell}^{1,\ast}}
[\pmb{\omega}].
\end{equation*}
The level $\ell $-cylinders of the symbolic spaces excluding these bad cylinders  are:
\begin{equation*}
{\mathrm{Good}}_\ell^\ast:= \left\{1, \dots ,N\right\}^\ell\setminus\mathrm{Bad}_{\delta,\ell}^{1,\ast} \,\text{ and }\,
{\mathrm{Good}}_{\ell,\iiv}^\ast:=\{\jjv\in{\mathrm{Good}}_\ell^\ast:\, j_1\neq i_1, \Phi(\jjv)=\Phi(\iiv)\}\setminus\mathrm{Bad}_{\delta,\ell}^{1,\ast}.
\end{equation*}

For $H\subset [0,1]^2$ let
\begin{equation*}
U_y(H,r) := \bigcup_{(x_0,y_0)\in H} \{(x,y):\; x=x_0 \text{ and } |y-y_0|<r\}.
\end{equation*}
Choose $\iiv\in{\mathrm{Good}}_\ell^\ast,\, \jjv\in{\mathrm{Good}}_{\ell,\iiv}^\ast$ and define
\begin{align*}
I_{\iiv,\jjv} &:= \proj_x(R_{\iiv} \cap U_y(R_{\jjv}, \langle\mathbf{a}\rangle^{(1-\delta)\ell}), \\
R_{\iiv,\jjv} &:= (I_{\iiv,\jjv}\times [0,1]) \cap R_{\iiv}\,,\\
\widetilde{R}_{\iiv,\jjv} &:= ([\iiv] \cap \Pi^{-1}(R_{\iiv,\jjv})).
\end{align*}
$R_{\iiv,\jjv}$ consists of those elements of $R_{\iiv}$ which are physically "too close" to $R_{\jjv}$,
see Figure \ref{fig:DimHOverlap}. As a result we get a cover of $V_{\ell}^m\,$:
\begin{equation}\label{eq:coverofV}
V_{\ell}^m \subset \mathrm{Bad}_{\delta,\ell}^1 \,\cup\, \bigcup_{\iiv\in {\mathrm{Good}}_\ell^\ast} \,\bigcup_{\jjv\in {\mathrm{Good}}_{\ell,\iiv}^\ast} \widetilde{R}_{\iiv,\jjv}.
\end{equation}
Namely, if $\ii\in V_{\ell}^m$ then
either $\ii\in\mathrm{Bad}_{\delta,\ell}^1$ or $\iiv:=\ii|\ell\in\mathrm{Good}_\ell^\ast$. In the second case,
there is a
$\jj\in \beta(\ii)\cap [i_1]^c\cap\mathrm{Good}_m $ with
$\mathrm{dist}_y(\Pi(\ii),\Pi(\jj))<\langle\mathbf{a}\rangle^\ell $.
Hence, $\jjv:=\jj|\ell \in {\mathrm{Good}}_{\ell,\iiv}^\ast$. As a result, with these notations, we have $\ii\in \widetilde{R}_{\iiv,\jjv}$.

\begin{center}
	\begin{figure}[h]
		\centering
		\includegraphics[width=.9\linewidth]{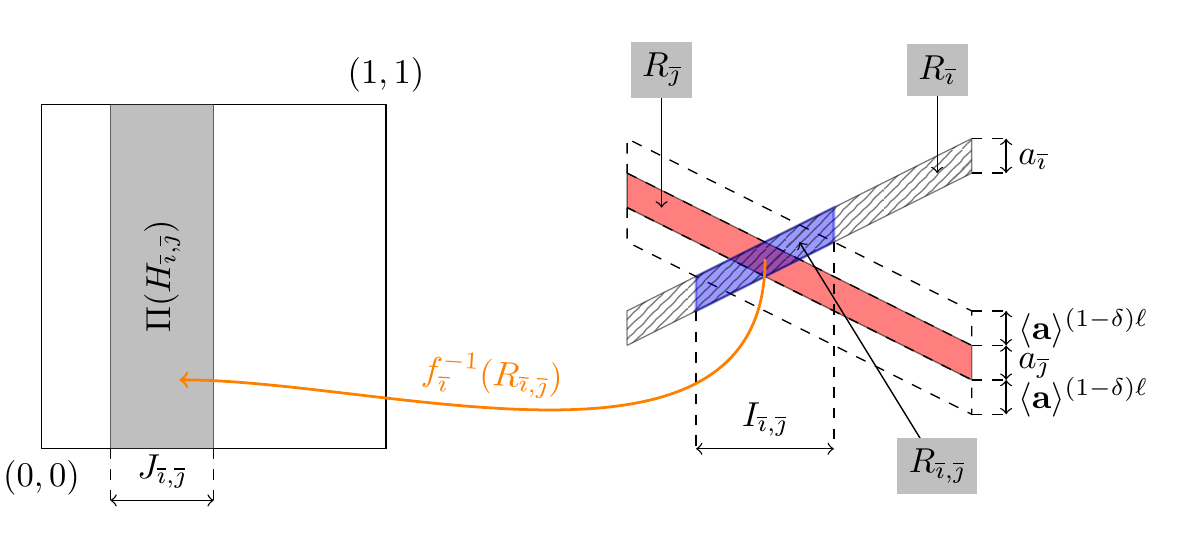}
		\caption{Intersecting parallelograms $R_{\iiv}$ and $R_{\jjv}$ in the proof of Lemma~\ref{r57}.}\label{fig:DimHOverlap}
	\end{figure}
\end{center}

If $I_{\iiv,\jjv}\neq \emptyset$, then there exists a non-empty interval $J_{\iiv,\jjv}$ such that
\begin{equation*}
f_{\iiv}^{-1}(R_{\iiv,\jjv}) = J_{\iiv,\jjv}\times [0,1].
\end{equation*}
With symbolic notation $H_{\iiv,\jjv}:= \Pi^{-1}(f_{\iiv}^{-1}(R_{\iiv,\jjv}))$ we can represent $\widetilde{R}_{\iiv,\jjv}$ as the concatenation
\begin{equation}\label{r43}
\widetilde{R}_{\iiv,\jjv} = \iiv H_{\iiv,\jjv}.
\end{equation}
On the other hand, $ H_{\iiv,\jjv} \subset \widetilde{\Pi}_{\mathcal{H}}^{-1}\left(J_{\iiv,\jjv}\right)$. Hence,
\begin{equation}\label{r48}
 \mu_{\mathbf{p}}\left(H_{\iiv,\jjv}\right)  \leq
 \mu_{\mathbf{p}}\left(\widetilde{\Pi} ^{-1}_{\mathcal{H}}(J_{\iiv,\jjv})\right)
 =\nu_{\mathbf{q}}\left(J_{\iiv,\jjv}\right).
\end{equation}
To continue we give an upper bound for $\nu_{\mathbf{q}}\left(J_{\iiv,\jjv}\right)$.
\begin{claim}\label{r46}Let
$\iiv\in {\mathrm{Good}}_\ell^\ast$ and
$\jjv\in {\mathrm{Good}}_{\ell,\iiv}^\ast$ and let $k:=u \cdot \ell $, where $u$ was defined in \eqref{r44}.
  If $\widetilde{\Pi}_{\mathcal{H}}^{-1}(J_{\iiv,\jjv})\cap S_k\ne \emptyset $ (recall \eqref{r51} for the definition of $S_k$), then
\begin{equation*}\label{r45}
  \nu_\mathbf{q}\left(J_{\iiv,\jjv}\right) \leq
  \langle\mathbf{q}\rangle^k \cdot  \langle\mathbf{b}\rangle^{-k\delta}.
\end{equation*}
\end{claim}
\begin{proof}[Proof of Claim~\ref{r46}]
If $I_{\iiv,\jjv}\neq \emptyset$ then transversality (recall Definition \ref{def:separations}) implies that
\begin{equation*}
|I_{\iiv,\jjv}| < 3K_1\cdot \langle\mathbf{a}\rangle^{(1-\delta)\ell}.
\end{equation*}
This is the very important point where use that neither $\iiv$ nor $\jjv$ are contained in $\mathrm{Bad}_{\delta,\ell }^{1,\ast}$.
Furthermore, $f_{\iiv}^{-1}$ expands along the $x$ axis by a factor between $\langle\mathbf{b}\rangle^{-(1-\delta)\ell}$ and $\langle\mathbf{b}\rangle^{-(1+\delta)\ell}$, hence
\begin{equation}\label{r53}
|J_{\iiv,\jjv}|< 3K_1 \left(\frac{\langle\mathbf{a}\rangle^{1-\delta}}{\langle\mathbf{b}\rangle^{1+\delta}}\right)^{\ell}.
\end{equation}

If we set $k$ as in Claim \ref{r46} then as we mentioned in \eqref{r54}
the right hand side of \eqref{r53} is less than $3K_1 \cdot \langle\mathbf{b}\rangle^k\,$:
\begin{equation*}\label{r52}
|J_{\iiv,\jjv}|<  3K_1 \cdot \langle\mathbf{b}\rangle^k.
\end{equation*}
Now assume that $\widetilde{\Pi}_{\mathcal{H}}^{-1}(J_{\iiv,\jjv})\cap S_k\ne \emptyset $.
Pick an arbitrary $\pmb{\omega}\in \widetilde{\Pi}_{\mathcal{H}}^{-1}(J_{\iiv,\jjv})\cap S_k$.
  Then
  $$ J_{\iiv,\jjv} \subset
 \left( \widetilde{\Pi}_{\mathcal{H}}(\pmb{\omega})-3K_1
 \langle\mathbf{b}\rangle^k,
  \widetilde{\Pi}_{\mathcal{H}}(\pmb{\omega})+3K_1
  \langle\mathbf{b}\rangle^k
 \right) .$$ Using that $\pmb{\omega}\in S_k$,  we get that
 $\nu_{\mathbf{q}} \big( \widetilde{\Pi}_{\mathcal{H}}(\pmb{\omega})-3K_1\langle\mathbf{b}\rangle^k,
  \widetilde{\Pi}_{\mathcal{H}}(\pmb{\omega})+3K_1\langle\mathbf{b}\rangle^k
 \big) \!\leq\!  \langle\mathbf{q}\rangle^k \cdot  \langle\mathbf{b}\rangle^{-k\delta}.$

\end{proof}
Now we conclude the proof of Lemma~\ref{r57}. From the cover \eqref{eq:coverofV} of $V_{\ell}^m$ together with \eqref{r43} we obtain that for $\ell  \geq m$
\begin{equation}\label{eq:measureofV}
\mu_{\mathbf{p}}(V_{\ell}^m) \leq \mu_{\mathbf{p}} (\mathrm{Bad}_{\delta,\ell}^1) + \sum_{\iiv\in {\mathrm{Good}}_\ell^\ast} \mu_{\mathbf{p}}([\iiv]) \mu_{\mathbf{p}}\Bigg(\bigcup\limits_{\jjv\in {\mathrm{Good}}_{\ell,\iiv}^\ast}H_{\iiv,\jjv}\Bigg).
\end{equation}
To further bound \eqref{eq:measureofV},
first observe that
\begin{equation*}\label{r49}
  \#{\mathrm{Good}}_{\ell,\iiv}^\ast \leq \langle\mathbf{N}\rangle^{(1+\delta)\ell} \mbox{ whenever }
  \iiv~\in~{\mathrm{Good}}_\ell^\ast.
\end{equation*}
Moreover, using \eqref{r48} and Claim~\ref{r46}, for an arbitrary $\iiv\in {\mathrm{Good}}_\ell^\ast$ we have

\begin{align*}
   \mu_{\mathbf{p}}\Bigg(\bigcup\limits_{\jjv\in {\mathrm{Good}}_{\ell,\iiv}^\ast}H_{\iiv,\jjv}\Bigg) &\leq
  \mu_{\mathbf{p}}\left(S_{k}^{c}\right)
 +
 \sum\limits_{\jjv\in {\mathrm{Good}}_{\ell,\iiv}^\ast \atop
 \widetilde{\Pi}_{\mathcal{H}}^{-1}(J_{\iiv,\jjv})\cap S_k\ne \emptyset
 }^{}   \mu_{\mathbf{p}}\left(H_{\iiv,\jjv}\right)
  \leq \mu_{\mathbf{p}}\left(S_{k}^{c}\right)
 +
 \sum\limits_{\jjv\in {\mathrm{Good}}_{\ell,\iiv}^\ast \atop
 \widetilde{\Pi}_{\mathcal{H}}^{-1}(J_{\iiv,\jjv})\cap S_k\ne \emptyset
 }^{}  \nu_{\mathbf{q}}(J_{\iiv,\jjv})\\
  &\leq  \mu_{\mathbf{p}}\left(S_{k}^{c}\right)
 +
 \sum\limits_{\jjv\in {\mathrm{Good}}_{\ell,\iiv}^\ast \atop
 \widetilde{\Pi}_{\mathcal{H}}^{-1}(J_{\iiv,\jjv})\cap S_k\ne \emptyset
 }^{} \langle\mathbf{q}\rangle^k \cdot  \langle\mathbf{b}\rangle^{-k\delta}\\
 &\leq  \mu_{\mathbf{p}}\left(S_{k}^{c}\right) +
 \left(\langle\mathbf{N}\rangle^{(1+\delta)} \cdot
 \langle\mathbf{q}\rangle^u \cdot  \langle\mathbf{b}\rangle^{-u\delta}\right)^{\ell }.
\end{align*}
Pluggung this back into \eqref{eq:measureofV}, we deduce from Lemma \ref{lemma:BadisSmall}, \eqref{r50} and \eqref{r59} that for every $m$,
$$
\lim\limits_{\ell \to\infty}
 \mu_{\mathbf{p}}(V_{\ell}^m)=0.
$$
By \eqref{r75}, this implies that $ \mu_{\mathbf{p}}\left(\mathcal{B}_2^m\right)=0.$

\end{proof}

\section{Proof of results for box dimension}\label{sec:boxdim}

We begin the section by briefly commenting on how the upper bound for $\overline{\dim}_{\rm B} \Lambda$, recall Theorem~\ref{thm:mainBox}, follows directly from the work of Fraser~\cite{Fraser12Boxlike} and then prove Theorem~\ref{thm:BoxwOverlaps} in Subsection~\ref{subsec:ProofBoxDim}.

For $\delta>0$ and a bounded set $F\subset\R^2$ let $N_\delta(F)$ denote the minimal number of closed axes parallel rectangles for which the vertical sides are not shorter than the horizontal sides but the vertical sides are not longer than $(K_0+1)$-times the horizontal sides, where $K_0$ was defined in Lemma~\ref{lemma:d/b_bounded}. Then
\begin{equation*}
\underline{\dim}_{\rm B} F = \liminf_{\delta\to 0} \frac{\log N_\delta(F)}{-\log \delta} \;\text{ and }\; \overline{\dim}_{\rm B} F = \limsup_{\delta\to 0} \frac{\log N_\delta(F)}{-\log \delta}.
\end{equation*}
In particular, it is enough to consider $\delta\to 0$ through the sequence $\delta_k=c^k$ for some $0<c<1$, see~\cite[Section 3.1]{FalconerBook}.

For $t\geq 0$ and any finite length word $\iiv\in\Sigma^\ast$, Fraser defined the modified singular value function $\psi^t$, which in our context is
\begin{equation*}
\psi^t(f_{\iiv}):= b_{\iiv}^{s_{\mathcal{H}}}\cdot a_{\iiv}^{t-s_{\mathcal{H}}},
\end{equation*}
where $s_{\mathcal{H}}= \dim_{\rm B} \Lambda_{\mathcal{H}}$. He showed that the unique solution $s$ of the equation
\begin{equation*}
\lim_{n\to\infty}\Big(\sum_{\iiv\in\Sigma_n}\psi^s(f_{\iiv})\Big)^{1/n}=1
\end{equation*}
is an upper bound for $\overline{\dim}_{\rm B}\Lambda$ and equals $\dim_{\rm B}\Lambda$ if $\Lambda$ satisfies the ROSC. In our context this equation simply becomes \eqref{a74}: $\sum_{i=1}^{N} b_{i}^{s_{\mathcal{H}}}a_{i}^{s-s_{\mathcal{H}}}=1$. The slight modification of the GL brother $\widetilde{\Lambda}$ ensures that the solution of \eqref{a74} for $\Lambda$ and $\widetilde{\Lambda}$ is the same.

For any TGL carpet $\Lambda$, it follows from Lemma~\ref{lemma:d/b_bounded} that the longer side of any parallelogram $R_{\iiv}$ is at most $(K_0+1)\cdot b_{\iiv}$. This implies that there exists a constant $C$ (independent of $\delta$) such that $N_{\delta}(\Lambda)\leq C\cdot N_{\delta}(\widetilde{ \Lambda})$. Hence, $\overline{\dim}_{\rm B}(\Lambda)\leq\overline{\dim}_{\rm B}(\widetilde{ \Lambda})\leq s$. Furthermore, when the ROSC is assumed, it is clear that the reversed inequalities also hold. This implies $\underline{\dim}_{\rm B}(\Lambda)\geq\underline{\dim}_{\rm B}(\widetilde{ \Lambda})= s$. This proves Theorem~\ref{thm:mainBox}.

In the presence of overlaps, one must be more careful when counting the intersections. The next subsection shows how to select a diagonally homogeneous subsystem from a higher iterate of $\mathcal{F}$.

\subsection{Diagonally homogeneous subsystems}\label{subsec:BoxPrelim}

Recall for a $\mathbf p\in\mathcal{P}_0$
\begin{align*}
h_{\mathbf p} &= -\sum_{i=1}^N p_i \log p_i = -\log \langle\mathbf{p}\rangle_\mathbf{p}, \\
\chi_{\mathbf p}^1 &= -\sum_{i=1}^N p_i \log b_i = -\log \langle \mathbf{b} \rangle_{\mathbf{p}} \;\text{ and }\; \chi_{\mathbf p}^2 = -\sum_{i=1}^N p_i \log a_i = -\log \langle \mathbf{a} \rangle_{\mathbf{p}}.
\end{align*}
The following is a Ledrappier--Young like formula for the solution $s$ of~\eqref{a74}. It generalizes the formula in Corollary~\ref{cor:dimH_BMcarpet} for the diagonally homogeneous case. A similar result for Bedford-McMullen like systems in arbitrary dimension was proved in \cite[Theorem 2.15]{FengHu09}.
\begin{claim}\label{claim:LedrappierforBox}
	For $\widetilde{\mathbf p}:=(\widetilde{p}_1,\ldots,\widetilde{p}_N)$ defined by $\widetilde{p}_i=b_i^{s_{\mathcal{H}}}a_i^{s-s_{\mathcal{H}}}$, recall \eqref{def:pandqForBox}, we have
	\begin{equation*}
	s = \frac{h_{\widetilde{\mathbf p}}}{\chi_{\widetilde{\mathbf p}}^2} + \left(1-\frac{\chi_{\widetilde{\mathbf p}}^1}{\chi_{\widetilde{\mathbf p}}^2}\right)s_{\mathcal{H}},
	\end{equation*}
	where $s_{\mathcal{H}}=\dim_{\rm B}\Lambda_{\mathcal{H}} $.
\end{claim}

\begin{proof}
	Immediately follows from the observation that $h_{\widetilde{\mathbf p}} = s_{\mathcal{H}}\chi_{\widetilde{\mathbf p}}^1+(s-s_{\mathcal{H}})\chi_{\widetilde{\mathbf p}}^2$.
\end{proof}

The following line of thought  is an adaptation of \cite[Section~6]{fraser_shmerkin_2016} in order to extract from an arbitrary shifted TGL IFS $\mathcal{F}=\{f_i\}_{i=1}^N$ with $M$ columns a subsystem of a high enough iterate of $\mathcal{F}^k$, which has some nice properties required to prove the theorem.

Let $\mathcal{F}^k:=\{f_{i_1\ldots i_k}: i_1,\ldots,i_k\in\Sigma_{k}\}$. The first step is to pass from $\mathcal{F}$ to a diagonally homogeneous subsystem of $\mathcal{F}^k$. Analogous arguments appear for example in \cite[Lemma~5.2]{barany2016dimension}, \cite[Lemma~6.2]{fraser_shmerkin_2016} or \cite[Lemma~4.9]{pardo-simon}.

\begin{definition}
	A subsystem $\mathcal{G}^{(k)}\subset\mathcal{F}^k$ is called a diagonally homogeneous subsystem if there exists $a^{(k)}$ and $b^{(k)}$ for which
	\begin{equation*}
	g^{(k)}_i(\underline{x})= \begin{pmatrix}
	b^{(k)} & 0 \\ d_i & a^{(k)}
	\end{pmatrix}\underline{x} + t_i, \;\;\text{ for every } g^{(k)}_i\in\mathcal{G}^{(k)}.
	\end{equation*}
\end{definition}

Fix an arbitrary vector $\mathbf{v}=(v_1,\ldots,v_N)$, where $v_i\in\N$. Let $V_{\jh}:=\sum_{i\in\mathcal{I}_{\jh}}v_i,\; \mathbf{V}:=(V_1,\ldots,V_M),\; V=V_1+\ldots+V_M$ and define
\begin{equation}\label{def:M_v}
\mathcal{M}_{\mathbf{v}}:=\{(i_1,\ldots,i_V)\in\Sigma_{V}:\; \#\{\ell\leq V:\, i_{\ell}=r\}=v_r \text{ for every } r=1,\ldots,N \}.
\end{equation}

\begin{claim}\label{claim:subsys}
	The subsystem $\mathcal{G}_{\mathbf{v}}=\{f_{i_1\ldots i_V}: (i_1,\ldots,i_V)\in \mathcal{M}_{\mathbf{v}}\}\subset \mathcal{F}^V$ with $M_{\mathbf{v}}$  columns
	\begin{enumerate}[(i)]
		\item is a diagonally homogeneous subsystem with $a^{(\mathbf{v})}=\prod_{r=1}^N a_r^{v_r} \;\text{ and }\; b^{(\mathbf{v})}=\prod_{r=1}^N b_r^{v_r},$
		\item has uniform vertical fibres with $\prod_{\ih=1}^M \frac{V_{\ih}!}{\prod_{r\in\mathcal{I}_{\ih}}v_r!}$ maps in each column and
		\item for the probability vectors $\overline{\mathbf{v}}=\mathbf{v}/V$ and $\overline{\mathbf{V}}=\mathbf{V}/V$
		\begin{align*}
		-N\log(V+1) + V \cdot h_{\overline{\mathbf v}} &\leq \log \# \mathcal{M}_{\mathbf{v}} \leq  V\cdot h_{\overline{\mathbf v}}, \\
		-N\log(V+1) + V\cdot h_{\overline{\mathbf V}} &\leq \log M_{\mathbf{v}} \leq  V\cdot h_{\overline{\mathbf V}}.
		\end{align*}
	\end{enumerate}
\end{claim}

\begin{proof}
	Parts (i) and (ii) are immediate. Part (iii) follows directly from \cite[Lemma~2.1.8]{DemboZeitouniLDP}.
	
\end{proof}

\begin{lemma}\label{lemma:DiagHomoSubsys}
	Let $\mathcal{F}=\{f_i\}_{i=1}^N$ be a shifted TGL IFS with $M$ columns. For every $k$ choose $\mathbf{v}_k=(v_{1,k},\ldots,v_{N,k})$ such that
	\begin{equation}\label{def:v_i,k}
	v_{i,k}=\lfloor k\widetilde{p}_i\rfloor \;\text{ for every } i=1,\ldots,N,
	\end{equation}
	where $\widetilde{p}_i$ was defined in \eqref{def:pandqForBox}.
	Let $V^{(k)}_{\jh}:=\sum_{i\in\mathcal{I}_{\jh}}v_{i,k},\; V^{(k)}=\sum_{\ih=1}^M V^{(k)}_{\ih}$ and define the subsystem
	\begin{equation*}
	\mathcal{G}^{(k)}=\mathcal{G}_{\mathbf{v}_k}=\{f_{i_1\ldots i_{V^{(k)}}}: (i_1,\ldots,i_{V^{(k)}})\in \mathcal{M}_{\mathbf{v}_k}\}\subset \mathcal{F}^{V^{(k)}},
	\end{equation*}
	where $\mathcal{M}_{\mathbf{v}_k}$ is defined by \eqref{def:M_v}. Then $\mathcal{G}^{(k)}$ satisfies the assertions of Claim \ref{claim:subsys} with $\mathbf{v}_k$. For brevity we write $a^{(\mathbf{v}_k)}=a^{(k)}$ and $b^{(\mathbf{v}_k)}=b^{(k)}$. Let
	\begin{equation*}
	N^{(k)}=\#\mathcal{M}_{\mathbf{v}_k} \text{ and } M^{(k)}=\#\big(\Phi(\mathcal{M}_{\mathbf{v}_k})\big)
	\end{equation*}
	denote the number of maps and columns in $\mathcal{G}^{(k)}$.
	
	Moreover, $\lim\limits_{k\to\infty} s^{(k)} = s$, where $s^{(k)}$ is the solution of $N^{(k)}(b^{(k)})^{s_{\mathcal{H}}}(a^{(k)})^{s^{(k)}-s_{\mathcal{H}}}=1$, i.e.
	\begin{equation*}
	s^{(k)}= \frac{\log N^{(k)}}{-\log a^{(k)}} + \left(1-\frac{\log b^{(k)}}{\log a^{(k)}}\right)s_{\mathcal{H}}.
	\end{equation*}
\end{lemma}
\begin{remark}
	The box dimension of the attractor of the IFS $\mathcal{G}^{(k)}$ is NOT equal to $s^{(k)}$, because $s_{\mathcal{H}}$ is not the box dimension $s_{\mathcal{H}}^{(k)}$ of the attractor of the IFS generated by the columns of $\mathcal{G}^{(k)}$. The problem is that $s_{\mathcal{H}}^{(k)}\not\to s_{\mathcal{H}}$ as $k\to\infty$ (except when $\dim_{\mathrm H}\Lambda=~\dim_{\mathrm B}\Lambda$).
\end{remark}
\begin{proof}
	It follows from \eqref{def:v_i,k} that $k-N\leq V^{(k)}\leq k$ and
	\begin{equation*}
	k\log\langle \mathbf{a} \rangle_{\widetilde{\mathbf p}} - \sum_{i=1}^N\log  a_i \leq \log a^{(k)} \leq k\log\langle \mathbf{a} \rangle_{\widetilde{\mathbf p}}.
	\end{equation*}
	Same holds for $\log b^{(k)}$. Furthermore, for the probability vector $\overline{\mathbf{v}}_k=\mathbf{v}_k/V^{(k)}$
	\begin{equation*}
	\widetilde{p}_i -\frac{1}{k}\leq \frac{v_{i,k}}{V^{(k)}}\leq \widetilde{p}_i + \frac{\widetilde{p}_i N}{k-N}.
	\end{equation*}
	Thus $\lim_{k\to\infty}h_{\overline{\mathbf v}_k}=h_{\widetilde{\mathbf p}}$. We can use Claim~\ref{claim:subsys} (iii) to bound $\log \# \mathcal{M}_{\mathbf{v}_k}$. Hence, putting together all the above we get
	\begin{equation*}
	\lim_{k\to\infty} s^{(k)} = \frac{h_{\widetilde{\mathbf p}}}{-\log\langle \mathbf{a} \rangle_{\widetilde{\mathbf p}}} + \left(1-\frac{\log\langle \mathbf{b} \rangle_{\widetilde{\mathbf p}}}{\log\langle \mathbf{a} \rangle_{\widetilde{\mathbf p}}}\right)s_{\mathcal{H}},
	\end{equation*}
	which is equal to $s$ due to Claim \ref{claim:LedrappierforBox}.
\end{proof}

If $\mathcal{G}^{(k)}$ already has non-overlapping columns, then the rest of the construction is not necessary. Otherwise, we can pass further to a subsystem $\mathcal{G}^{(k,\ell)}\subset \big(\mathcal{G}^{(k)}\big)^\ell$ by throwing away "not too many" columns of $\big(\mathcal{G}^{(k)}\big)^\ell$ in order to ensure that $\mathcal{G}^{(k,\ell)}$ has non-overlapping columns.

Projecting $\mathcal{G}^{(k)}$ to the $x$-axis gives a subsystem of $\mathcal{H}^{V^{(k)}}$
\begin{equation*}
\mathcal{G}_{\mathcal{H}}^{(k)}:= \{h_{\ih_1\ldots\ih_{V^{(k)}}}:\, \text{there exists } (i_1,\ldots,i_{V^{(k)}})\in \mathcal{M}_{\mathbf{v}_k} \text{ s. t. } \Phi(i_1\ldots i_{V^{(k)}})=\ih_1\ldots\ih_{V^{(k)}}\},
\end{equation*}
which has a total of $M^{(k)}$ maps, each with contracting ratio $b^{(k)}$.  Observe that $\mathcal{G}_{\mathcal{H}}^{(k)}$ also satisfies Hochman's Exponential Separation Condition, because this condition is assumed for $\mathcal{H}$ and this property passes on to any subsystem. Hence, the Hausdorff and box dimension of $\mathcal{G}_{\mathcal{H}}^{(k)}$ satisfies
\begin{equation}\label{eq:dimOfG^k_H}
s^{(k)}_{\mathcal{H}} = \frac{\log M^{(k)}}{-\log b^{(k)}}\,.
\end{equation}
It follows from the definition of box dimension that for every $\varepsilon>0$ there exists a subset of the columns of $\big(\mathcal{G}^{(k)}\big)^\ell$, which are non-overlapping and have cardinality
\begin{equation}\label{eq:BoundMkl}
M^{(k,\ell)}\geq C_{\varepsilon} \big((b^{(k)})^\ell\big)^{-(s^{(k)}_{\mathcal{H}}-\varepsilon)} \stackrel{\eqref{eq:dimOfG^k_H}}{=} C_{\varepsilon} \cdot  \big(M^{(k)}\big)^\ell \big(b^{(k)}\big)^{\ell\varepsilon}.
\end{equation}
This is the subsystem $\mathcal{G}^{(k,\ell)}$ which we will use in the proof of Theorem~\ref{thm:BoxwOverlaps} under condition~(i). When condition (ii) of Theorem~\ref{thm:BoxwOverlaps} is assumed we use $\mathcal{G}^{(k,\ell)}=\big(\mathcal{G}^{(k)}\big)^\ell$ since in this case non-overlapping columns are assumed for $\Lambda$. Next, we present our argument to count the number of intersections within a column when $\Lambda$ has non-overlapping columns.

\subsection{Counting intersections}\label{subsec:DiagHomoBoxDim}

Let $\mathcal{F}$ be an arbitrary TGL IFS and $\mathcal{G}^{(k,\ell)}= \big(\mathcal{G}^{(k)}\big)^\ell$ be the subsystem defined in the previous subsection. Then $\mathcal{G}^{(k,\ell)}$ is diagonally homogeneous with main diagonal $((b^{(k)})^\ell,(a^{(k)})^\ell)$, has uniform vertical fibres with $(N^{(k)}/M^{(k)})^\ell$ maps in each column and the columns are non-overlapping. For every $f_{\iiv}\in\mathcal{G}^{(k,\ell)},\, \iiv$ can be written
\begin{equation*}
\iiv=\iiv_1\iiv_2\ldots\iiv_{\ell},\, \text{ where } \iiv_j\in \mathcal{M}_{\mathbf{v}_k} \text{ for } j=1,\ldots,\ell.
\end{equation*}
Let $\Sigma^{(k,\ell)}:=\{\iiv:\, f_{\iiv}\in\mathcal{G}^{(k,\ell)}\}$ and for the rest of the subsection fix such an $\iiv\in\Sigma^{(k,\ell)}$. Let
\begin{equation*}
\Sigma_{\iiv}^{\sim} := \{\jjv=\jjv_1\ldots\jjv_{\ell}\in\Sigma^{(k,\ell)}:\; \Phi(\jjv)=\Phi(\iiv)  \text{ and } \jjv\neq\iiv\},
\end{equation*}
i.e. $\Sigma_{\iiv}^{\sim}$ collects those $\jjv$ which belong to the symbolic column of $\iiv$. Recall $\Lambda_{\iiv}=f_{\iiv}(\Lambda),\, R_{\iiv}=f_{\iiv}([0,1]^2)$. Let
\begin{equation*}
\widetilde{R}_{\iiv}:= \big(\proj_x(f_{\iiv}(\Lambda))\times[0,1]\big)\cap R_{\iiv} \;\text{ and }\; \delta^{(k)}_{\ell}:=(a^{(k)})^\ell.
\end{equation*}
Our aim is to give a uniform upper bound for $N_{\delta^{(k)}_{\ell}}\big(\widetilde{R}_{\iiv}\cap (\cup_{\jjv\in\Sigma_{\iiv}^{\sim}}\widetilde{R}_{\jjv})\big)$. Observe that for every $\jjv\in\Sigma_{\iiv}^{\sim}$
\begin{equation}\label{eq:infectedBoxes}
N_{\delta^{(k)}_{\ell}}\big(\widetilde{R}_{\iiv}\cap \widetilde{R}_{\jjv}\big) = N_{\delta^{(k)}_{\ell}}\big(\proj_x(\Lambda_{\iiv}\cap \Lambda_{\jjv})\big) = N_{\delta^{(k)}_{\ell}/(b^{(k)})^\ell}\big(h_{\iiv}^{-1}(\proj_x(\Lambda_{\iiv}\cap \Lambda_{\jjv}))\big).
\end{equation}

We state a result of Lalley \cite[Theorem~1]{Lalley88}, which gives the precise asymptotic of $N_\delta(\Lambda_{\mathcal{H}})$. A set $\{r_1,\ldots,r_M\}$ of positive reals is $\tau$-arithmetic, if $\tau>0$ is the greatest number such that each $r_i$ is an integer multiple of $\tau$, and non-arithmetic if no such $\tau$ exists. We use the notation $f(\delta)\sim g(\delta)$ to denote that $\lim_{\delta\to0}f(\delta)/g(\delta)=1$. Let $F$ be a self-similar set on $[0,1]$ with contracting ratios $\{r_1,\ldots,r_M\}$. Assume $F$ satisfies the strong OSC and let $\dim_{\mathrm H}F=\dim_{\mathrm B}F=t$, where $t$ is the solution of $\sum_{\ih=1}^M r_{\ih}^t=1$.
\begin{prop}\cite[Theorem~1]{Lalley88}\label{prop:exactN_delta}
If $\{\log r_1^{-1},\ldots,\log r_M^{-1}\}$ is a non-arithmetic set, then for some $K>0$
\begin{equation*}
N_\delta(F)\sim K \delta^{-t} \; \text{ as } \delta\to0.
\end{equation*}
On the other hand, if $\{\log r_1^{-1},\ldots,\log r_M^{-1}\}$ is $\tau$-arithmetic, then for the subsequence $\delta_n=e^{-n\tau}$ there exists a constant $K'>0$ such that
\begin{equation*}
N_{\delta_n}(F)\sim K'\delta_n^{-t} \; \text{ as } n\to\infty.
\end{equation*}
\end{prop}
\begin{remark}
The reason why we can not handle both types of overlaps simultaneously for the box dimension is that we are unaware of an analogous result in the case that SOSC is not assumed. This question could be of independent interest.
\end{remark}

We use the proposition for the self-similar set $\Lambda_{\mathcal{H}}$ with contracting ratios $(r_1,\ldots,r_M)$. If $\{\log r_1^{-1},\ldots,\log r_M^{-1}\}$ is $\tau$-arithmetic, then we can choose $\ell=\ell(n)$ so that
\begin{equation*}
\min\{e^{-\tau},1\}\cdot e^{-\tau n}< \delta_{\ell}^{(k)}\leq \max\{e^{-\tau},1\}\cdot e^{-\tau n},
\end{equation*}
which implies that
\begin{equation*}
\lim_{n\to\infty}\frac{\ell(n)}{n}=\frac{\tau}{-\log a^{(k)}} \;\text{ and }\; \lim_{n\to\infty} \frac{N_{e^{-\tau n}}(\Lambda_{\mathcal{H}})}{N_{\delta_{\ell}^{(k)}}(\Lambda_{\mathcal{H}})} = c
\end{equation*}
for some universal constant $c$. Thus the proposition implies that
\begin{equation}\label{eq:PreciseN(R)}
N_{\delta^{(k)}_{\ell}}\big(\widetilde{R}_{\iiv}\big)= N_{\delta^{(k)}_{\ell}/(b^{(k)})^\ell}(\Lambda_{\mathcal{H}})= (C+o(1))(b^{(k)}/a^{(k)})^{\ell s_{\mathcal{H}}},
\end{equation}
where the constant $C$ only depends on whether $\{\log r_1^{-1},\ldots,\log r_M^{-1}\}$ is $\tau$-arithmetic or not and the $o(1)\to 0$ as $\ell\to\infty$. The next lemma ensures that a positive proportion of these boxes do not get covered by boxes coming from the cover of $\widetilde{R}_{\jjv}$ for some $\jjv\in\Sigma_{\iiv}^{\sim}$.

\begin{lemma}\label{lemma:CountingIntersections}
If $\mathcal{F}$ satisfies transversality and
\begin{equation}\label{ass:IntersectingBoxes}
\frac{N^{(k)}}{M^{(k)}}\big(1+K_1^{s_{\mathcal{H}}}\big) < \left(\frac{b^{(k)}}{a^{(k)}}\right)^{s_{\mathcal{H}}},
\end{equation}
then there exists $K_3<1$ such that for $\ell$ large enough and every $\iiv\in \Sigma^{(k,\ell)}$ we have
\begin{equation*}
N_{\delta^{(k)}_{\ell}}\Big(\widetilde{R}_{\iiv}\cap \bigcup_{\jjv\in\Sigma_{\iiv}^{\sim}}\widetilde{R}_{\jjv}\Big) \leq K_3 N_{\delta^{(k)}_{\ell}}\big(\widetilde{R}_{\iiv}\big).
\end{equation*}
\end{lemma}

\begin{proof}
Fix $\jjv\in\Sigma^{(k,\ell)}$ such that $|\jjv\wedge\iiv|=z$, where we count $\iiv_m,\jjv_m\in\mathcal{M}_{\mathbf{v}_k}$ as one symbol. Thus, $z\in\{0,1,\ldots,\ell-1\}$. Since $\mathcal{F}$ satisfies transversality, then so do all of its subsystems, in particular $\mathcal{G}^{(k,\ell)}$ as well. Hence,
\begin{equation*}
|\proj_x\big(\widetilde{R}_{\iiv}\cap \widetilde{R}_{\jjv}\big)| \leq K_1 \big(b^{(k)}\big)^{z} \big(a^{(k)}\big)^{\ell-z},
\end{equation*}
see Figure~\ref{fig:BoxOverlap}. This together with \eqref{eq:infectedBoxes} and Proposition~\ref{prop:exactN_delta} yields that
\begin{equation*}
N_{\delta^{(k)}_{\ell}}(\widetilde{R}_{\iiv}\cap \widetilde{R}_{\jjv}) \leq (C+o(1)) K_1^{s_{\mathcal{H}}} \left(\frac{b^{(k)}}{a^{(k)}}\right)^{z s_{\mathcal{H}}}.
\end{equation*}
Since $\mathcal{G}^{(k,\ell)}$ has uniform vertical fibres, it follows that $\#\{\jjv\in\Sigma_{\iiv}^{\sim}:\, |\jjv\wedge\iiv|=z\}\leq(N^{(k)}/M^{(k)})^{\ell-z}$. Thus from a simple union bound we get
\begin{align*}
N_{\delta^{(k)}_{\ell}}\Big(\widetilde{R}_{\iiv}\cap \bigcup_{\jjv\in\Sigma_{\iiv}^{\sim}}\widetilde{R}_{\jjv}\Big) &\leq \sum_{z=0}^{\ell-1} \left(\frac{N^{(k)}}{M^{(k)}}\right)^{\ell-z} (C+o(1)) K_1^{s_{\mathcal{H}}} \left(\frac{b^{(k)}}{a^{(k)}}\right)^{z s_{\mathcal{H}}} \\
&= \underbrace{\frac{K_1^{s_{\mathcal{H}}}}{\frac{M^{(k)}}{N^{(k)}} \big(\frac{b^{(k)}}{a^{(k)}}\big)^{s_{\mathcal{H}}} -1 }}_{=:K_3} (C+o(1))\! \left( \!\!\left(\frac{b^{(k)}}{a^{(k)}}\right)^{\ell s_{\mathcal{H}}} \!\!\!-\! \left(\frac{M^{(k)}}{N^{(k)}}\right)^\ell \right) \!\leq K_3 N_{\delta^{(k)}_{\ell}}\big(\widetilde{R}_{\iiv}\big),
\end{align*}
where the last inequality holds if $N^{(k)}/M^{(k)}\leq (b^{(k)}/a^{(k)})^{ s_{\mathcal{H}}}$. This holds, because \eqref{ass:IntersectingBoxes} is an even stronger assumption. Furthermore, simple arithmetic shows that $K_3<1$ if and only if \eqref{ass:IntersectingBoxes} holds.
\end{proof}

\begin{center}
	\begin{figure}[h]
		\centering
		\includegraphics[width=.85\linewidth]{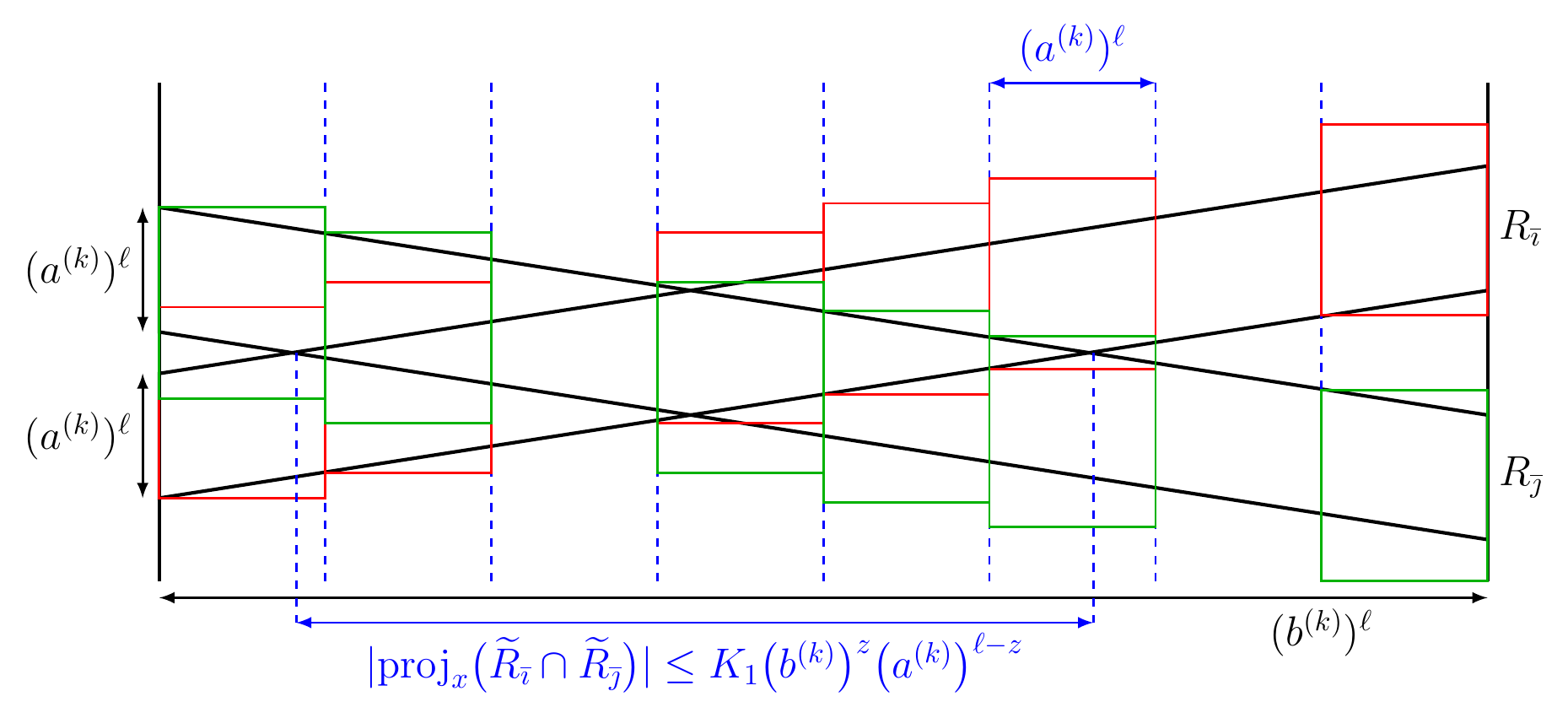}
		\caption{Intersecting parallelograms $R_{\iiv}$ and $R_{\jjv}$ in the proof of Lemma \ref{lemma:CountingIntersections}.}\label{fig:BoxOverlap}
	\end{figure}
\end{center}

\subsection{Proof of Theorem~\ref{thm:BoxwOverlaps}}\label{subsec:ProofBoxDim}

Throughout the proof, $s$ is the target box dimension defined as the solution of \eqref{a74}: $\sum_{i=1}^{N} b_{i}^{s_{\mathcal{H}}}a_{i}^{s-s_{\mathcal{H}}}=1$. Fix $\varepsilon>0$. We work with the subsystem $\mathcal{G}^{(k,\ell)}$ defined in Subsection~\ref{subsec:BoxPrelim}. It will be enough to cover the subset
\begin{equation*}
\bigcup_{\iiv\in \mathcal{G}^{(k,\ell)}} f_{\iiv}(\Lambda)\subseteq \Lambda,
\end{equation*}
with boxes of size $\delta^{(k)}_{\ell}:=(a^{(k)})^{\ell}$. Recall $\widetilde{R}_{\iiv}= (\proj_x(f_{\iiv}(\Lambda))\times[0,1])\cap (f_{\iiv}([0,1]^2))$.

\begin{proof}[Conclusion of proof assuming condition (i) of Theorem~\ref{thm:BoxwOverlaps}]
Assume $\mathcal{F}$ generates a shifted TGL carpet $\Lambda$ for which $\mathcal{H}$ satisfies Hochman's Exponential Separation Condition and the columns independently satisfy ROSC. In this case it is enough to use the definition of box dimension to bound $N_{\delta^{(k)}_{\ell}}\big(\widetilde{R}_{\iiv}\big)\geq C_{\varepsilon}(b^{(k)}/a^{(k)})^{\ell(s_{\mathcal{H}}-\varepsilon)}$ for some constant $C_{\varepsilon}$ depending only on $\varepsilon$. Recall from Lemma~\ref{lemma:DiagHomoSubsys} that $s^{(k)}\to s$. We choose $k$ so large that $s^{(k)}\geq s-\varepsilon$ and we bound
\begin{align}
\liminf_{\ell\to\infty}\frac{\log N_{\delta^{(k)}_{\ell}}(\Lambda)}{-\log \delta^{(k)}_{\ell}} &\geq \liminf_{\ell\to\infty} \frac{\log\big( M^{(k,\ell)}  (N^{(k)}/M^{(k)})^{\ell} (b^{(k)}/a^{(k)})^{\ell (s_{\mathcal{H}}-\varepsilon)} \big)}{-\ell \log a^{(k)}} \label{eq:liminfDimB}\\
&\geq \underbrace{ \frac{\log N^{(k)}}{-\log a^{(k)}} + \left(1-\frac{\log b^{(k)}}{\log a^{(k)}}\right)s_{\mathcal{H}} }_{=s^{(k)}\geq s-\varepsilon}-\varepsilon  \geq s-2\varepsilon, \nonumber
\end{align}
where for the second inequality we substituted the lower bound for $M^{(k,\ell)}$ from \eqref{eq:BoundMkl}. Letting $\varepsilon\searrow0$ yields $\underline{\dim}_{\mathrm B}\Lambda\geq s$ as claimed.
\end{proof}

\begin{proof}[Conclusion of proof assuming condition (ii) of Theorem~\ref{thm:BoxwOverlaps}]
For the remainder we assume that $\mathcal{F}$ has non-overlapping columns, satisfies transversality and \eqref{cond:BoxDimCond}:
\begin{equation*}
h_{\widetilde{\mathbf p}}-h_{\widetilde{\mathbf q}} < s_{\mathcal{H}}(\log\langle \mathbf{b} \rangle_{\widetilde{\mathbf p}}-\log\langle \mathbf{a} \rangle_{\widetilde{\mathbf p}}),
\end{equation*}
where $h_{\widetilde{\mathbf p}}=-\log\langle \widetilde{\mathbf{p}}\rangle_{\widetilde{\mathbf p}}$ and $\widetilde{p}_i=b_i^{s_{\mathcal{H}}}a_i^{s-s_{\mathcal{H}}}$. We need to check that condition \eqref{ass:IntersectingBoxes} of Lemma~\ref{lemma:CountingIntersections} is satisfied, since it ensures that a positive proportion of the boxes needed to cover $f_{\iiv}(\Lambda)$ are not intersected by any boxes covering $f_{\jjv}(\Lambda)$ for $\jjv\neq\iiv$.
\begin{claim}\label{claim:TechCondBox}
Assumption \eqref{cond:BoxDimCond} implies condition \eqref{ass:IntersectingBoxes} of Lemma~\ref{lemma:CountingIntersections} for all $k$ large enough.
\end{claim}
\begin{proof}[Proof of Claim~\ref{claim:TechCondBox}]
We know from Subsection~\ref{subsec:BoxPrelim} that
\begin{align*}
\log a^{(k)} &= k \log\langle \mathbf{a} \rangle_{\widetilde{\mathbf p}} +O(1),\;  \log N^{(k)} = k h_{\widetilde{\mathbf p}} +o(k), \\
\log b^{(k)} &= k \log\langle \mathbf{b} \rangle_{\widetilde{\mathbf p}} +O(1),\; \log M^{(k)} = k h_{\widetilde{\mathbf q}} +o(k).
\end{align*}
Taking the logarithm of each side of \eqref{ass:IntersectingBoxes}, substituting these values and dividing by $k$ gives
\begin{equation*}
h_{\widetilde{\mathbf p}}-h_{\widetilde{\mathbf q}} + \frac{1}{k}\log (1+K_1^{s_{\mathcal{H}}}) < s_{\mathcal{H}}(\log\langle \mathbf{b} \rangle_{\widetilde{\mathbf p}}-\log\langle \mathbf{a} \rangle_{\widetilde{\mathbf p}}),
\end{equation*}
with an error of $o(1)$ as $k\to\infty$ on either side. The second term on the left hand side also tends to zero as $k\to\infty$, thus \eqref{cond:BoxDimCond} indeed implies \eqref{ass:IntersectingBoxes} for large $k$.
\end{proof}

The conclusion of the proof of Theorem~\ref{thm:BoxwOverlaps} is now analogous to the calculation of~\eqref{eq:liminfDimB} with the exception that we need the precise value of $N_{\delta^{(k)}_{\ell}}\big(\widetilde{R}_{\iiv}\big)$ from \eqref{eq:PreciseN(R)} and we can use $\mathcal{G}^{(k,\ell)}=(\mathcal{G}^{(k)})^{\ell}$, so the number of columns $M^{(k,\ell)}=(M^{(k)})^{\ell}$. Choose $k$ so large that $s^{(k)}\geq s-\varepsilon$ and condition \eqref{ass:IntersectingBoxes} hold simultaneously. Using Lemma~\ref{lemma:CountingIntersections} we can basically repeat the calculation of~\eqref{eq:liminfDimB}
\begin{equation*}
\liminf_{\ell\to\infty}\frac{\log N_{\delta^{(k)}_{\ell}}(\Lambda)}{-\log \delta^{(k)}_{\ell}} \geq \liminf_{\ell\to\infty} \frac{\log\big( (N^{(k)})^{\ell} (1-K_3)(C+o(1))(b^{(k)}/a^{(k)})^{\ell s_{\mathcal{H}}} \big)}{-\ell \log a^{(k)}} = s^{(k)}.
\end{equation*}
This concludes the proof of Theorem~\ref{thm:BoxwOverlaps}.
\end{proof}

\subsection{Proof of Theorem \ref{b62}}\label{subsec:ProofBoxHausd}
The theorem claims that for a shifted TGL carpet $\Lambda$
\begin{equation*}
(i)\, \dim_{\rm H}\Lambda=\dim_{\rm B}\Lambda \qquad (ii)\, s_{\mathcal{H}}=\dim_{\mathrm H}\nu_{\widetilde{\mathbf q}} \qquad (iii)\, \sum_{j\in \mathcal{I}_{\ih} }a_j^{s-s_{\mathcal{H}}}=1 \text{ for every } \ih \in[M]
\end{equation*}
are equivalent, provided ROSC and No Dimension Drop (NDD, recall Definition~\ref{def:34}) hold. We show that $(i)\Leftrightarrow(iii),\; (iii)\Rightarrow (ii)$ and $(ii)\Rightarrow(i)$.

Proof of $(i)\Leftrightarrow(iii)$. Let $\widetilde{ \Lambda}$ be the GL brother of $\Lambda$, recall Definition \ref{def:GLBrother}. For a $\mathbf{p}\in\mathcal{P}_0$ let $\widetilde{\nu}_\mathbf{p}$ denote the push forward of the Bernoulli measure $\mu_{\mathbf p}$ on $\widetilde{ \Lambda}$. We have $\dim_{\mathrm H}\nu_\mathbf{p}=\dim_{\mathrm H}\widetilde{\nu}_\mathbf{p}$ for every $\mathbf{p}\in\mathcal{P}_0$. Indeed, in the beginning of Section~\ref{sec:dimH_lowerbound} we proved $\dim_{\mathrm H}\nu_\mathbf{p}=D(\mathbf{p})$ assuming ROSC and NDD, furthermore, Gatzouras--Lalley proved $\dim_{\mathrm H}\widetilde{\nu}_\mathbf{p}=D(\mathbf{p})$ \cite[Proposition~3.3]{GatzourasLalley92}. Hence, $\dim_{\mathrm H}\Lambda=\dim_{\mathrm H}\widetilde{ \Lambda}$. Also, assuming NDD, $s_{\mathcal{H}}$ is the unique real which satisfies $\sum_{\ih=1}^M r_{\ih}^{s_{\mathcal{H}}}=1$. This implies $\dim_{\mathrm B}\Lambda=\dim_{\mathrm B}\widetilde{ \Lambda}$. The analogous claim of $(i)\Leftrightarrow(iii)$ for $\widetilde{\Lambda}$ was proved in \cite[Theorem~4.6]{GatzourasLalley92}. Thus $(i)\Leftrightarrow(iii)$ in our setting as well.

Proof of $(iii)\Rightarrow(ii)$. Condition $(iii)$ implies that the vector $\widetilde{\mathbf{q}}$ is simply $\widetilde{q}_{\ih}=r_{\ih}^{s_{\mathcal{H}}}$ for $\ih\in[M]$, where $r_{\ih}=b_j$ if $j\in\mathcal{I}_{\ih}$. NDD is assumed, thus $\dim_{\mathrm H}\nu_{\widetilde{\mathbf q}} = h_{\widetilde{\mathbf{q}}}/\chi^1_{\widetilde{\mathbf{q}}}= s_{\mathcal{H}}\chi^1_{\widetilde{\mathbf{q}}}/\chi^1_{\widetilde{\mathbf{q}}}=s_{\mathcal{H}}$.

Proof of $(ii)\Rightarrow(i)$. We can use Claim~\ref{claim:LedrappierforBox} and \eqref{eq:dimHmuBaranyKaenmaki} to see that
\begin{equation*}
0 \leq \dim_{\mathrm B}\Lambda-\dim_{\mathrm H}\Lambda \leq \dim_{\mathrm B}\Lambda - \dim_{\mathrm H} \nu_{\widetilde{\mathbf p}} = \left(1-\chi_{\widetilde{\mathbf p}}^1 / \chi_{\widetilde{\mathbf p}}^2\right)\left(s_{\mathcal{H}} - \dim_{\mathrm H} \nu_{\widetilde{\mathbf q}}\right).
\end{equation*}
Clearly, $(ii)$ implies $\dim_{\rm H}\Lambda=\dim_{\rm B}\Lambda$. This concludes the proof of Theorem~\ref{b62}.

\section{Examples}\label{sec:ex}


We now treat the examples presented in Subsection \ref{subsec:OurContrib} in detail.

We do not calculate numerically the exact value of the dimensions for the TGL carpet of Figure \ref{fig:GLandTGL}, rather just comment why $\dim_{\rm H}\Lambda<\dim_{\rm B}\Lambda<\dim_{\rm{Aff}}\Lambda$. It satisfies the ROSC, thus its dimensions are equal to its GL brother. Clearly, the IFSs on $[0,1]$ generated from a vertical line in each of the columns do not have the same dimension. Hence, the third condition of \eqref{eq:dimB=Hiff} of Theorem~\ref{b62} does not hold. Furthermore, $\dim_{\rm B}\Lambda_{\mathcal{H}}<1$ because there is an empty column. Thus, Corollary \ref{cor:dimB=dimA} implies that $\dim_{\rm B}\Lambda<\dim_{\rm{Aff}}\Lambda$.

Except for the "$X\equiv X$" example, all the other ones of Subsection \ref{subsec:OurContrib} satisfy $\Lambda_{\mathcal{H}}=[0,1]$, hence Corollary~\ref{cor:dimB=dimA} implies $\dim_{\rm B}\Lambda=\dim_{\rm{Aff}}\Lambda$.

\subsection{The self-affine smiley: a non diagonally homogeneous example}\label{subsec:smiley}

The smiley is constructed from the TGL IFS
\begin{equation*}
\mathcal{F}=\left\{f_i(x)=\begin{pmatrix} b & 0 \\ d_i & a_i \end{pmatrix}x+t_i\right\}_{i=1}^8,
\end{equation*}
where $b=0.2,\, a_1=\ldots=a_5=0.1,\, a_6=a_7=a_8=0.13$ and the off-diagonal elements $d_1=-0.2,\, d_2=-0.1,\, d_3=d_7=d_8=0,\, d_4=0.1,\, d_5=d_6=0.2$. The translations were chosen so that
the mouth is constructed from $f_1,\ldots,f_5$, the nose from $f_6$ and the eyes from $f_7$ and $f_8$. It is non diagonally homogeneous since the mouth is thinner than the nose and eyes.
Clearly, $\Lambda$ does not have uniform vertical fibres, thus Theorem \ref{b62} implies $\dim_{\rm H}\Lambda<\dim_{\rm B}\Lambda$. The numerical values of the dimensions given in Figure~\ref{fig:smiley} were obtained using \textit{Wolfram Mathematica 11.2}. The box dimension was calculated from $\sum_{i=1}^{N} b_{i}^{s_{\mathcal{H}}}a_{i}^{s-s_{\mathcal{H}}}=1$, recall~\eqref{a74}, while the maximization of $D(\mathbf{p})$ \eqref{def:D(p)} gave the Hausdorff dimension.

\subsection{Example for \texorpdfstring{$\dim_{\rm H}\Lambda=\dim_{\rm B}\Lambda$}{dimHLambda=dimBLambda}}\label{subsec:Ex2}
Define the matrices
\begin{equation*}
A_1:= \begin{pmatrix} 1/3 & 0 \\ 0 & a \end{pmatrix},\; A_2:=\begin{pmatrix} 1/3 & 0 \\ 1/2-a & a \end{pmatrix},\; A_3:=\begin{pmatrix} 1/3 & 0 \\ a-1/2 & a \end{pmatrix}.
\end{equation*}
For $a\in(0,1/3)$ define the IFS $\mathcal{F}_a$ consisting of
\begin{align*}
f_1(\underline{x}) &= A_1 \underline{x} + \begin{pmatrix} 1/3 \\ 0 \end{pmatrix}, &f_2(\underline{x}) &= A_1\underline{x} + \begin{pmatrix} 1/3 \\ 1-a \end{pmatrix}, &f_3(\underline{x}) &= A_2\underline{x} + \begin{pmatrix} 0 \\ 1/2 \end{pmatrix}, \\
f_4(\underline{x}) &= A_2\underline{x} + \begin{pmatrix} 2/3 \\ 0 \end{pmatrix}, &f_5(\underline{x}) &= A_3\underline{x} + \begin{pmatrix} 0 \\ 1/2-a \end{pmatrix}, &f_6(\underline{x}) &= A_3\underline{x} + \begin{pmatrix} 2/3 \\ 1-a \end{pmatrix}.
\end{align*}
The attractor $\Lambda_a$ is shown in Figure \ref{fig:carpetex2} for $a=3/10$. Falconer and Miao showed in \cite{FalconerMiao07} how to calculate the box dimension and later B\'ar\'any in \cite{barany15_LYformula} showed that the same value is a lower bound for the Hausdorff dimension. Hence, $\dim_{\rm H}\Lambda_a=\dim_{\rm B}\Lambda_a$.

Alternatively, we can now argue that $\Lambda_a$ is a diagonally homogeneous TGL carpet for every $a\in(0,1/3)$ satisfying ROSC with uniform vertical fibres. Hence, our results apply. After some basic arithmetic, the dimension formula simplifies to
\begin{equation}\label{eq:dimHExample}
\dim_{\rm H}\Lambda_a =\dim_{\rm B}\Lambda_a= 1-\dfrac{\log 2}{\log a}.
\end{equation}

\subsection{Overlapping example}\label{subsec:overlappingex}
With a modification of the translation vectors in the previous example, we construct a carpet with overlapping cylinders, see Figure \ref{fig:carpetex2}. Define
\begin{align*}
f_1(\underline{x}) &= A_1 \underline{x} + \begin{pmatrix} 1/3 \\ 1/4 \end{pmatrix}, &f_2(\underline{x}) &= A_1\underline{x} + \begin{pmatrix} 1/3 \\ 3/4-a \end{pmatrix}, &f_3(\underline{x}) &= A_2\underline{x} + \begin{pmatrix} 0 \\ 1/4 \end{pmatrix}, \\
f_4(\underline{x}) &= A_2\underline{x} + \begin{pmatrix} 2/3 \\ 1/4 \end{pmatrix}, &f_5(\underline{x}) &= A_3\underline{x} + \begin{pmatrix} 0 \\ 3/4-a \end{pmatrix}, &f_6(\underline{x}) &= A_3\underline{x} + \begin{pmatrix} 2/3 \\ 3/4-a \end{pmatrix},
\end{align*}
where the matrices $A_1,A_2$ and $A_3$ are from Subsection~\ref{subsec:Ex2}. For $a\in(0,1/3)$ the attractor $\Lambda_a$ is a diagonally homogeneous TGL carpet with uniform vertical fibres and non-overlapping columns. Transversality must be satisfied in order to apply our results. It would suffice to check \eqref{a79} in Lemma~\ref{a87}, but in fact the constant $K_1$ in Definition~\ref{def:separations} of transversality can be directly bounded in this example.

\begin{claim}\label{claim:CheckTrans}
Transversality holds for every $a<1/6$ with
\begin{equation*}
K_1< \frac{1/9-a/3}{(1/2-a)(1/3-2a)}.
\end{equation*}
\end{claim}
\begin{proof}
For brevity we write $d:=1/2-a$ and $b=1/3$. Let $\iiv$ and $\jjv$ be two words of length $n$ such that $i_1\neq j_1$ and $\phi(i_1)\ldots\phi(i_n)=\phi(j_1)\ldots\phi(j_n)$. Since $R_{\iiv} \cap R_{\jjv}\neq\emptyset$ and due to the symmetry in the construction, we may assume $i_1=3$ and $j_1=5$, hence $d_{i_1}=d$. A simple geometric exercise gives that $K_1\leq (\min_{\iiv}\, \tan \gamma_{\iiv})^{-1}$, where $\tan \gamma_{\iiv}=d_{\iiv} / b_{\iiv}$. We need a lower bound for $\tan \gamma_{\iiv}$. From \eqref{a93} we get that
\begin{equation*}
\tan \gamma_{i_1 \dots i_n}=\frac{d_{i_1 \dots i_n}}{b_{i_1 \dots i_n}} = \frac{1}{a}\,\sum_{\ell=1}^nd_{i_\ell}\Big(\frac{a}{b}\Big)^\ell = \frac{1}{a}\left( \frac{da}{b}+\sum_{\ell=2}^nd_{i_\ell}\Big(\frac{a}{b}\Big)^\ell\right).
\end{equation*}
This is minimal if $d_{i_\ell}=-d$ for every $\ell\geq2$. Thus, we obtain the lower bound
\begin{equation*}
\tan \gamma_{i_1 \dots i_n}\geq \frac{d}{b} \left( 1- \sum_{\ell=2}^n\Big(\frac{a}{b}\Big)^{\ell-1}\right) \geq \frac{d}{b} \left(1-\frac{a/b}{1-a/b}\right) = \frac{d(b-2a)}{b(b-a)}.
\end{equation*}
This remains positive iff $a<b/2=1/6$. Substituting $d$ and $b$ gives the bound for $K_1$.
\end{proof}
\begin{corollary}
For every $a<1/6:\; \dim_{\rm H}\Lambda_a=\dim_{\rm B}\Lambda_a=1-\log2/\log a$.
\end{corollary}
\begin{proof}
For uniform vertical fibres both conditions \eqref{cond:main} and \eqref{cond:BoxDimCond} simplify to
\begin{equation*}
\frac{\log a}{\log b}>\frac{\log N}{\log M},\; \text{ which is satisfied here iff } a\in(0,1/6).
\end{equation*}
Thus, Claim~\ref{claim:CheckTrans} and Corollary~\ref{cor:dimH_BMcarpet} together imply that for every $a\in(0,1/6)$ we have $\dim_{\rm H}\Lambda_a=\dim_{\rm B}\Lambda_a=1-\log2/\log a$.
\end{proof}

\subsection{Example \texorpdfstring{"$X\equiv X$"}{"XX"}}\label{subsec:ExXequivX}
This diagonally homogeneous carpet, recall Figure \ref{fig:XequalX}, is a modification of the previous from Subsection \ref{subsec:overlappingex} in order to show an overlapping example for which all dimensions are different. Indeed, clearly it does not have uniform vertical fibres and there are empty columns.

The main diagonal of each matrix in the TGL IFS is $b_i\equiv b =0.28$ and $a_i\equiv a$. The off-diagonal elements are either $d_i=\pm(1/2-a)$ or $0$. The translation vectors were chosen so that $\Lambda_a$ is symmetric on both lines $x=1/2$ and $y=1/2$. In Figure \ref{fig:XequalX} $a=0.045$.

Transversality for the system can be checked the same way as in Claim \ref{claim:CheckTrans}, to obtain that transversality holds for every $a<b/2=0.14$ with
\begin{equation*}
K_1<\frac{0.28(0.28-a)}{(1/2-a)(0.28-2a)}.
\end{equation*}

\begin{corollary}
We have $\dim_{\mathrm H}\Lambda_a<\dim_{\mathrm B}\Lambda_a<\dim_{\mathrm{Aff}}\Lambda_a$, where
\begin{align*}
\dim_{\mathrm H}\Lambda_a &= 0.78556\cdot\log \Big(2\cdot 2^{\frac{1.27297}{-\log a}}+3^{\frac{1.27297}{-\log a}}\Big),\; &\text{for every } a&<0.10405\ldots\,, \\
\dim_{\mathrm B}\Lambda_a &= \frac{0.84730}{-\log a} +0.86303,\; &\text{for every } a&<0.10254\ldots\,, \\
\dim_{\mathrm{Aff}}\Lambda_a &= 1+\frac{0.67294}{-\log a},\; &\text{for every } a&<0.28\,.
\end{align*}
\end{corollary}
\begin{proof}
The formulas are applications of the ones in Corollary \ref{cor:dimH_BMcarpet} and \eqref{def:dimA}. The affinity dimension is independent of overlaps. The bound for $a$ in case of the Hausdorff dimension was obtained using Proposition \ref{prop:diaghomoCondition}. The value $x_0=0.56255\ldots $ for which $R(x_0)=1$ was calculated using \textit{Wolfram Mathematica 11.2}. Then \eqref{cond:main} holds for every $a<b^{1/x_0}=0.10405\ldots\,$. The bound on $a$ for the box dimension simply comes from substituting the parameters into the second inequality in \eqref{cond:DiagHomoCase}.
\end{proof}

\subsection{Negative entries in the main diagonal}\label{subsec:NegEntries}

Throughout we assumed that $0<a_i<b_i<1$. We now comment on letting $a_i$ or $b_i<0$. For convenience, assume ROSC and non-overlapping columns.

\begin{prop}\label{prop:NegDiagonalElements}
The dimension results of Theorems \ref{thm:maindimresult} and \ref{thm:mainBox} extend to TGL carpets satisfying the ROSC under the weaker condition that $0<|a_i|<|b_i|<1$ and for every fixed $\ih\in\{1,\ldots,M\}$ and every $k,\ell\in\mathcal{I}_{\ih}:\; b_k=b_\ell$.
\end{prop}
\begin{proof}[Sketch of proof] All lower triangular matrix can be written
\begin{equation*}
\begin{pmatrix} b_i & 0 \\ d_i & a_i \end{pmatrix} = \begin{pmatrix} |b_i| & 0 \\ \bar d_i & |a_i| \end{pmatrix}\cdot L 
\end{equation*}
where $\bar d_i=d_i$ or $-d_i$ and $L$ is a reflection on one or both of the coordinate axis. Since $L([-1,1]^2)=[-1,1]^2$, such compositions fit into the framework of Fraser's box-like sets \cite{Fraser12Boxlike}. Furthermore, the direction-$x$ dominates property is preserved. Hence, the proof of the box dimension from Section \ref{sec:boxdim} immediately extends to this setting.

The lower bound for the Hausdorff dimension follows from B\'ar\'any--K\"{a}enm\"{a}ki \cite{BARANYAnti201788} cited in Theorem \ref{thm:BaranyKaenmaki}. Since in any given column all $b_i$ have the same sign and we have ROSC, the column structure is preserved for every level. Thus, the dimension of the projected measure $\nu_{\mathbf q}$ is not affected by the negative $a_i, b_i$. For the upper bound, we can modify the metric defined on $\Sigma$ in Lemma \ref{lemma:MetricSpace} to be
\begin{equation*}
d(\ii,\jj):= \prod_{k=1}^{|\iih\wedge\jjh|}|b_{i_k}| + \prod_{k=1}^{|\ii\wedge\jj|}|a_{i_k}|.
\end{equation*}
One can easily check that $d(\ii,\jj)$ is indeed a metric and the natural projection $\Pi:\Sigma\to\Lambda$ is Lipschitz. Only the lengths of the sides of a parallelogram are important, its orientation is not. The Bernoulli measure defined in \eqref{def:GLBernoulliMeasure} can be modified by again putting $a_i$ and $b_i$ in absolute value. The original proof of Gatzouras and Lalley \cite{GatzourasLalley92} does not use that $a_i, b_i>0$, only that $0<|a_i|<|b_i|<1$.
\end{proof}

In general, if a column has $b_i$ of different signs, then the initial column structure can easily be destroyed. This is true even if $|b_i|\equiv b$ and possibly empty columns also have width~$b$, see Figure \ref{fig:badcolumns}. This motivates us to call a TGL carpet \textit{symmetric} if $N_{\ih}=N_{M-\ih+1}$ for $\ih=1,\ldots,\lfloor M/2\rfloor$ (empty columns are allowed) and $|b_i|\equiv 1/M$. For a particular symmetric carpet, in the next subsection, we show that the dimension formulas hold.
\vspace{-0.5cm}
\begin{figure}[h]
	\centering
	\includegraphics[scale=1.35]{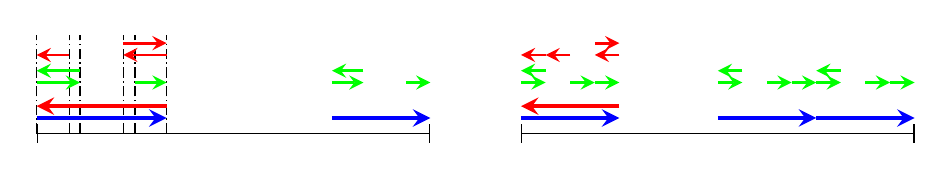}\vspace{-0.5cm}
	\caption{Orientation reversing maps generally destroy the column structure. First and second level cylinders of the horizontal IFS $\widetilde{\mathcal{H}}$ are shown, arrows indicating the orientation. Left: different $|b_i|$, right: equal $|b_i|$ and gap as well.}\label{fig:badcolumns}
\end{figure}

\subsection{A family of self-affine continuous curves}\label{subsec:exZipper}

Let $a\in(0,1/5]$ and $d=(1-5a)/4$. Define the matrices
\begin{equation*}
A=\begin{pmatrix} 1/3 & 0 \\ d & a \end{pmatrix} \;\text{ and }\; A^-=\begin{pmatrix} -1/3 & 0 \\ 0 & a \end{pmatrix}.
\end{equation*}
$A^-$ is orientation reversing. We introduce the parameterized family of IFSs $\mathcal{F}_a$ given by the functions
\begin{align*}
f_1(\underline{x}) &= A\underline{x}, &f_2(\underline{x}) &= A\underline{x} + \begin{pmatrix} 1/3 \\ a+d \end{pmatrix}, &f_3(\underline{x}) &= A^-\underline{x} + \begin{pmatrix} 2/3 \\ 2(a+d) \end{pmatrix}, \\
& &f_4(\underline{x}) &= A\underline{x} + \begin{pmatrix} 1/3 \\ 3a+2d \end{pmatrix}, &f_5(\underline{x}) &= A\underline{x} + \begin{pmatrix} 2/3 \\ 4a+3d \end{pmatrix}.
\end{align*}
The translation vectors are chosen so that $f_1(\underline{0})=\underline{0},\, f_5((1,1))=(1,1)$ and $f_i((1,1))=f_{i+1}(\underline{0})$. This ensures that $\Lambda_a$ is a continuous curve in $\R^2$, see Figure \ref{fig:zipperex}. Curves satisfying this property are also called affine zippers in the literature, see for example \cite{AseevTetenov2003,BKK18_zipper}.
Clearly, the attractor $\Lambda_a$ is a symmetric, diagonally homogeneous TGL carpet satisfying the ROSC for every value of $a$.
For $a=1/5$ all cylinders $R_{\ii|n}$ are rectangles, however it is not a classical Bedford-McMullen carpet, since $A^-$ contains a negative element.
\begin{prop}\label{prop:zipper}
For every $a\in(0,1/5]$, the Hausdorff and box dimension of $\Lambda_a$ are given by the continuous, strictly increasing functions
\begin{equation*}
\dfrac{1}{\log 3} \cdot \log\Big(2+3^{\frac{\log3}{-\log a}}\Big) = \dim_{\rm H} \Lambda_a < \dim_{\rm B} \Lambda_a = 1+\frac{\log(3/5)}{\log a} .
\end{equation*}
\end{prop}

\begin{proof}
$A^-$ can be written as the composition of the reflection on the vertical axis with the diagonal matrix $\mathrm{Diag}(1/3,a)$. Hence, the proof of the box dimension carries over without difficulty.

The argument for the Hausdorff dimension follows that in Proposition \ref{prop:NegDiagonalElements}, with an extra argument why the dimension of $\nu_{\mathbf q}$ is not affected by $A^-$.

The symbolic space $\Sigma=\{1,\ldots,5\}^\N$ codes the IFS $\mathcal{F}_a$ on $[0,1]^2$ and $\widetilde{\mathcal{H}}_a$ on $[0,1]$ (recall \eqref{def:verticalIFS}). Fix a $\mathbf{p}=(p_1,\ldots,p_5)\in\mathcal{P}$. Due to the symmetry and diagonally homogeneous property we may assume that $p_1=p_5$. Let $\mu_{\mathbf p}$ be the Bernoulli measure on $\Sigma$ and $\nu_{\mathbf p}=\Pi_\ast\mu_{\mathbf p}$ its push forward. Define the IFS $\mathcal{H}_a:= \{h_i(x)=x/3+(i-1)/3,\, i=1,2,3\}$, which is coded by $\Sigma_{\mathcal{H}}=\{1',2',3'\}^\N$. The map $\phi:  \{1,\ldots,5\}\to\{1',2',3'\}$ is defined
\begin{equation*}
\phi(1) = 1',\;\; \phi(2)=\phi(3)=\phi(4)=2',\;\; \phi(5)=3'.
\end{equation*}
For $\iiv=i_1\ldots i_n\in \{1',2',3'\}^n$ let us denote $J^k(\iiv):= \{j:\, i_j=k,\, j\leq |\iiv|\},\;\#^{k}(\iiv):=|J^k(\iiv)|$ and define  $\nu_{\mathbf q} := (\mathrm{proj}_x)_\ast \nu_{\mathbf p}$. We claim that
\begin{equation}\label{eq:nuqinZipper}
\nu_{\mathbf q}( h_{\iiv}([0,1]))  = p_1^{\#^{1'}(\iiv)+\#^{3'}(\iiv)}\cdot (p_2+p_3+p_4)^{\#^{2'}(\iiv)},
\end{equation}
i.e. $\nu_{\mathbf q}$ is the push forward $(\Pi_{\mathcal{H}})_\ast\mu_{\mathbf q}$ of the Bernoulli measure $\mu_{\mathbf q}$ on $\Sigma_{\mathcal{H}}$ defined by the vector $\mathbf{q}=(q_1,q_2,q_3)=(p_1,p_2+p_3+p_4,p_5)$. This implies that $$\dim_{\rm H}\nu_{\mathbf q}=\frac{\log \langle \mathbf{q} \rangle_{\mathbf{q}} }{-\log 3}.$$
To see \eqref{eq:nuqinZipper}, choose an arbitrary $\iiv\in\{1',2',3',\}^n$. We determine those $\iih\in\{1,\ldots,5\}^n$ for which $\proj_x f_{\iih}([0,1]^2)=h_{\iiv}([0,1])$. For indices $j\in J^{2'}(\iiv)$ we can choose $2,3$ or $4$ in $\iih$. Let $J_{\ell}^{3}(\iih):=\{j:\, \iih_j=3,\, j\leq\ell\leq |\iih|\} \subseteq J^{2'}(\iiv)$. Orientation is reversed at each $j\in J_\ell^3(\iih)$. $|J_{\ell}^{3}(\iih)|$ uniquely determines $\iih_\ell$ if $i_\ell=1'$ or $3'$. Namely, whenever
\begin{equation*}
|J_{\ell}^{3}(\iih)| \text{ is}\begin{cases}
\text{odd, if } i_\ell= 1' \text{ then necessarily } \iih_\ell=5 \text{ and if } i_\ell= 3' \text{ then } \iih_\ell=1;\\
\text{even, if } i_\ell= 1' \text{ then necessarily } \iih_\ell=1 \text{ and if } i_\ell= 3' \text{ then } \iih_\ell=5.
\end{cases}
\end{equation*}
For indices $j\in J^{2'}(\iiv)\setminus J_{|\iih|}^3(\iih)$ we can freely choose $\iih_j=2$ or $4$. These are precisely the $\iih$ for which $\proj_x f_{\iih}([0,1]^2)=h_{\iiv}([0,1])$. Using that $p_1=p_5$, the measure equals
\begin{equation*}
\nu_{\mathbf q}( h_{\iiv}([0,1])) = p_1^{\#^1(\iih)+\#^5(\iih)} \binom{\#^{2'}(\iiv)}{\#^3(\iih)} p_3^{\#^3(\iih)} \binom{\#^{2'}(\iiv)-\#^3(\iih)}{\#^2(\iih)} p_2^{\#^2(\iih)}\cdot p_4^{\#^4(\iih)},
\end{equation*}
which after two applications of the binomial theorem yields $\eqref{eq:nuqinZipper}$.

Finally, we conclude that $\dim_{\rm H}\Lambda_a<\dim_{\rm B}\Lambda_a$ since $\Lambda_a$ does not have uniform vertical fibres.
\end{proof}


\section{Three-dimensional applications}\label{sec:ThreeD}
We can compute the Hausdorff dimension of some self-affine carpets in $\mathbb{R}^3$. We do not aim for full generality, rather just demonstrate how our results can be applied.
Throughout this section we always use the following definitions:
\begin{definition}\label{a02}
	Let  $\mathcal{F}$ be a TGL carpet  on $[0,1]^2$
	of the form \eqref{def:IFS_F}, that is
	\begin{equation*}\label{a0098}
	\mathcal F = \{f_i(\underline{x}):=  A_i \cdot \mathbf{x}
	+t_i\}_{i=1}^N, \text{ where } A_i=\begin{pmatrix}
	b_i & 0 \\ d_i & a_i
	\end{pmatrix} \text{ and } t_i=\begin{pmatrix} t_{i,1} \\ t_{i,2}
	\end{pmatrix}, \ \underline{x}\in[0,1]^2.
	\end{equation*}
	Furthermore, let the vectors $\mathbf{u}=(u_1, \dots ,u_N),\, \mathbf{v}=(v_1, \dots ,v_N),\,\pmb{\lambda}=(\lambda_1, \dots ,p\lambda_N)$ 
	be such that for every $1\leq i\leq N$
	$$
	u_i,v_i\in \mathbb{R} \mbox{ and } \lambda_i\in(-1,1)\setminus\left\{0\right\}.
	$$
	We say that the three dimensional self-affine IFS
	\begin{equation*}\label{a0099}
	\widehat{\mathcal{F}}
	:=
	\left\{F_i(\widehat{\underline{x}}):= \widehat{A}_i
	\cdot \widehat{\underline{x}}
	+\widehat{\mathbf{t}}_i
	\right\}_{i=1}^{N},\;
	\mbox{where }
	\widehat{A}_i= \left(
	\begin{array}{ccc}
	b_i & 0 & 0 \\
	d_i & a_i & 0 \\
	u_i & v_i & \lambda_i \\
	\end{array}
	\right),\
	\widehat{ \mathbf{t}}_i:=\left(
	\begin{array}{c}
	t_{i,1} \\
	t_{i,2} \\
	t_{i,3}\\
	\end{array}
	\right)
	\end{equation*}
	on $[0,1]^3$ is an \texttt{uplift of $\mathcal{F}$} corresponding to $(\mathbf{u},\mathbf{v},\pmb{\lambda})$
	if the following conditions hold:
	\begin{description}
		\item[(C1)] For all $1 \leq i \leq N$ we have
		\begin{equation}\label{a01}
		0<|\lambda_i|<a_i<b_i<1.
		\end{equation}
		\item[(C2)] $\widehat{\mathcal{F}}$ satisfies the ROSC (see Definition \ref{def:GLCarpets}).
	\end{description}

	Let $\Lambda$ and $\widehat{\Lambda}$ be the attractor of $\mathcal{F}$ and $\widehat{\mathcal{F}}$ respectively.
	We write $\Pi$ and $\widehat{\Pi}$ for the natural projection from $\Sigma:=\left\{1, \dots ,N\right\}^{\mathbb{N}}$ to $\Lambda$ and  $\widehat{\Lambda}$ respectively. For a probability vector $\mathbf{p}:=(p_1, \dots ,p_N)$ we set
	$\nu_{\mathbf{p}}:=\Pi_*(\mathbf{p}^{\mathbb{N}})$ and
	$\widehat{\nu}_{\mathbf{p}}:=\widehat{\Pi}_*(\mathbf{p}^{\mathbb{N}})$
\end{definition}

We obtain as a corollary of \cite[Theorem 2.3, Proposition 5.8 and Proposition 5.9]{BARANYAnti201788} that 
\begin{corollary}[B\'ar\'any, K\"aenm\"aki]
	Assume that for an uplift $\widehat{\mathcal{F}}$ of $\mathcal{F}$ we have $\mathbf{u}=\mathbf{v}=\mathbf{0}$ and all components of $\pmb{\lambda}$ are equal to the same $\lambda$. Moreover, assume that for a probability vector $\mathbf{p}=(p_1, \dots ,p_N)$ we have $h_{\mathbf p}<\chi_{\mathbf p}^1+\chi_{\mathbf p}^2$ (i.e. the entropy is less than the sum of the Lyapunov exponents). Then $\dim_{\rm H} \nu_{\mathbf{p}} =  \dim_{\rm H} \widehat{\nu}_{\mathbf{p}}$.
	
\end{corollary}
That is, the computation of the Hausdorff dimension of a Bernoulli measure for the three-dimensional non-overlapping system $\widehat{\mathcal{F}}$ is traced back to the corresponding two-dimensional possibly overlapping system $\mathcal{F}$. In this way, if $\mathcal{F}$ satisfies the conditions of Theorem~\ref{thm:maindimresult} then we can determine $\dim_{\rm H} (\widehat{\nu}_{\mathbf p})$ for the three-dimensional system.

In general, we cannot approximate the Hausdorff dimension
of a self-affine  set in $\mathbb{R}^3$ by the Hausdorff dimension of self-affine (or even ergodic) measures (see \cite[Theorem 2.8]{das2017hausdorff}).
However, this is possible in some special cases.

\begin{theorem}\label{a0087}
	Given a diagonally homogeneous  TGL of the form
	\begin{equation*}\label{a0098}
	\mathcal F = \{f_i(\mathbf{x}):=  A \cdot \underline{x}
	+ \mathbf{t}_i\}_{i=1}^N, \text{ where } A=\begin{pmatrix}
	b& 0 \\ d_i & a
	\end{pmatrix} \text{ and } t_i=\begin{pmatrix} t_{i,1} \\ t_{i,2}
	\end{pmatrix}, \ \underline{x}\in[0,1]^2,
	\end{equation*}
	we assume that
	\begin{description}
		\item[(i)] $\mathcal{F}$ has uniform vertical fibres (i.e. each column has the same number of maps).
		\item[(ii)] The projection of $\Lambda$ to the $x$-axis is the whole interval $[0,1]$ (this means that $1/b$ is equal to the number of columns $M$). We assume this to guarantee that the box and affinity dimensions of $\Lambda$ coincide (see Corollary \ref{cor:dimB=dimA}).
		\item[(iii)] Moreover, we assume  that the parameter $a$ is sufficiently small so that
		both conditions \eqref{a0095},
		\eqref{a79} and the transversality condition hold:
		\begin{equation}\label{a0094}
		a<\min\left\{
		b^{\frac{\log N}{\log M}}
		,
		\frac{bd_*}{2+d_*}\right\},
		\end{equation}
		where $d_*$ was defined in Lemma \ref{a87} as
		$$
		d_*:=\min\limits_{1 \leq \jh \leq M \atop \mathcal{P}_{\jh}\ne \emptyset }
		\min\limits_{(k,\ell) \in\mathcal{P}_{\jh} }
		|d_k-d_\ell |,
		$$
		where $(k,\ell) \in\mathcal{P}_{\jh}$ if $f_k([0,1]^2)$ and $f_\ell([0,1]^2)$ belong to the same column and have disjoint interior. 
	\end{description}
	We consider the self-affine IFS $\widehat{\mathcal{F}}$ which is an uplift of $\mathcal{F}$ corresponding to $(\mathbf{u},\mathbf{v},\pmb{\lambda})$ according to Definition~\ref{a02}. That is \eqref{a01} holds and $\mathbf{u}$, $\mathbf{v}$ and
	$\pmb{\lambda}$ are chosen such that
	\begin{equation*}\label{a0093}
	\widehat{\mathcal{F}}
	:=
	\big\{F_i(\widehat{\underline{x}}):= \widehat{A}_i
	\cdot \widehat{\underline{x}}
	+\widehat{\mathbf{t}}_i
	\big\}_{i=1}^{N},
	\mbox{ where }
	\widehat{A}_i=
	\begin{pmatrix}
	b & 0 & 0 \\
	d_i & a & 0 \\
	u_i & v_i & \lambda_i \\
	\end{pmatrix}
	,\
	\widehat{ t}_i:=
	\begin{pmatrix}
	t_{i,1} \\
	t_{i,2} \\
	t_{i,3}\\
	\end{pmatrix},\, \widehat{\underline{x}}\in[0,1]^3
	\end{equation*}                                       
	
	satisfies:
	\begin{itemize}
		\item $F_i\left([0,1]^3\right) \subset [0,1]^3$ holds for all $i\in\left\{1, \dots ,N\right\}$ and
		\item the set $F_i\left([0,1]^3\right)\cap F_j\left([0,1]^3\right)$ has empty interior for all $i\neq j\in\left\{1, \dots ,N\right\}$.
	\end{itemize}

	Let
	$\mathbf{p}:=\Big(\underbrace{1/N, \dots ,1/N}_{N}\Big)$.
	Using the notation of Definition \ref{a02}  we have
	\begin{equation}\label{a0096}
	\dim_{\rm H} \widehat{\nu}_{\mathbf{p}}=    \dim_{\rm H}\widehat{ \Lambda}=\dim_{\rm B}\widehat{ \Lambda}=\dim_{\rm{Aff}}\widehat{ \Lambda}
	=
	1+\frac{\log (Nb)}{-\log a}.
	\end{equation}
\end{theorem}
To give the upper bound in the proof of this theorem, first we need to extend the scope of Lemma \ref{lemma:d/b_bounded} to $\mathbb{R}^3$.

\begin{lemma}\label{a0092}
	There exists $K_x,K_y$ and $K_z$ such that for an arbitrary $n$ and
	$(i_1, \dots ,i_n)\in(1, \dots ,N)^n$ we have
	\begin{equation*}\label{a0091}
	\widehat{A}_{i_1 \dots i_n} \leq
	\begin{pmatrix}
	b^n  & 0 &  0\\
	K_x \cdot b^n & a^n & 0 \\
	K_y \cdot b^n    & K_z \cdot  b^n & \lambda_{i_1 \dots i_n} \\
	\end{pmatrix},
	\end{equation*}
	that is all the elements of the matrix on the right-hand side are greater than or equal to the corresponding element on the left-hand side.
\end{lemma}
\begin{proof}
	For every $n$ and  $(i_1, \dots ,i_n)\in(1, \dots ,N)^n$
	we introduce $x_{i_1 \dots i_n} , y_{i_1 \dots i_n} $ and $z_{i_1 \dots i_n} $ such that
	$$
	\widehat{A}_{i_1 \dots i_n} =
	\begin{pmatrix}
	b^n  & 0 &  0\\
	x_{i_1 \dots i_n}  \cdot b^n & a^n & 0 \\
	y_{i_1 \dots i_n}  \cdot b^n    & z_{i_1 \dots i_n}  \cdot  b^n & \lambda_{i_1 \dots i_n} \\
	\end{pmatrix}.
	$$
	Since the existence of $K_x$ was proved in Lemma~\ref{lemma:d/b_bounded},
	it suffices to prove that $y_{i_1 \dots i_n}$ and
	$z_{i_1 \dots i_n}$ are uniformly bounded in $(i_1, \dots ,i_n)\in\Sigma^*$. To do so, observe that
	\begin{eqnarray}
	\label{a0090}  z_{i_1 \dots i_{n+1}} &=&z_{i_1 \dots i_{n}}\frac{a}{b}+\frac{\lambda_{i_1 \dots i_{n}}}{b^n}\frac{v_{i_{n+1}}}{b},   \\
	\label{a0089}y_{i_1 \dots i_{n+1}}    &=&
	y_{i_1 \dots i_{n}}+ z_{i_1 \dots i_{n}}\frac{d_{i_{n+1}}}{b}
	+\frac{\lambda_{i_1 \dots i_{n}}}{b^n}\frac{u_{i_{n+1}}}{b}.
	\end{eqnarray}
	By \eqref{a01} we obtain from \eqref{a0090}
	that there is an $r\in (0,1)$ and $c>0$ such that
	\begin{equation}\label{a0088}
	z_{i_1 \dots i_n}<c \cdot r^n \mbox{ for all }n \mbox{ and } (i_1, \dots ,i_n)\in\left\{1, \dots ,N\right\}^n.
	\end{equation}
	Namely, we can write down the formula for $z_{i_1 \dots i_n}$ inductively and thus we get that
	$z_{i_1 \dots i_n} \leq \left(a/b\right)^n \cdot \max_i\left\{v_i+n\frac{\max_i\{v_i\}}{b}\right\}$. From here we get that \eqref{a0088} holds.
	This settles the existence of $K_z$.
	Substituting \eqref{a0088}  into \eqref{a0089} and using \eqref{a01} again we obtain the existence of $K_y$. Namely,  the second and third summands in \eqref{a0089} are exponentially small. More precisely,
	$$
	K_y =\max\left\{u_i\right\}+\sum\limits_{n=1}^{\infty }\left(
	c \cdot r^n \cdot \frac{\max\left\{d_i\right\}}{b}+\left(\frac{\max\left\{|\lambda_{i}|\right\}}{b}\right)^n \cdot \frac{\max\left\{u_i\right\}}{b}\right),
	$$
	where all of  the maximums are taken for $i\in\left\{1, \dots ,N\right\}$.
\end{proof}

\begin{proof}[Proof of Theorem \ref{a0087}]\
	\begin{description}
		\item[Lower bound]  Observe that if condition \eqref{a0094} holds then
		it follows from Lemma~\ref{a87} that the transversality condition holds.
		Moreover, as we noted in Section~\ref{subsec:ResDiagHomo},
		condition \eqref{a0094}  also implies that conditions
		\eqref{cond:main} and \eqref{cond:BoxDimCond} hold  when $\mathbf{p}$ is chosen as above to be the uniform vector. In this way the conditions of Theorems~	\ref{thm:maindimresult}  and
		\ref{thm:BoxwOverlaps} are satisfied. As an application of these theorems, we obtain that
		$$\dim_{\rm H} \nu_{\mathbf{p}}=
		\frac{\log N}{-\log a} + \left(1-\frac{\log b}{\log a}\right)\frac{\log M}{-\log b}=
		1+\frac{\log (Nb)}{-\log a}.
		$$
		This implies that
		$$
		1+\frac{\log (Nb)}{-\log a}<
		\dim_{\rm H} \nu_{\mathbf{p}}
		\leq
		\dim_{\rm H} \widehat{\nu}_{\mathbf{p}}
		\leq
		\dim_{\rm H}\widehat{ \Lambda}.
		$$
		\item[Upper bound] 
		It is enough to prove that
		\begin{equation}\label{a0086}
		\dim_{\rm{Aff}}\widehat{ \Lambda} \leq  1+\frac{\log (Nb)}{-\log a}.
		\end{equation}
		This follows from Lemma \ref{a0092} since
		the cylinder $F_{i_1, \dots i_n}\left(\left[0,1\right]^3\right)$ can be covered by
		$N^n \cdot b^n/a^n$
		axes parallel rectangular box of dimensions
		$a^n\times K_x \cdot a^n\times (K_y+K_z) \cdot a^n$. This immediately implies that \eqref{a0086} holds.
	\end{description}
	
\end{proof}

\begin{example}
	Recall the attractor in the center of Figure~\ref{fig:3Dcarpet}. It is defined by an IFS 
	\begin{align*}
	\widehat{\mathcal{F}} &=  \big\{F_i(\widehat{\underline{x}}):= \widehat{A}_i
	\cdot \widehat{\underline{x}}
	+\widehat{\mathbf{t}}_i
	\big\}_{i=1}^{6},
	\mbox{where for } 0<\lambda<a<1/3 \\
	\widehat{A}_1 &= \widehat{A}_5 = 
	\begin{pmatrix}
	1/3 & 0 & 0 \\
	1-a & a & 0 \\
	1-\lambda & 0 & \lambda \\
	\end{pmatrix}
	,\,
	\widehat{A}_2 = \widehat{A}_6 = 
	\begin{pmatrix}
	1/3 & 0 & 0 \\
	a-1 & a & 0 \\
	0 & 0 & \lambda \\
	\end{pmatrix}
	,\,
	\widehat{A}_3 = \widehat{A}_4 = 
	\begin{pmatrix}
	1/3 & 0 & 0 \\
	0 & a & 0 \\
	\lambda-1 & 0 & \lambda \\
	\end{pmatrix}
	.
	\end{align*}
	The translations are chosen appropriately so that $\widehat{\mathcal{F}}$ satisfies the ROSC and the projection to the $xy$-plane looks like the one on the right-hand side of Figure~\ref{fig:3Dcarpet}. If $\lambda<a<1/6$, then the conditions of Theorem~\ref{a0087} hold and we have from \eqref{a0096} that for $\mathbf{p}=(1/6,\ldots,1/6)$
	\begin{equation*}
	\dim_{\rm H} \widehat{\nu}_{\mathbf{p}}=    \dim_{\rm H}\widehat{ \Lambda}=\dim_{\rm B}\widehat{ \Lambda}=\dim_{\rm{Aff}}\widehat{ \Lambda}
	=
	1-\frac{\log 2}{\log a}.
	\end{equation*}
\end{example}

\appendix
\section{No Dimension Drop is equivalent to Weak Almost Unique Coding}\label{app:NDD_WAUC}

In this appendix, we prove that for self-similar IFSs on the line and Bernoulli measures the separation conditions No Dimension Drop (NDD) and Weak Almost Unique Coding (WAUC) are equivalent. We recall notation and definitions.

\subsection*{Notation}
Let $\mathcal H = \{ h_{\ih}(x):=  r_{\ih} x + u_{\ih}\}_{\ih=1}^M$ be a contractive self-similar IFS on the real line with attractor $\Lambda_{\mathcal{H}}$. The symbolic space is $\Sigma_{\mathcal{H}}=\{1,2,\ldots,M\}^\mathbb N$ and the natural projection is $\Pi_{\mathcal{H}}(\iih) := \lim_{n\to\infty} h_{\iih|n}(0)$ for $\iih\in\Sigma_{\mathcal{H}}$. Define a partition of $\Sigma_{\mathcal{H}}$ by
\begin{equation*}
\xi(\iih):= \Pi^{-1}_{\mathcal{H}}\Pi_{\mathcal{H}}(\iih).
\end{equation*}
As we noted earlier in this paper, $\xi$ is a measurable partition of $\Sigma$.
We write $\widehat{\xi}$ for the $\sigma$-algebra generated by the measurable partition $\xi$. Then the elements of  $\widehat{\xi}$
are unions of the elements of  $\widehat{\xi}$.
For a probability vector $\mathbf{q}=(q_1,\ldots,q_M)$ we denote the Bernoulli measure on $\Sigma_{\mathcal{H}}$ by $\mu_{\mathbf q}$.
Then there exists a $\widehat{\Sigma}_\mathcal{H} \subset \Sigma_{\mathcal{H}}$, with $\mu_{\mathbf{q}}(\widehat{\Sigma}_\mathcal{H})=1$ such that for all
$\iih\in \widehat{\Sigma}_\mathcal{H}$ there exists a probability measure
$\mu_{\xi(\iih)}$ defined on $\xi(\iih)$ such that
\begin{itemize}
	\item For all $A \subset \Sigma$ Borel set the mapping $\iih\mapsto \mu_{ \xi(\iih)}(A)$ is $\widehat{\xi}$-measurable and
	\item for all Borel sets $U \subset \Sigma_{\mathcal{H}}$ we have
	\begin{equation}\label{a09}
	\mu_{\mathbf{q}}(U)=\int \mu_{\xi(\iih)}(U)d\mu_{\mathbf{q}}(\iih).
	\end{equation}
\end{itemize}
The push forward measure $\nu_{\mathbf q}=(\Pi_{\mathcal{H}})_{\ast}\mu_{\mathbf q}$ is the self-similar measure with support $\Lambda_{\mathcal{H}}$. The entropy and Lyapunov exponent of the system are
\begin{equation*}
h_{\mu_{\mathbf q}} = -\log \langle\mathbf{q}\rangle_\mathbf{q} \;\text{ and }\; \chi_{\nu_{\mathbf q}} = -\sum_{\ih=1}^M q_{\ih} \log r_{\ih} = -\log \langle \mathbf{r} \rangle_{\mathbf{q}},
\end{equation*}
respectively, where $\langle\mathbf{c}\rangle_{\mathbf{q}}=\prod_{\ih=1}^M c_i^{q_i}$.
Now we recall two separation conditions from Definition \ref{def:34}.
\subsection*{Definitions}
We say that $\mathcal{H}$ has \texttt{No Dimension Drop (NDD)} if for all probability vectors $\mathbf{q}$ with strictly positive entries we have
\begin{equation*}\label{a31}
\dim_{\rm H} \nu_{\mathbf q}=
\frac{h_{\mu_{\mathbf q}}}{\chi_{\nu_{\mathbf q}}}.
\end{equation*}
We say that $\mathcal{H}$ has \texttt{Weak Almost Unique Coding (WAUC)} if for all probability vectors $\mathbf{q}$ with strictly positive entries there exists a set $\mathcal{B}_\mathcal{H}\subset\Sigma_{\mathcal{H}}$ (may depend on $\mathbf q$) for which
\begin{equation*}\label{b98}
\mu_{\mathbf q}(\mathcal{B}_\mathcal{H})=0 \text{ and for every } \iih\in \Sigma_{\mathcal{H}}\setminus\mathcal{B}_\mathcal{H}:\
\#(\xi(\iih)\setminus\mathcal{B}_\mathcal{H})=1.
\end{equation*}

\begin{prop}
	For any self-similar IFS on the line the conditions NDD and WUAC are equivalent.
\end{prop}

Let $\delta_{\iih}$ denote the Dirac-delta measure concentrated on the point $\iih\in\Sigma_{\mathcal{H}}$. We show the assertion in two steps. Namely, we prove that
\begin{equation}\label{eq:StrategyNDD_WAUC}
\mathrm{NDD} \;\;\Longleftrightarrow\;\; \mu_{\xi(\iih)}=\delta_{\iih} \text{ for } \mu_{\mathbf q}\text{-a.e. } \iih\in\Sigma_{\mathcal{H}} \;\;\Longleftrightarrow\;\; \mathrm{WAUC}.
\end{equation}

\begin{proof}[Proof of first equivalence in \eqref{eq:StrategyNDD_WAUC} ]
	The result of B\'ar\'any--K\"aenm\"aki \cite[Theorem 2.3.]{BARANYAnti201788} for dimension $d=1$ states that for every Bernoulli measure $\mu_{\mathbf q}$ its push forward $\nu_{\mathbf q}$ is exact dimensional. Moreover,
	\begin{equation*}
	\dim_{\rm H} \nu_{\mathbf q}=
	\frac{h_{\mu_{\mathbf q}}-H}{\chi_{\nu_{\mathbf q}}}, \,\text{ where }\, H=-\int \log \mu_{\xi(\iih)}([\ih_1])\mathrm{d}\mu_{\mathbf q}(\iih)\geq 0.
	\end{equation*}
	From the definition of NDD we get that
	\begin{equation*}
	\mathrm{NDD} \;\;\Longleftrightarrow\;\; H=0 \;\;\Longleftrightarrow\;\; \mu_{\xi(\iih)}([\ih_1])=1 \text{ for } \mu_{\mathbf q}\text{-a.e. } \iih\in\Sigma_{\mathcal{H}}.
	\end{equation*}
	Thus it suffices to show that 
	\begin{equation}\label{a12}
	\mu_{\xi(\iih)}([\ih_1])=1 \text{ for } \mu_{\mathbf q}\text{-a.e. } \iih\in\Sigma_{\mathcal{H}} 
	\;\Longleftrightarrow\;
	\mu_{\xi(\iih)}=\delta_{\iih} \text{ for } \mu_{\mathbf q}\text{-a.e. } \iih\in\Sigma_{\mathcal{H}}.
	\end{equation}
	The $\Longleftarrow$ direction in \eqref{a12} is obvious.
	In the other direction we show that for $\mu_{\mathbf q}$-a.e. $\iih\in\Sigma_{\mathcal{H}}$
	\begin{equation*}
	\mu_{\xi(\iih)}([\ih_1])=1  \;\;\Longrightarrow\;\; \mu_{\xi(\iih)}([\ih_1\ldots\ih_n])=1 \text{ for every } n   \;\;\Longrightarrow\;\; \mu_{\xi(\iih)}=\delta_{\iih}.
	\end{equation*}
	To see the first implication fix $n$. Let $\mathcal{H}^{(n)}:=\{h_{\ih_1\ldots\ih_n}:\, (\ih_1\ldots\ih_n)\in~\{1,\ldots,M\}^n\}$ and $\Sigma_{\mathcal{H}}^{(n)}$ be the symbolic space of infinite sequences of $n$-tuples $(\ih_1\ldots\ih_n)$. There is a natural one-to-one bijection between the elements of $\Sigma_{\mathcal{H}}$ and $\Sigma_{\mathcal{H}}^{(n)}$. A Bernoulli measure $\mu_{\mathbf q}$ on $\Sigma_{\mathcal{H}}$ naturally defines a Bernoulli measure $\mu_{\mathbf{q}}^{(n)}$ on $\Sigma_{\mathcal{H}}^{(n)}$ by $\mu_{\mathbf{q}}^{(n)}([\ih_1\ldots\ih_n])=\prod_{j=1}^n\mu_{\mathbf q}([\ih_j])$. Applying \cite[Theorem 2.3.]{BARANYAnti201788} to this system yields the first implication. The second implication follows from the Monotone Convergence Theorem
	\begin{equation*}
	1=\lim_{n\to\infty}\mu_{\xi(\iih)}([\ih_1\ldots\ih_n]) = \mu_{\xi(\iih)}(\iih) \;\Longrightarrow\; \mu_{\xi(\iih)}=\delta_{\iih} \text{ for } \mu_{\mathbf q}\text{-a.e. } \iih\in\Sigma_{\mathcal{H}}.
	\end{equation*}
\end{proof}

\begin{proof}[Proof of second equivalence in \eqref{eq:StrategyNDD_WAUC}]\
	
	\begin{description}
		\item[$\Longrightarrow$ direction] We claim that the conditions in the definition of WAUC are satisfied with
		\begin{equation*}\label{a08}
		\mathcal{B}_{\mathcal{H}}:=
		(\Sigma_{\mathcal H}\setminus \widehat{\Sigma}_{\mathcal H})
		\bigcup
		\left\{
		\iih\in\Sigma_{\mathcal{H}}
		:
		\mu_{\xi(\iih)}\ne\delta_{\iih}
		\right\}.
		\end{equation*}
		By assumption $\mu_{\mathbf q}(\mathcal{B}_\mathcal{H})=0$. Moreover, for any $\iih\in\Sigma_{\mathcal{H}}\setminus\mathcal{B}_{\mathcal{H}}$ we have $\iih\in\widehat{\Sigma}_{\mathcal{H}}$, so the probability measure $\mu_{\xi(\iih)}$ exists and
		$  \mu_{\xi(\iih)}=\delta_{\iih}$.
		If  $\jjh\in \xi(\iih)\setminus \mathcal{B}_{\mathcal{H}}$ then $\xi(\jjh)=\xi(\iih)$, thus
		$$
		\delta_{\iih}=\mu_{\xi(\iih)}=\mu_{\xi(\jjh)}=\delta_{\jjh}.
		$$
		That is $\iih=\jjh$. We showed that for every  $\iih\in\Sigma_{\mathcal{H}}\setminus\mathcal{B}_{\mathcal{H}}:\; \xi(\iih)\setminus\mathcal{B}_\mathcal{H}=\left\{\iih\right\}$.
		
		\item[$\Longleftarrow$ direction] Clearly, WAUC is equivalent to the existence of $\mathcal{B}_{\mathcal{H}}$
		with $\mu_{\mathbf{q}}(\mathcal{B}_{\mathcal{H}})=0$ such that
		\begin{equation}\label{a06}
		\mbox{ if } \iih\not\in\mathcal{B}_{\mathcal{H}} \mbox{ then }
		\xi(\iih)\setminus \mathcal{B}_{\mathcal{H}}=\left\{\iih\right\}.
		\end{equation}
		Using \eqref{a09} for  $\mathcal{B}_{\mathcal{H}}$ we obtain that
		the set
		$$
		\widehat{\Sigma}_{\mathcal{H}}:=
		\left\{\iih\in\widehat{\Sigma}_{\mathcal{H}}:
		\mu_{\xi(\iih)}(\mathcal{B}_{\mathcal{H}})=0
		\right\}
		$$
		has full measure:
		\begin{equation}\label{a05}
		\mu_{\mathbf{q}}(\widehat{\Sigma}_{\mathcal{H}})=1,
		\end{equation}
		where we remind the reader that $\widehat{\Sigma}_{\mathcal{H}}$ is the set of those $\iih\in\Sigma_{\mathcal{H}}$ for which the conditional probability measure $\mu_{\xi(\iih)}$ exists. Assume that $\iih\in \widehat{\Sigma}_{\mathcal{H}}$. Then
		\begin{equation*}\label{a04}
		\mu_{\xi(\iih)} \left(\xi(\iih)\setminus\mathcal{B}_{\mathcal{H}}\right)=1
		\mbox{ and by \eqref{a06}: }
		\xi(\iih)\setminus\mathcal{B}_{\mathcal{H}}=\left\{\iih\right\}.
		\end{equation*}
		That is $\mu_{\xi(\iih)}(\left\{\iih\right\})=1$ whenever $\iih\in \widehat{\Sigma}_{\mathcal{H}}$. Combining this with \eqref{a05} we
		get that for a $\mu_{\mathbf{q}}$-full measure set of $\iih$ we have
		$\mu_{\xi(\iih)}=\delta_{\iih}$.
	\end{description}
	
\end{proof}

\subsection*{Acknowledgment}
\begin{wrapfigure}{r}{0.1\textwidth} 
	\centering\vspace{-0.4cm}
	\includegraphics[width=0.1\textwidth]{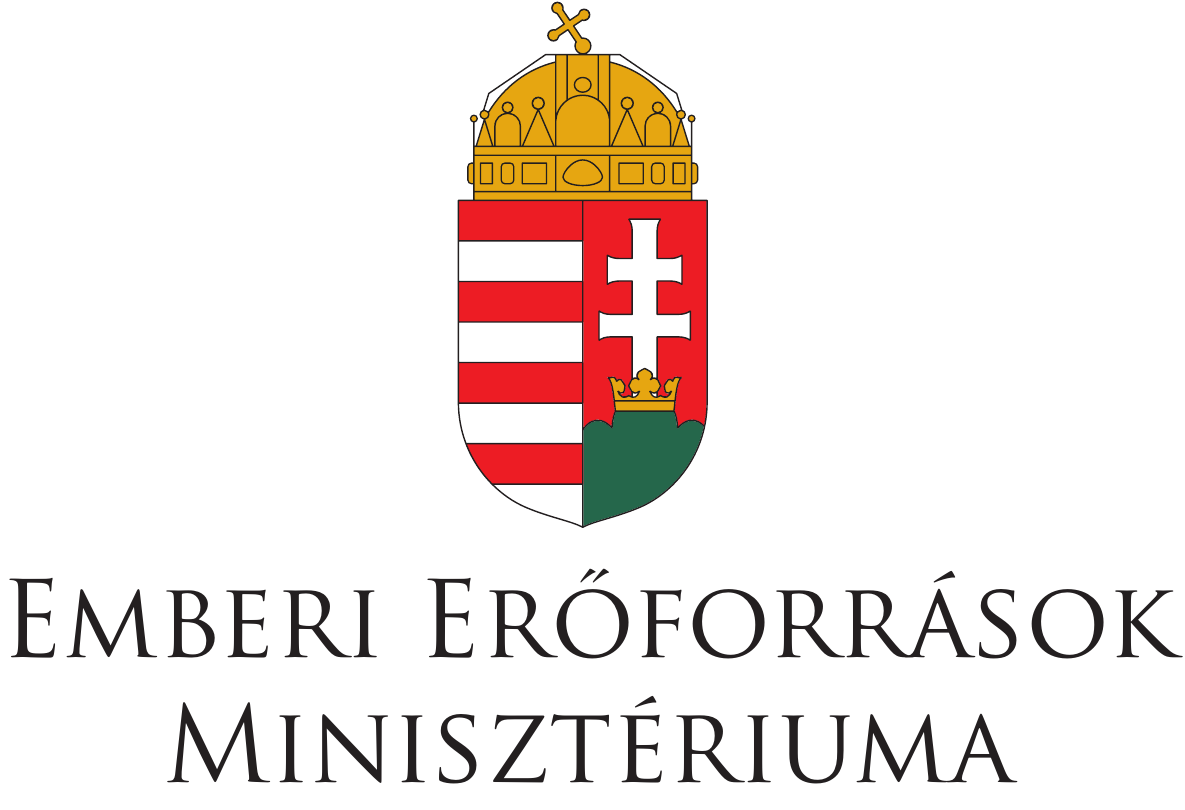}
\end{wrapfigure}
Both authors acknowledge support from the grant OTKA K123782 and MTA-BME Stochastic Research Group. Furthermore, I.K. was partially supported by the \'UNKP--17--3--IV. New National Excellence Program of the Ministry of Human Capacities.
\vspace{0.2cm}

\bibliographystyle{abbrv}
\bibliography{biblio_carpet}

\begin{thebibliography}{10}

\bibitem{AseevTetenov2003}
V.~V. Aseev, A.~V. Tetenov, and A.~S. Kravchenko.
\newblock On selfsimilar {J}ordan curves on the plane.
\newblock {\em Siberian Mathematical Journal}, 44(3):379--386, 2003.

\bibitem{BARANSKIcarpet_2007}
K.~Bara{\'n}ski.
\newblock Hausdorff dimension of the limit sets of some planar geometric
  constructions.
\newblock {\em Advances in Mathematics}, 210(1):215 -- 245, 2007.

\bibitem{BaranskiTriag_2008}
K.~Bara{\'n}ski.
\newblock Hausdorff dimension of self-affine limit sets with an invariant
  direction.
\newblock {\em Discrete Continuous Dynamical Systems - A}, 21(4):1015--1023,
  2008.

\bibitem{barany15_LYformula}
B.~B\'ar\'any.
\newblock On the {L}edrappier--{Y}oung formula for self-affine measures.
\newblock {\em Mathematical Proceedings of the Cambridge Philosophical
  Society}, 159(3):405--432, 2015.

\bibitem{BaranyHochmanRapaport}
B.~{B{\'a}r{\'a}ny}, M.~{Hochman}, and A.~{Rapaport}.
\newblock {Hausdorff dimension of planar self-affine sets and measures}.
\newblock {\em ArXiv e-prints}, 1712.07353, Dec. 2017.

\bibitem{BARANYAnti201788}
B.~B\'ar\'any and A.~K\"{a}enm\"{a}ki.
\newblock Ledrappier--{Y}oung formula and exact dimensionality of self-affine
  measures.
\newblock {\em Advances in Mathematics}, 318:88 -- 129, 2017.

\bibitem{BKK18_zipper}
B.~B{\'a}r{\'a}ny, G.~Kiss, and I.~Kolossv{\'a}ry.
\newblock Pointwise regularity of parameterized affine zipper fractal curves.
\newblock {\em Nonlinearity}, 31(5):1705--1733, 2018.

\bibitem{barany_rams_simon_triang_2017}
B.~B\'ar\'any, M.~Rams, and K.~Simon.
\newblock On the dimension of triangular self-affine sets.
\newblock {\em Ergodic Theory and Dynamical Systems}, pages 1--33.
\newblock to appear.

\bibitem{barany2016dimension}
B.~B{\'a}r{\'a}ny, M.~Rams, and K.~Simon.
\newblock On the dimension of self-affine sets and measures with overlaps.
\newblock {\em Proceedings of the American Mathematical Society},
  144(10):4427--4440, 2016.

\bibitem{Bedford84_phd}
T.~Bedford.
\newblock {\em Crinkly curves, Markov partitions and box dimensions in
  self-similar sets}.
\newblock PhD thesis, University of Warwick, 1984.

\bibitem{das2017hausdorff}
T.~Das and D.~Simmons.
\newblock The {H}ausdorff and dynamical dimensions of self-affine sponges: a
  dimension gap result.
\newblock {\em Inventiones mathematicae}, 210(1):85--134, 2017.

\bibitem{DemboZeitouniLDP}
A.~Dembo and O.~Zeitouni.
\newblock {\em Large {D}eviations {T}echniques and {A}pplications}, volume~38
  of {\em Stochastic Modelling and Applied Probability}.
\newblock Springer-Verlag Berlin Heidelberg, 2010.

\bibitem{FalconerMiao07}
K.~Falconer and J.~Miao.
\newblock Dimensions of self-affine fractals and multifractals generated by
  upper-triangular matrices.
\newblock {\em Fractals}, 15(03):289--299, 2007.

\bibitem{FalconerBookI_1986}
K.~J. Falconer.
\newblock {\em The geometry of fractal sets}, volume~85 of {\em Cambridge
  Tracts in Mathematics}.
\newblock Cambridge University Press, Cambridge, 1986.

\bibitem{falconer_1988}
K.~J. Falconer.
\newblock The {H}ausdorff dimension of self-affine fractals.
\newblock {\em Mathematical Proceedings of the Cambridge Philosophical
  Society}, 103(2):339--350, 1988.

\bibitem{FalconerBook}
K.~J. Falconer.
\newblock {\em Fractal Geometry: Mathematical Foundations and Applications}.
\newblock Wiley, 1990.

\bibitem{FengOral}
D.-J. Feng.
\newblock Dimension of invariant measures for affine iterated function systems.
\newblock oral communication.

\bibitem{FengHu09}
D.-J. Feng and H.~Hu.
\newblock Dimension theory of iterated function systems.
\newblock {\em Communications on Pure and Applied Mathematics},
  62(11):1435--1500, 2009.

\bibitem{FengWang2005}
D.-J. Feng and Y.~Wang.
\newblock A class of self-affine sets and self-affine measures.
\newblock {\em Journal of Fourier Analysis and Applications}, 11(1):107--124,
  2005.

\bibitem{Fraser12Boxlike}
J.~M. Fraser.
\newblock On the packing dimension of box-like self-affine sets in the plane.
\newblock {\em Nonlinearity}, 25(7):2075--2092, 2012.

\bibitem{fraser_shmerkin_2016}
J.~M. Fraser and P.~Shmerkin.
\newblock On the dimensions of a family of overlapping self-affine carpets.
\newblock {\em Ergodic Theory and Dynamical Systems}, 36(8):2463--2481, 2016.

\bibitem{GatzourasLalley92}
D.~Gatzouras and S.~P. Lalley.
\newblock Hausdorff and box dimensions of certain self-affine fractals.
\newblock {\em Indiana University Mathematics Journal}, 41(2):533--568, 1992.

\bibitem{Hochman_Annals14}
M.~Hochman.
\newblock On self-similar sets with overlaps and inverse theorems for entropy.
\newblock {\em Annals of Mathematics}, 180(2):773--822, 2014.

\bibitem{HochmanIndDim}
M.~{Hochman}.
\newblock {On self-similar sets with overlaps and inverse theorems for entropy
  in $\mathbb{R}^d$}.
\newblock {\em ArXiv e-prints}, 1503.09043, Mar. 2015.

\bibitem{Hu98BoxDim}
H.~Hu.
\newblock Box dimensions and topological pressure for some expanding maps.
\newblock {\em Communications in Mathematical Physics}, 191(2):397--407, 1998.

\bibitem{Lalley88}
S.~P. Lalley.
\newblock The packing and covering functions of some self-similar fractals.
\newblock {\em Indiana University Mathematics Journal}, 37(3):699--709, 1988.

\bibitem{LedrappierYoungI_1985}
F.~Ledrappier and L.-S. Young.
\newblock The metric entropy of diffeomorphisms: Part {I}: Characterization of
  measures satisfying {P}esin's entropy formula.
\newblock {\em Annals of Mathematics}, 122(3):509--539, 1985.

\bibitem{LedrappierYoungII_1985}
F.~Ledrappier and L.-S. Young.
\newblock The metric entropy of diffeomorphisms: Part {II}: Relations between
  entropy, exponents and dimension.
\newblock {\em Annals of Mathematics}, 122(3):540--574, 1985.

\bibitem{mcmullen84}
C.~McMullen.
\newblock The {H}ausdorff dimension of general {S}ierpi{\'n}ski carpets.
\newblock {\em Nagoya Mathematical Journal}, 96:1--9, 1984.

\bibitem{Ose68}
V.~I. Oseledets.
\newblock A multiplicative ergodic theorem. {C}haracteristic {L}japunov,
  exponents of dynamical systems.
\newblock {\em Transactions of the Moscow Mathematical Society}, 19:197--231,
  1968.

\bibitem{pardo-simon}
L.~Pardo-Sim\'on.
\newblock Dimensions of an overlapping generalization of {B}ara\'nski carpets.
\newblock {\em Ergodic Theory and Dynamical Systems}, pages 1--31, 2017.

\bibitem{peres_94infinitemeasure}
Y.~Peres.
\newblock The self-affine carpets of {M}c{M}ullen and {B}edford have infinite
  {H}ausdorff measure.
\newblock {\em Mathematical Proceedings of the Cambridge Philosophical
  Society}, 116(3):513--526, 1994.

\bibitem{PrzytyckiUrbanski1989}
F.~Przytycki and M.~Urba{\'n}ski.
\newblock On the {H}ausdorff dimension of some fractal sets.
\newblock {\em Studia Mathematica}, 93(2):155--186, 1989.

\bibitem{simmons2012conditional}
D.~Simmons.
\newblock Conditional measures and conditional expectation; {R}ohlin's
  {D}isintegration {T}heorem.
\newblock {\em Discrete \& Continuous Dynamical Systems-A}, 32(7):2565--2582,
  2012.

\bibitem{solomyak_1998}
B.~Solomyak.
\newblock Measure and dimension for some fractal families.
\newblock {\em Mathematical Proceedings of the Cambridge Philosophical
  Society}, 124(3):531--546, 1998.

\end{thebibliography}

\end{document}